\documentclass[a4paper, 11pt]{article}

\usepackage[utf8]{inputenc}    
\usepackage[T1]{fontenc}
\usepackage{lmodern}
\usepackage{graphicx}
\usepackage[english]{babel}
\usepackage{subfigure}
\usepackage{amsmath, amsfonts, amssymb,amsthm}  
\usepackage{textcomp}
\usepackage{mathrsfs}
\usepackage{multirow}
\usepackage{float}
\usepackage{color}
\usepackage{bm}
\usepackage{mathdots}

\usepackage{hyperref}

\usepackage{xcolor}

\usepackage{verbatim}

\usepackage{float}

\usepackage{cellspace, tabularx}
\addparagraphcolumntypes{X}
\usepackage{cellspace, tabularx}
\addparagraphcolumntypes{X}

\theoremstyle{plain}
\newtheorem{theorem}{Theorem}[section]
\newtheorem{proposition}[theorem]{Proposition}
\newtheorem{lemma}[theorem]{Lemma}

\newtheorem{remark}[theorem]{Remark}
\newtheorem{definition}[theorem]{Definition}

\newcommand{\R}{\mathbb{R}}
\newcommand{\C}{\mathbb{C}}

\newcommand{\N}{\mathbb{N}}
\newcommand{\X}{\mathcal{X}}
\newcommand{\E}{\mathcal{E}}
\newcommand{\CC}{\mathcal{C}}
\newcommand{\B}{\mathcal{B}}
\renewcommand{\bf}{\bar{f}}

\newcommand{\bh}{\bar{h}}
\newcommand{\tf}{\tilde{f}}
\newcommand{\tg}{\tilde{g}}
\renewcommand{\th}{\tilde{h}}
\newcommand{\ba}{\bar{\alpha}}
\newcommand{\ta}{\tilde{\alpha}}
\newcommand{\even}{\mathrm{even}}
\newcommand{\odd}{\mathrm{odd}}
\newcommand{\bu}{\bar{u}}
\newcommand{\bl}{\bar{\lambda}}
\newcommand{\tR}{\tilde{R}}

\newcommand{\ds}{\displaystyle}

\ExplSyntaxOn

\cs_new_protected:Nn \__youthdoo_a: {\allowbreak\mspace{0mu plus 1mu}}
\cs_new_protected:Nn \__youthdoo_b: {\mspace{0mu plus 1mu}}

\NewDocumentCommand{\longformula}{m}
{
	\tl_set:Nn \l_tmpa_tl { #1 }
	\regex_replace_all:nnN { ([0-9]{4})([0-9]{4}) } { \1 \c{__youthdoo_a:} \2 } \l_tmpa_tl
	\regex_replace_all:nnN { ([0-9])([0-9]) } { \1 \c{__youthdoo_b:} \2 } \l_tmpa_tl
	\regex_replace_all:nnN { ([0-9])([0-9]) } { \1 \c{__youthdoo_b:} \2 } \l_tmpa_tl
	\regex_replace_all:nnN { \)(\^[0-9])? } { \0\c{allowbreak} } \l_tmpa_tl
	\tl_use:N \l_tmpa_tl
}
\ExplSyntaxOff

\def\restriction#1#2{\mathchoice
	{\setbox1\hbox{${\displaystyle #1}_{\scriptstyle #2}$}
		\restrictionaux{#1}{#2}}
	{\setbox1\hbox{${\textstyle #1}_{\scriptstyle #2}$}
		\restrictionaux{#1}{#2}}
	{\setbox1\hbox{${\scriptstyle #1}_{\scriptscriptstyle #2}$}
		\restrictionaux{#1}{#2}}
	{\setbox1\hbox{${\scriptscriptstyle #1}_{\scriptscriptstyle #2}$}
		\restrictionaux{#1}{#2}}}
\def\restrictionaux#1#2{{#1\,\smash{\vrule height .8\ht1 depth .85\dp1}}_{\,#2}}

\title{
Validated enclosure of renormalization fixed points \\
via Chebyshev series and the DFT
}

\author{Maxime Breden 
\footnote{CMAP, CNRS, \'Ecole polytechnique, Institut Polytechnique de
		Paris, 91120 Palaiseau, France. \texttt{maxime.breden@polytechnique.edu}} $\quad$  
		Jorge Gonzalez 
		\thanks{JG was partialy supported by NSF grant DMS - 2001758}
		\footnote{Charles E. Schmidt College of Science,
		Florida Atlantic University, 777 Glades Rd., Boca Raton, FL 33431, USA
		\texttt{jorgegonzale2013@fau.edu}} 
		$\quad$ J.D. Mireles James 
		 \thanks{J.D.M.J partially supported by NSF grant DMS - 2307987.}
		\footnote{Department of Mathematics and Statistics,
		Florida Atlantic University, 777 Glades Rd., Boca Raton, FL 33431, USA
		\texttt{jmirelesjames@fau.edu}}}
\date{\today}

\begin{document}

	\maketitle
	
	\begin{abstract}
This work develops a computational 
framework for proving existence, uniqueness, isolation, and 
stability results for degree $d$, real analytic, unimodal functions
fixed by $m$-th order Feigenbaum-Cvitanovi\'{c}
renormalization operators.
Here the order $m$ of the operator 
refers to the number of function compositions involved 
in its definition.  The degree $d$ of the fixed function is the number of 
derivatives vanishing at its (unique) critical point.  
Our approach builds on the earlier work of 
Lanford, Eckman, Wittwer, Koch, Burbanks, Osbaldestin, and
Thurlby \cite{iii1982computer,eckmann1987complete,MR0727816,
burbanks2021rigorous2,burbanks2021rigorous1}, with the 
point of departure being that we discretize the domain of the renormalization operators 
using Chebyshev rather than Taylor series. The advantage of Chebyshev series is that 
they are naturally adapted 
to spaces of real analytic functions, in the sense that they converge on ellipses
containing real intervals rather than on disks in $\mathbb{C}$.  
This facilitates the development of a functional analytic approach 
independent of order and degree. 

The main disadvantage of working with Chebyshev series  
is that the operations of rescaling and composition, 
essential to the definition of the renormalization operators, 
are less straight forward for Chebysehv than for Taylor series.
These difficulties are overcome via a combination of a-priori 
information about decay rates in Banach spaces of rapidly decaying 
coefficient sequences, with a-posteriori 
estimates on Chebyshev interpolation errors for analytic functions.

Our arguments are implemented in the Julia programming language 
and exploit extended precision floating point interval arithmetic. 
The method is used to prove the existence of degree $2$ 
renormalization fixed points of order $m = 2, \ldots, 10$, 
and to obtain validated bounds on the values of their the universal constants. 
We also illustrate that the method works for higher degrees by 
proving the existence of degre $4$ fixed points with $m = 2, 3$. 
To stress the utility of the arbitrary precision implementation, we 
reprove the existence of the classical $m=d=2$ 
Feigenbaum fixed point and compute its universal constants to 500 correct decimal digits. The code for the project is available at \cite{renorCode}.  	
\end{abstract}

\section{Introduction} \label{sec:intro}
The present work is concerned with real analytic, 
unimodal functions of a single variable which 
solve certain functional equations of Feigenbaum-Cvitanovi\'{c}
type.  By unimodal we mean a mapping which 
takes a closed interval $I = [a, b]$ into itself and which, (I) has a single 
interior critical point $c \in (a,b)$ and (II) 
is strictly increasing on $[a, c]$ and strictly decreasing on $[c, b]$.
There is no loss of generality in taking $I = [-1,1]$ and 
$c = 0$.  Indeed, this simplifies the discussion by   recentering it on even functions and we assume this is the case 
moving forward.

Note that a real analytic function on $I$ can be analytically continued 
to an open set in $\mathbb{C}$ containing $I$, and that any such open set   
contains an ellipse of the following form.  
\begin{definition} \label{def:elippse} 
For $\rho > 1$, define the closed Bernstein ellipse 
$\E_\rho \subset \mathbb{C}$ of radius $\rho$ by
\begin{equation*}
			\E_\rho=\left\{ \frac{1}{2}\left(z+z^{-1}\right),
			\ z\in\C,\ 1\leq \vert z\vert \leq \rho \right\}.
\end{equation*}
\end{definition}

\noindent This motivates the following definition.

\begin{definition} \label{def:functionSpace}
For $\rho > 1$, let  $X_\rho$ denote the set of all 
$f \colon \mathcal{E}_\rho \to \mathbb{C}$
with the following properties.
\begin{itemize}
\item \textbf{Analyticity:} $f$ is analytic on the interior of $\mathcal{E}_\rho$,
\item \textbf{Continuity to the boundary:} $f$ is bounded and continuous 
on the closed ellipse $\mathcal{E}_\rho$,
\item \textbf{Reality: } $f(x) \in \mathbb{R}$ when $x \in \mathcal{E}_\rho \cap \mathbb{R}$,
\item \textbf{Evenness: } $f(x) = f(-x)$ for $x \in \mathcal{E}_\rho \cap \mathbb{R}$.
\end{itemize}
\end{definition}
\noindent Denoting by $\left\Vert \cdot\right\Vert_{\CC^0_\rho}$ the supremum norm on $\E_\rho$, i.e.,
\[
\| f \|_{\CC^0_\rho} = \sup_{z \in \mathcal{E}_\rho} | f(z)|,
\]
we have that $\left(X_\rho,\ \left\Vert \cdot\right\Vert_{\CC^0_\rho}\right)$ is a Banach space.

\begin{definition} \label{def:FixedPoint}
Let $\rho > 1$ and 
\[
f^m = \underbrace{f\circ\ldots\circ f}_{m \text{ times}},
\] 
denote the composition of $f$ with itself $m$ times.
We refer to 
\begin{equation} \label{eq:renormDef}
\frac{1}{f^m(0)} f^m(f^m(0) z) = f(z), \quad \quad \quad z \in \mathcal{E}_\rho, f \in X_\rho, 
\end{equation}	
as the (symmetric) $m$-th order Feigenbaum-Cvitanovi\'{c} equation.
\end{definition}

We use the notation $R_m(f)$ as shorthand for  
the left hand side of Equation \eqref{eq:renormDef}, and 
refer to $R_m$ as an $m$-th order renormalization operator (for
even unimodal maps).
Note that, if $f$ solves Equation \eqref{eq:renormDef}, then  
\begin{equation} \label{eq:normalization}
f(0) = \frac{1}{f^m(0)} f^m(f^m(0) 0) =  \frac{1}{f^m(0)} f^m(0) = 1,
\end{equation}
and we see that this form of the fixed point equation
imposes the standard normalization
that the function achieve a maximum value of $1$.

\bigskip

\begin{remark}[Additional terminology] 
\label{rem:FP_properties}
We refer to a solution $f \in X_\rho$ of 
 Equation \eqref{eq:renormDef} as a renormalization fixed point, 
 and distinguish such fixed points as follows.

\begin{itemize}
\item \textbf{Order of the solution:} we refer to  
$m \geq 2$ as the order of the renormalization fixed point $f$,
when $m$ is the number of compositions in the 
definition of Equation \eqref{eq:renormDef}.
\item \textbf{Degree of the solution:} 
Consider $f \in X_\rho$ with $f(0) = 1$.  
Since $f$ is even and analytic, it has Taylor 
series expansion centered at $x= 0$ of the form 
\begin{equation} \label{eq:degree}
f(z) = 1 + a_2 z^2 + a_4 z^4 + a_6 z^6 + \ldots,
\end{equation}
convergent on some disk $|z| < R$.   
The even number $d \geq 2$ denotes the first index such that  
$a_{d} \neq 0$.  We refer to this even number 
as the degree of $f$, and it measures the ``flatness'' of $f$ at $x = 0$.
\item \textbf{Kneading sequence:} this is a partial 
invariant for unimodal maps, introduced by Thurston and Milnor \cite{MR970571}.
The definition of the kneading sequence is somewhat technical, 
but it involves tracking the 
behavior of the orbit of the critical point $x = 0$. More precisely, one constructs a 
sequence of symbols $L = \mbox{``left''}$, $R = \mbox{``right''}$ and $C = 
\mbox{``center''}$ where  the $j$-th term in the sequence
records whether the $j$-th iterate of $x = 0$ is to the left of 
zero, to the right of zero, or zero itself (the ``center point'' of $I$).  
Such a sequence may be periodic, eventually periodic, or 
a-periodic.  The kneading sequence is a partial invariant in the sense
 that unimodal maps with different 
kneading sequences are not topologically conjugate.  
\item \textbf{ Schwarzian derivative:}
We refer to 
\begin{equation} \label{eq:Schwartz}
(Sf)(x) = \frac{f'''(x)}{f'(x)} - \frac{3}{2} \left( \frac{f''(x)}{f'(x)}\right)^2, 
\end{equation}
as the Schwarzian derivative of $f$ on $[-1,1]$, and say that 
$f$ has negative Schwarzian derivative if   
$(Sf)(x) < 0$ on $[-1,1]\setminus\{0\}$.
The significance of this is that a unimodal map $f$ with 
negative Schwarzian derivative 
has at most one attracting periodic orbit. The combinatorics of 
this periodic orbit are closely related to the kneading sequence.  
\end{itemize}
\end{remark}

\bigskip

The family of polynomials 
\begin{equation}\label{eq:quadFamily}
q_{d, \mu}(x)  = 1 - \mu x^{d},   \quad \quad \quad (d \mbox{ even}),
\end{equation}
provide a standard model for even, degree $d$, unimodal maps. 
We refer to $q_\mu(x) = q_{2,\mu} = 1 - \mu x^2$ 
as the standard quadratic family, and to $q_{4, \mu}(x) = 1 - \mu x^4$ 
as the standard quartic family.

Note also that, while we are interested in solutions of Equation 
\eqref{eq:renormDef} with negative Schwarzian derivative,  
we will check this condition a-posteriori rather than imposing 
it a-priori in the function space.  
A related remark is that the 
notion of kneading sequence plays little role in the present work, 
other than as a heuristic guide for locating 
approximate renormalization fixed points of various orders.
However we do compute (at least numerically) the kneading sequence
associated with each fixed point whose existence is proven here, 
if only to illustrate the (well known) fact that distinct fixed points of the same
order have distinct kneading sequences.
(Indeed, in the degree $2$ case there is exactly one renormalization fixed point 
of order $m$ for each  ''minimal'' kneading sequence of length $m-1$, though 
we do not try to exhaust these in the present work).

\bigskip

Given a renormalization fixed point $f \in X_\rho$,
we are also interested in the unstable spectrum of the linearized equation at $f$.
It is well known that 
the normalizations imposed above (evenness of $f$ and $f(0)=1$) impose that the 
unstable spectrum consists of a single unstable eigenvalue.  
More precisely,  let $D R_m(f)$ denote Fr\'{e}chet derivative of $R_m$ 
at $f$.  We seek $\xi \in X_\rho$ with $\|\xi\|_0 \neq 0$
and a complex 
number $\lambda \in \mathbb{C}$ so that 
\[
DR_m(f) \xi = \lambda \xi, 
\]
and $|\lambda| > 1$.  As another bit of introductory notation, we write 
\[
\alpha = f^m(0).
\]
The constants $\lambda$ (unstable eigenvalue) 
and $\alpha$ (or its negative reciprocal $-1/\alpha$) are
referred to as the universal scalings or \textit{Feigenbaum constants}
associated with a renormalization fixed point $f$.
The dynamical meaning of the universal constants 
is reviewed briefly in Section \ref{sec:background}.

\bigskip

\begin{figure}[t]
\begin{center}
\includegraphics[height=12.25cm]{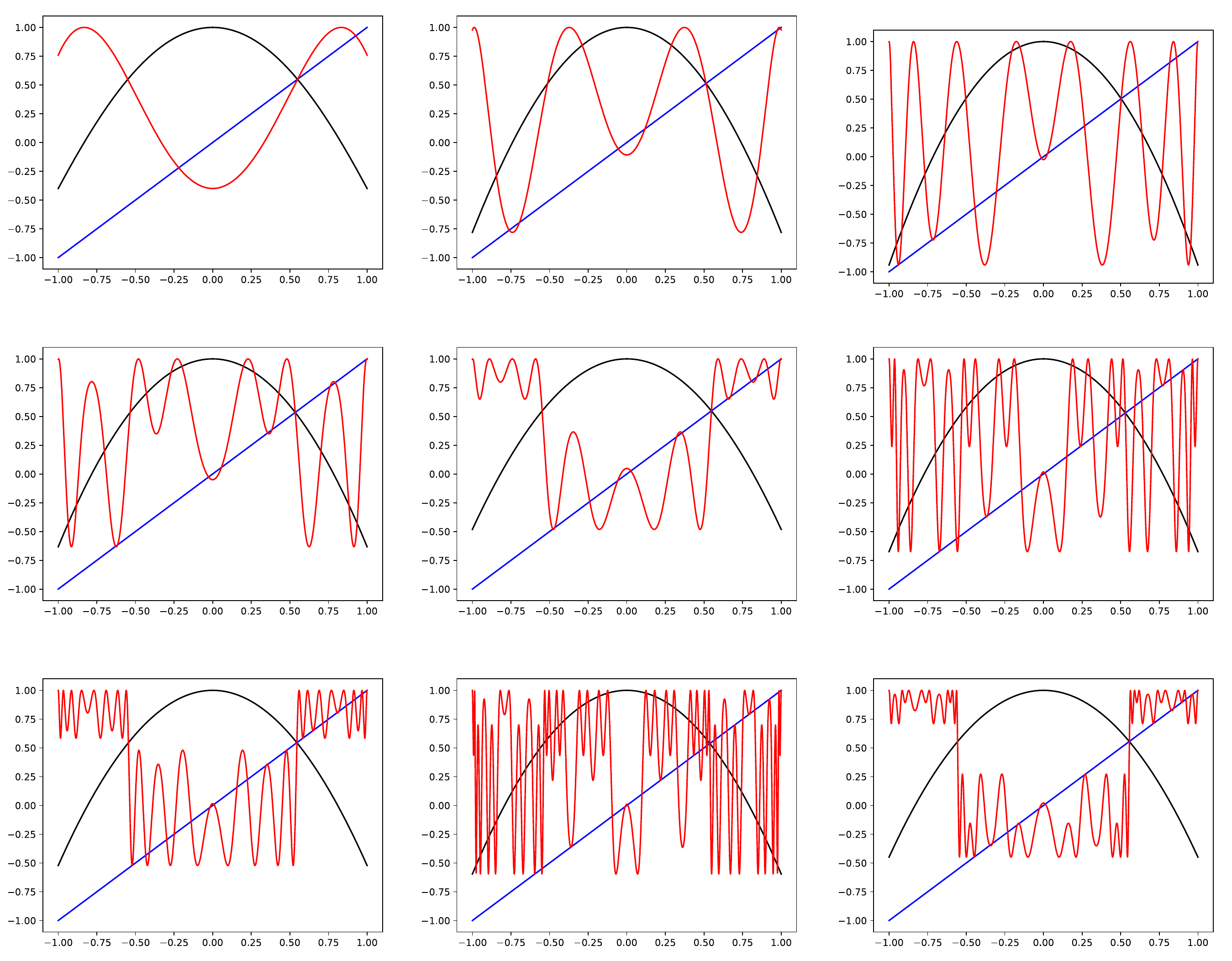}
\end{center}
\caption{\textbf{Some degree two solutions of equation \eqref{eq:renormDef} with order $m = 2, \ldots 10$:} 
The sub-figures are arranged in increasing order,
with an $m= 2$ fixed point depicted in the top left, and an $m=10$ fixed point depicted
in the bottom right. In each frame the black 
line is the graph of the approximate fixed point $\bar{f}_m$, described in Theorem~\ref{thm:main}. 
The red curve illustrates the composition of $\bar{f}_m$ with itself 
$m$ times (but with no rescaling, so the result is not fixed).
Each of these fixed points has degree two (quadratic maximum). 
The blue line denotes the fixed point line $y=x$, and we note that fixed points of the 
$m$-th composition (red curve) denote period $m$ orbits of $\bar{f}_m$.}
\label{fig:theFixedPoints}
\end{figure}

The following theorem illustrates the utility of the constructive 
methods developed in the main body of the present work.
These results, and others, are discussed in greater detail in  
Section \ref{sec:results}.
 
\begin{theorem}[Existence of renormalization fixed points through order ten]\label{thm:main}
Let $\bar{f}_m$, $m = 2, \ldots, 10$ denote the polynomials 
whose graphs are illustrated in Figure \ref{fig:theFixedPoints} and 
whose coefficients in the Chebyshev basis are given in the folder \texttt{renor\_code\_submit}.

There exist real numbers $0 < \epsilon_2, \ldots, \epsilon_{10} \ll 1$, and 
degree $d=2$ functions $f_2, \ldots, f_{10} \in X_{\rho}$, with $\rho=2$,
having that 
 \[
 \sup_{z \in \mathcal{E}_{\rho} } \left| 
 f_m(z) - \bar{f}_m(z)
 \right| \leq \epsilon_m,
 \]
 such that $f_m$ is an $m$-th order solution of Equation \eqref{eq:renormDef}.
 Each renormalization fixed point $f_m$ 
 is locally unique.
Interval enclosures of the associated universal constants 
$\lambda_m$ and $\alpha_m$,
the degree of each polynomial approximation $\bf_m$, 
and upper bounds for the $\epsilon_m$, $2 \leq m \leq 10$,
are provided in Table \ref{table:1}.
\end{theorem}

Note that $C^0$ bounds on the ellipse $\mathcal{E}_{\rho}$ automatically yield 
error bounds on the real interval $[-1,1]$, and that
the values of  the $\epsilon_m$ obtained in the computer assisted proofs 
are actually slightly sharper than those reported in Table \ref{table:1}. 
See Section \ref{sec:results} for more details.  
Using the same $\rho = 2$ for all the renormalization fixed points of Theorem~\ref{thm:main} is convenient, but we emphasize that this  only provides a lower bound on the domain of analyticity for these renormalization fixed points, which is not necessarily sharp, and that smaller values of $\rho$ are required for other renormalization fixed points.

\begin{table}[t]
	\begin{center}
		\begin{tabular}{|c| c | c| c| c | c |} 
			\hline
			$m$ & $K_m$ & $\epsilon_m$  & $\alpha_m$ & $\lambda_m$  \\ 
			\hline\hline
			2 & 21 & $8.233\times 10^{-18}$ &
			-0.39953528052313448${\footnotesize\frac{0}{9}}$ & 
			4.6692016091029906$\frac{0}{9}$  \\ 
			\hline
			3 & 15 & $2.191 \times 10^{-19}$ &-0.107789504292550755$\frac{0}{9}$  
			& 55.247026588671997372$\frac{0}{9}$   \\
			\hline
			4 & 15 &  $1.511 \times 10^{-19}$ & -0.025760531854625116 $\frac{0}{9}$  
			& 981.594976534071427646 $\frac{0}{9}$  \\
			\hline
			5 & 15 &  $1.868 \times 10^{-19}$ &-0.049681005072783868$\frac{0}{9}$  
			& 255.545253865903316209$\frac{0}{9}$  \\
			\hline
			6 & 15 &  $1.951\times 10^{-19}$ & 0.047781479795169254 $\frac{0}{9}$  
			& 218.41179514049496309$\frac{0}{9}$   \\
			\hline
			7 & 15 &  $1.314 \times 10^{-19}$ & 0.017056946896285930$\frac{0}{9}$  
			& 2253.792576403832804223$\frac{0}{9}$   \\
			\hline
			8 & 15 &  $1.004 \times 10^{-19}$ & 0.015062576067338632$\frac{0}{9}$  
			&  2304.557844448592270375$\frac{0}{9}$  \\
			\hline
			9 & 15 &  $1.159 \times 10^{-19}$ & 0.008887154873072033$\frac{0}{9}$  
			&  7918.223563171202619286$\frac{0}{9}$  \\
			\hline
			10 & 15 & $1.523 \times 10^{-19}$  & 0.020848923604938857$\frac{0}{9}$  
			& 1110.537874176532781602 $\frac{0}{9}$  \\
			\hline
		\end{tabular}
	\end{center}
	\caption{\textbf{Data associated with Theorem \ref{thm:main}:}
		the table records the 
		rigorously verified correct digits for the universal constants 
		associated with $m$-th order, degree $2$ renormalization fixed points.
		In this table, we use the notation notation $x = 1.23\frac{4}{6}$ to 
		mean that $x \in [1.234, 1.236]$ and note that the recorded interval 
		enclosures are part of what is established in the computer assisted proof. 
		Each approximate fixed point $\bar{f}_m$ is of degree $2K_m$.
		The $\lambda_m$ and $\alpha_m$ are the universal constants associated with the $f_m$,
		and the $\epsilon_m$ are the bounds on the error between the $\bar{f}_m$ 
		and the actual renormalization fixed points $f_m$. 
	}
	\label{table:1}
\end{table}	

\bigskip

\begin{figure}[t]
\begin{center}
\includegraphics[height=11.00cm]{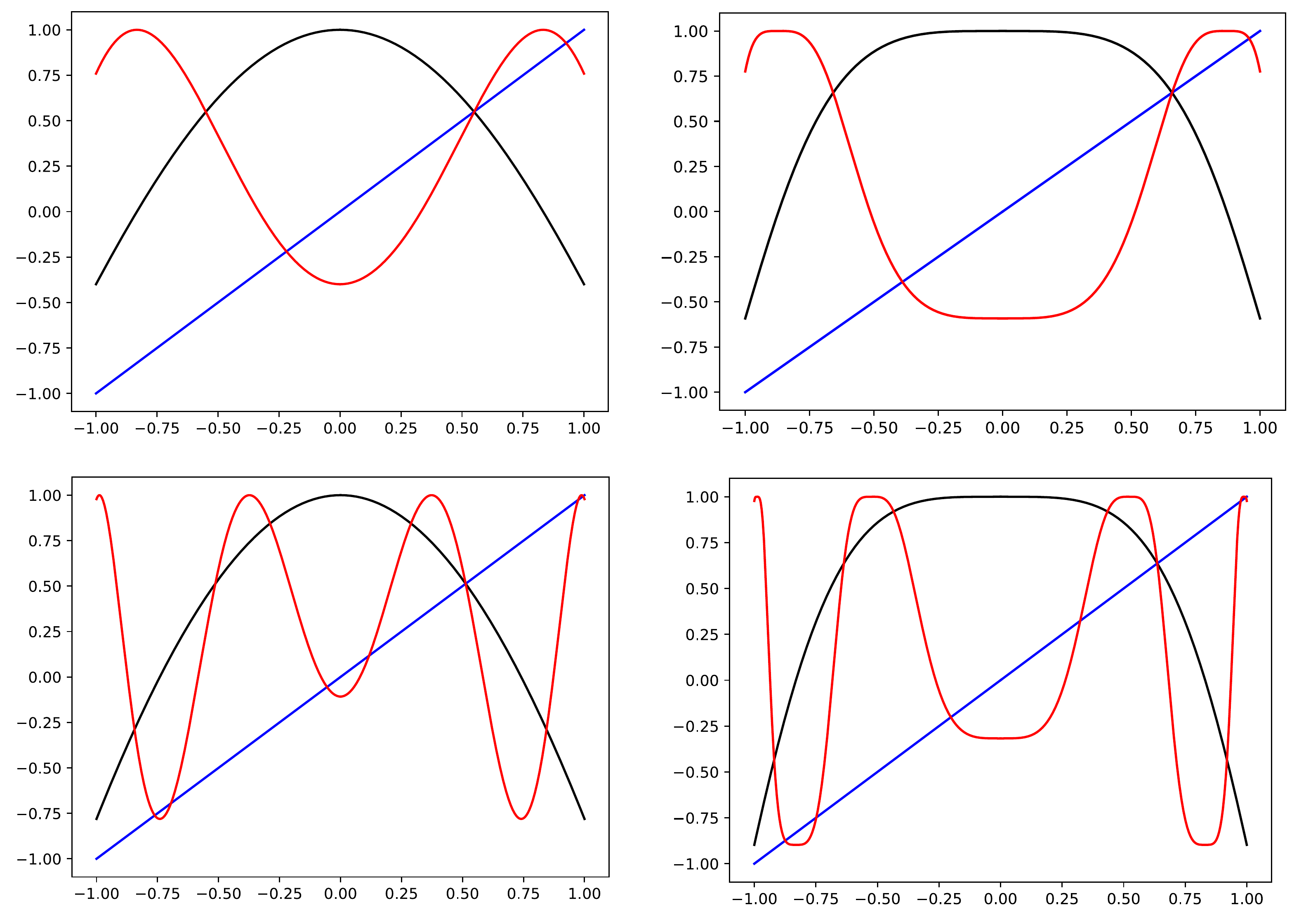}
\end{center}
\caption{\textbf{$m = 2, 3$ Solutions of Equation \eqref{eq:renormDef} with degree $d = 2,4$:} 
The figure contrasts fixed points of Equation \eqref{eq:renormDef} with differing degrees.
The top and bottom left frames illustrate the $m=2$ and $m=3$ 
Feigenbaum fixed points of degree $d = 2$. These two functions were already plotted 
in the top left and center frames of Figure \ref{fig:theFixedPoints} and are reproduced here
for the sake of comparison.  
The top and bottom right frames illustrate the $m=2$ and $m=3$ quartic fixed point described in 
Theorem \ref{thm:d2}. These have quartic maximum and are related to the $q_{d,\mu}$ given in 
Equation \eqref{eq:quadFamily}, with $d = 4$.  Comparing the images 
side-by-side highlights the qualitative
difference between the fixed points near the origin.}
\label{fig:higherDegreeFixedPoints}
\end{figure}

Our method applies also to the study of higher degree renormalization fixed points.  
To see this, we prove a theorem involving fixed points whose derivatives vanish to
order 4 at the critical point $x = 0$.

\begin{theorem}[Existence of renormalization fixed points with $d=4$ for $m=2,3$]\label{thm:d2}
	Let $\bar{g}_m$, $m = 2,  3$ denote the the polynomials 
	whose graphs are illustrated in the right frames of
	Figure \ref{fig:higherDegreeFixedPoints}, and 
	whose coefficients in the Chebyshev basis are given in the folder \texttt{renor\_code\_submit}.
	
	For $m=2,3$, there exist real numbers $0 < \epsilon_m\ll 1$, $\rho_m>1$ and 
	a function $g_m \in X_{\rho_m}$
	having that 
	\[
	\sup_{z \in \mathcal{E}_{\rho_m} } \left| 
	g_m(z) - \bar{g}_m(z)
	\right| \leq \epsilon_m,
	\]
	and that $g_m$ is an order $m$ solution of Equation \eqref{eq:renormDef} 
	of degree $d=4$. That is, we have that 
	\[
	g_m''(0) = 0, \quad \quad m = 2,3.
	\]
	Each renormalization fixed point $g_m$ 
	is locally unique.
	Interval enclosures of the associated universal constants 
	$\alpha_m$ and $\lambda_m$,
	the degree of each polynomial approximation $\bar{g}_m$, the values of $\rho_m$
	and upper bounds for the $\epsilon_m$ 
	are provided in Tables \ref{table:thm_extrafixedpoints_d2_alphalambda}.
\end{theorem}

\begin{table}[h!]
	\begin{center}
		\begin{tabular}{|c| c|c|c| c| c |} 
			\hline
			$m$ & $K_m$ & $\rho_m$ &  $\epsilon_m$  & $\alpha_m$ \\ 
			\hline\hline
			$2$ & $100$ & $1.197$ & $10^{-40}$ & $0.59160991663443815013962435438162895379022\pm 10^{-41}$  \\ \hline
			$3$ & $100$ & $1.180$ & $10^{-45}$ & $-0.3172430469535930185874716238506883387118289297\pm 10^{-46}$  \\ \hline
		\end{tabular}
		
\bigskip
		
		\begin{tabular}{|c| c |} 
			\hline
			$m$ &  $\lambda_m$ \\ 
			\hline\hline
			$2$  & $ 7.28468621707334336430893056799555306947804\pm 10^{-41}$  \\ \hline
			$3$ & $ 85.7916290913560487440851221841293886768048969103\pm 10^{-46}$  \\ \hline
		\end{tabular}
	\end{center}
	\caption{\textbf{Data associated with Theorem~\ref{thm:d2} ($d=4$):}
 Each approximate fixed point $\bar{g}_m$ is of degree $2K_m$.
		The $\lambda_m$ and $\alpha_m$ are the universal constants associated with the $g_m$, which are proven to be analytic on at least the Bernstein ellipse of radius $\rho_m$,
		and the $\epsilon_m$ are the bounds on the error between the $\bar{g}_m$ 
		and the actual renormalization fixed points $g_m$.}
	\label{table:thm_extrafixedpoints_d2_alphalambda}
\end{table}


A number of additional results are discussed in Section \ref{sec:results},
including some global non-uniqueness results.  In particular, 
for $m = 5 \ldots 10$, we prove the existence
of at least two (and in some cases more) distinct renormalization fixed points of degree $d=2$,
together with their associated universal constants. 
Numerically computed kneading sequences associated with each renormalization 
fixed point are reported in Table \ref{table:kneading_seq}. 
Note that while we do establish the existence of multiple fixed points at 
several orders, we do not (in this paper) prove that we have found all 
possible fixed points for each order $m$.  Indeed, for degree two fixed points
it is known that there exists an unique fixed point for each ``allowable'' kneading sequence, 
and that for $m \geq 4$ there are several such sequences at each order.
We refer to \cite{de2011combination} for much more complete discussion.

Our computations are implemented in 
multiple precision interval arithmetic, and to highlight the value of this 
we have carried out, for the classical case of $m = 2$,
verified computations of the renormalization fixed point, as well as
the eigenvalue and eigenfunction, using 
extended precision interval arithmetic and more than $600$ Chebyshev modes.
This results in a guaranteed interval enclosure of the $m=2$ universal constants $\lambda$ and $\alpha$
-- also known as the (firs and second) Feigenbaum constants --  
precise to 500 digits, see Theorem~\ref{thm:m2}.
Even more accurate enclosures could in fact easily be obtained with
longer runtimes, and  we refer again to Section \ref{sec:results} for more details.

A sketch of the proof  of Theorem \ref{thm:main} is as follows.  For each 
$m = 2, \ldots, 10$, 
we numerically approximate a solution of Equation \eqref{eq:renormDef}
using an iterative Newton scheme.
This requires discretizing the elements of $X_\rho$, 
and for this we use Chebyshev series truncated after the $K_m$-th mode
(see \cite{MR1874071,MR4050406}
and also Section \ref{sec:Cheb} for precise definitions).
It is also necessary to seed the Newton method with a reasonable
initial guess. For the fixed points described in Theorem \ref{thm:main}
we start the Newton
iteration from an appropriately chosen 
function $q_{1,\mu} \in X_\rho$ in the 
quadratic family (see Equation \eqref{eq:quadFamily}).
Methods for choosing good starting values of $\mu$, 
depending on the desired degree $m \geq 2$ of the 
renormalization fixed point are well known.  
See for example \cite{de2011combination}.

Running the numerical Newton method results in a truncated Chebyshev series (i.e., a polynomial), 
which we denote by $\bar f_m$. The graphs of these polynomials are  
illustrated in Figure \ref{fig:theFixedPoints} for $m = 2, \ldots, 10$, and 
their coefficients are stored in the files described in the Theorem. 
It must be mentioned that similar numerical schemes, 
based on Chebyshev series approximation, were used in
the works of Mathar \cite{mathar2010chebyshev} and Molteni
\cite{molteni2016efficientmethodcomputationfeigenbaum}.
However these works only focus on the case of $m=2$ and do not 
consider validated error bounds.

After performing the numerical calculations described above, 
the existence of a true renormalization fixed point
$f_m$ near the approximate solution $\bar f_m$ 
is established via a Newton-Kantorovich 
argument.  This requires obtaining
mathematically rigorous bounds on the defect associated 
with the approximate solution, as well as on some other condition numbers.
Here we work in a Banach space of rapidly decaying infinite sequences of 
Chebyshev series coefficients, endowed with a weighted $\ell^1$ norms.
We note that elements of this sequence space are also interpreted as functions
in $X_\rho$, and that the ability to go back and forth 
between these interpretations facilitates certain error bounds
 obtained by combining 
interval arithmetic with interpolation and asymptotic decay rate 
estimates. Describing these in a way that covers all $m, d \geq 2$
is one of the main tasks of the present work.
Interval enclosures of the universal constants are obtained by applying similar  
arguments to the simultaneous eigenvalue/eigenfunction problem.
The details are discussed in Section \ref{sec:results}.

The proof of Theorem \ref{thm:d2} is similar.  The Newton scheme is now seeded using 
an appropriate function $f_{4,\mu}(x) = 1 + \mu x^4$ from the quartic family.  
After the Newton method converges numerically, we have to project the numerical 
guess $\bar g_m$ from the hypothesis of Theorem \ref{thm:d2} into the subspace of 
functions with second derivative equal to zero.  We then show that the renormalization 
operator leaves this subspace invariant, and argue that the true solution lies in the 
desired subspace as well.  Again, the details are in Section \ref{sec:results} .

\bigskip

\begin{remark}[Normalizations] \label{rem:normalization}
{\em
The Feigenbaum-Cvitanovich equation can be formulated in various
equivalent ways.  
For example, in the $m= 2$ case, the equation
is often stated as 
\begin{equation} \label{eq:oldFeig}
\frac{1}{f(1)} f(f(-f(1) x)) = f(x), \quad \quad \quad \quad -1 \leq x \leq 1,
\end{equation}
subject to the additional constraint 
\[
f(0) = 1.
\]
The minus sign in front of the innermost $-f(1)$ term is removed by restricting to 
a space of even functions.  Then,   
when representing $f$ using a Taylor series, the $f(0) = 1$ constraint is then
enforced by simply imposing that the constant term in the Taylor expansion
is one, and solving for the remaining coefficients (evenness is also imposed on 
the level of Taylor coefficients).

When working with Chebyshev series, 
the $f(0) = 1$ constraint is more global, as it involves 
the Chebyshev coefficients of all orders
(see again Section \ref{sect:setup}).
This is one reason for imposing the normalization $f(0) = 1$
directly in the functional equation, as discussed 
at Equation \eqref{eq:normalization} above.
}
\end{remark}

\begin{remark}[Chebyshev versus Taylor series] \label{rem:cheb}
{\em
Heuristically speaking, the reason for studying solutions of 
Equation \eqref{eq:renormDef} on Bernstein ellipses 
is already suggested, at least implicitly, by the theoretical work of 
Epstein and Lascoux in \cite{MR634164}, and the 
numerical work of Nauenberg in \cite{MR894403}.
These authors provide detailed quantitative information about 
the shape of the domain of analyticity of the Feigenbaum 
function (solutions of the $m= 2$ case of Equation \eqref{eq:renormDef})
and show for example that its domain has a fractal shaped
 boundary enclosing the set 
$\mathbb{R} \cup i \mathbb{R}$. In particular, there exists a disk 
of radius $R > 0$ large enough so that its interior is not contained in the domain.
See for example Figure 1 of \cite{MR894403}.

Then, it is only by good luck  
that the domain of analyticity is large enough to contain a disk at the origin
of radius $r \approx 2.5 > 1$.  Because of this, the 
fixed point has a Taylor series at the origin which is convergent on $[-1,1]$,   
a necessary condition for the success of 
computer assisted methods of proof based on 
Taylor series. Note however that  
there is no guarantee that this luck will hold
as either the order or the degree of the solution is increased.

Indeed, the results of Eckmann and Wittwer in \cite{MR788690}
show that when $m= 2$ and the degree $d$ of the solution
of Equation \eqref{eq:renormDef}
is increased, the resulting domain of analyticity shrinks to $\mathbb{R} \cup i \mathbb{R}$
as $d \to \infty$.  Then, for $d$ large enough, there is 
no single disk in the complex plane containing $[-1,1]$ on which the  
desired fixed point is analytic, and it will therefore not
be representable by a single 
convergent power series at $x=0$.
Similarly, there is no guarantee that as $m$ varies we will 
always have such a disk either. 

These considerations suggests the use of Bernstein ellipses,
and after settling on such, 
the choice of Chebyshev series is natural. 
}
\end{remark}

\subsection{Computer assisted proof in Renormalization theory} \label{sec:renormLit}
The strategy of computer assisted proof sketched above is, in broad outline,
the one used by Lanford in the original computer assisted  
existence proof of the $m= 2$ fixed point \cite{iii1982computer}.
Similar arguments were later used by Eckman, Wittwer, and Koch 
in the complete proof of the complete Feigenbaum 
conjectures \cite{iii1982computer,eckmann1987complete,MR0727816}.
The last reference just cited discusses additional theorems in 
renormalization theory proven in the same style.
We refer also the works of  Burbanks, Osbaldestin, and
Thurlby \cite{burbanks2021rigorous2,burbanks2021rigorous1}
for more recent computer assisted results for $R_m$ when $m = 2$
for degree $d = 2$ and $d = 4$ fixed points.
We note however the the works just cited impose the constraint $d = 4$ 
in their functional analytic setup, so that this approach requires 
modifying the arguments for each $d$ considered.

There are also computer assisted theorems for renormalization operators
other than Feigenbaum-Cvitanovi\'{c} 
(that is, for situations other than period doubling in one parameter 
families of one dimensional unimodal maps) which 
use arguments similar to those in the outline sketched above.  
We refer for example to the works of Koch and Wittwer
\cite{MR0859824,MR1291247}, Koch \cite{MR2112709,MR4574150},
Gaidashev and Koch \cite{MR2818692}, Arioli and Koch 
\cite{MR2594332},
and Gaidashev and Yampolsky \cite{MR4480373}.
We also remark that  computer assisted proofs based
on Newton-Kantorovich and/or contraction mapping arguments
are common in a wide variety of mathematical settings
including dynamical studies of ordinary, delay, and partial differential 
equations.  The interested reader is referred to the review articles 
of van den Berg and Lessard \cite{MR3444942} and of 
G\'{o}mez-Serrano \cite{MR3990999} for more general discussion.

All of this being said, studying Equation \eqref{eq:renormDef} on 
 $\mathcal{X}_\rho$ and discretizing using 
 using Chebyshev series
completely changes the technical character of the arguments,
justifying the need for the present work.  More precisely,
the operations of composition and rescaling 
(see again Equation \eqref{eq:renormDef})
are less natural for Chebyshev than for Taylor series.
For example, when working with Taylor series, 
compositions are worked out via
iterated Cauchy products -- which scale poorly as $m$ is increased.  
For Chebyshev series, an analogous approach leads to iterated
discrete convolutions of even worse complexity.  Similarly, while rescaling 
is a diagonal linear operator in Taylor coefficient space, this is not so 
in Chebyshev space.  

To overcome these issues 
we employ techniques based on 
interpolation.  Indeed, recalling that a Chebyshev series is 
 ``a Fourier series in disguise,'' we 
utilize the discrete Fourier transform (DFT) to evaluate compositions and 
rescalings ``in grid space,'' where such operations are
essentially ``diagonal'' (that is, they have algorithmic complexity which scales
with $N$,  the degree of the Chebyshev series approximation).
Of course this introduces the additional cost of the DFT, but this
 can be made ``essentially linear'', that is $N \log(N)$,
  by employing the fast Fourier transform (FFT). 

The DFT techniques just described could also be adapted for computer
assisted proofs with Taylor series as well.  The real advantage then 
of using Chebyshev series 
is that the Bernstein ellipses, the natural domains of analyticity 
for Chebyshev series, are especially well adapted for studying 
real analytic functions on intervals.  Chebyshev series are also known to 
have excellent uniform approximation properties on closed intervals.

For computer assisted proofs using the interpolation based DFT methods 
just described, bounding discretization and truncation errors requires 
managing the so called ``aliasing error'' which is, more precisely, 
the difference between orthogonal projection and interpolation.
To obtain the necessary bounds, we start from the aliasing formula 
for Chebyshev series~\cite[Theorem 4.2]{Tre13}, and then derive interpolation 
error estimates between appropriate function spaces. These 
developments are described in detail in Section \ref{sec:aPriori}. Similar ideas 
where used by Figueras, Haro, and Luque for computing interval enclosures 
of function compositions for Fourier series
in their work on computer assisted KAM theory \cite{MR3709329}.

We remark that, validated methods for computing rigorous interval 
enclosures of functions represented by 
Chebyshev series were developed by Joldes and Brisbarre in 
\cite{MR2920548}, with extensions by Benoit, Joldes, and Mezzarobba
in \cite{MR3614019}.  Existence proofs using Chebyshev series
for differential equations with polynomial nonlinearities were developed by 
Lessard and Reinhardt in \cite{MR3148084}, and extensions to 
nonpolynomial applications in celestial mechanics
and mechanical engineering are developed in 
\cite{MR3741385,MR3896998} by exploiting ideas from 
differential algebra.  Here, one computes a rigorous 
enclosure of  a non-polynomial function by solving the
polynomial differential equation it satisfies.
While this idea works well for problems
with nonlinearities given by elementary functions, 
it is not clear how to extend this idea 
to the function compositions appearing 
in renormalization theory, and this motivates the interpolation 
based approach developed below.

\begin{figure}[h!]
\begin{center}
\includegraphics[height=8.0cm]{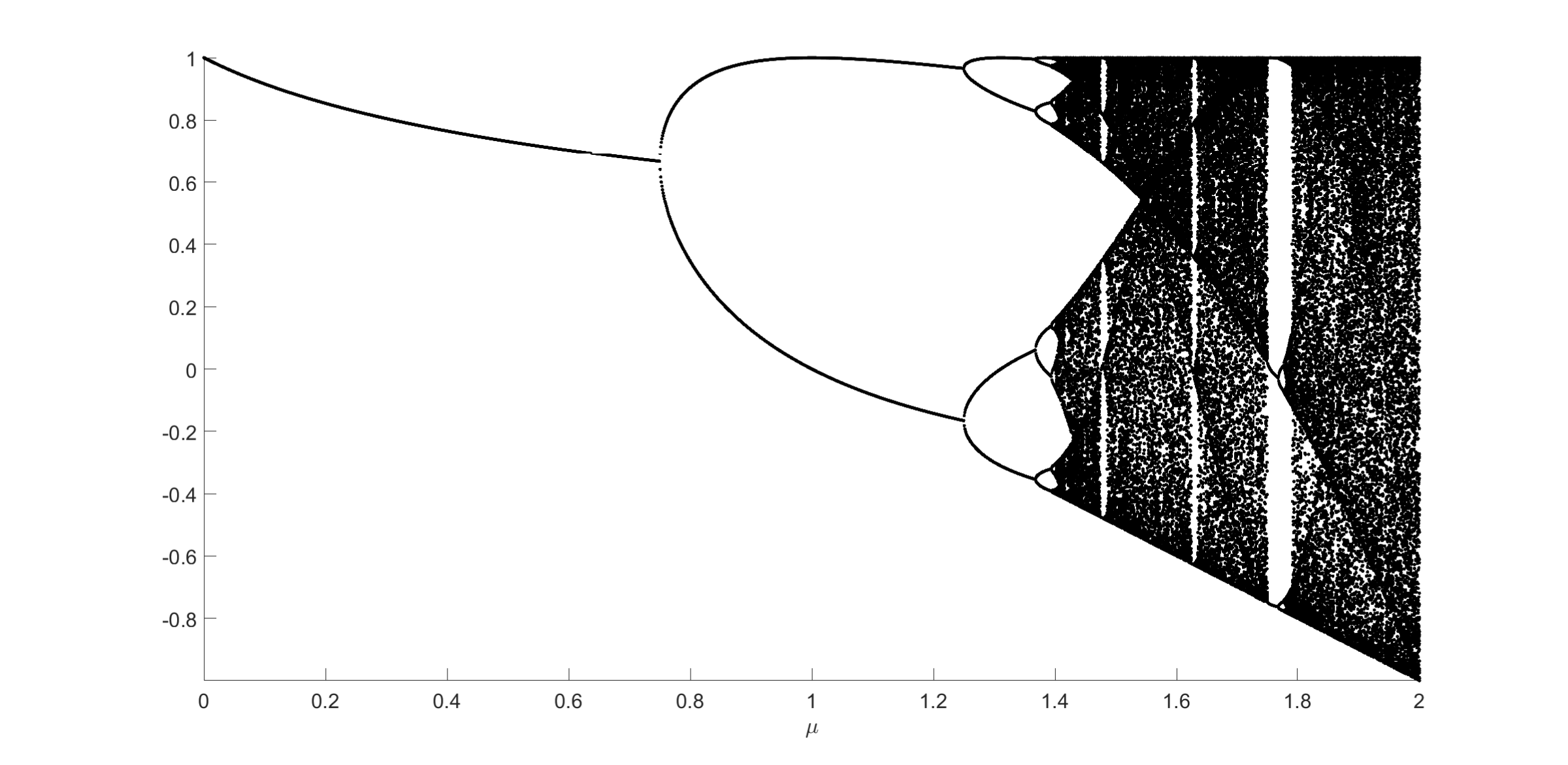} 
\end{center} \caption{\textbf{Classical period doubling cascades:} the image illustrates several 
period doubling cascades for the quadratic family $q_{2,\mu}$ of Equation \eqref{eq:quadFamily}, 
by plotting accumulation points of orbits for different values of $\mu$. The left side of the image
depicts the cascade of stable period $2^n$ orbits bifurcating from the stable fixed point and 
accumulating to the chaotic regime.  After the onset of chaos, one 
sees also the stable period three window (for $\mu$ slightly below $1.8$), and its doubling bifurcations into orbits of period $3\cdot 2^n$.  
Other stable period 5,  and 6 windows can be seen by looking carefully after the first onset of chaos, and each undergoes its own period doubling cascade. The locations of all of these doublings are governed by the 
Feigenbaum constant $\lambda_2$ given in the first row of Table \ref{table:1} (see Section~\ref{sec:background}). }
\label{fig:cascades}
\end{figure}

\begin{figure}[h!]
\begin{center}
\includegraphics[height=8.0cm]{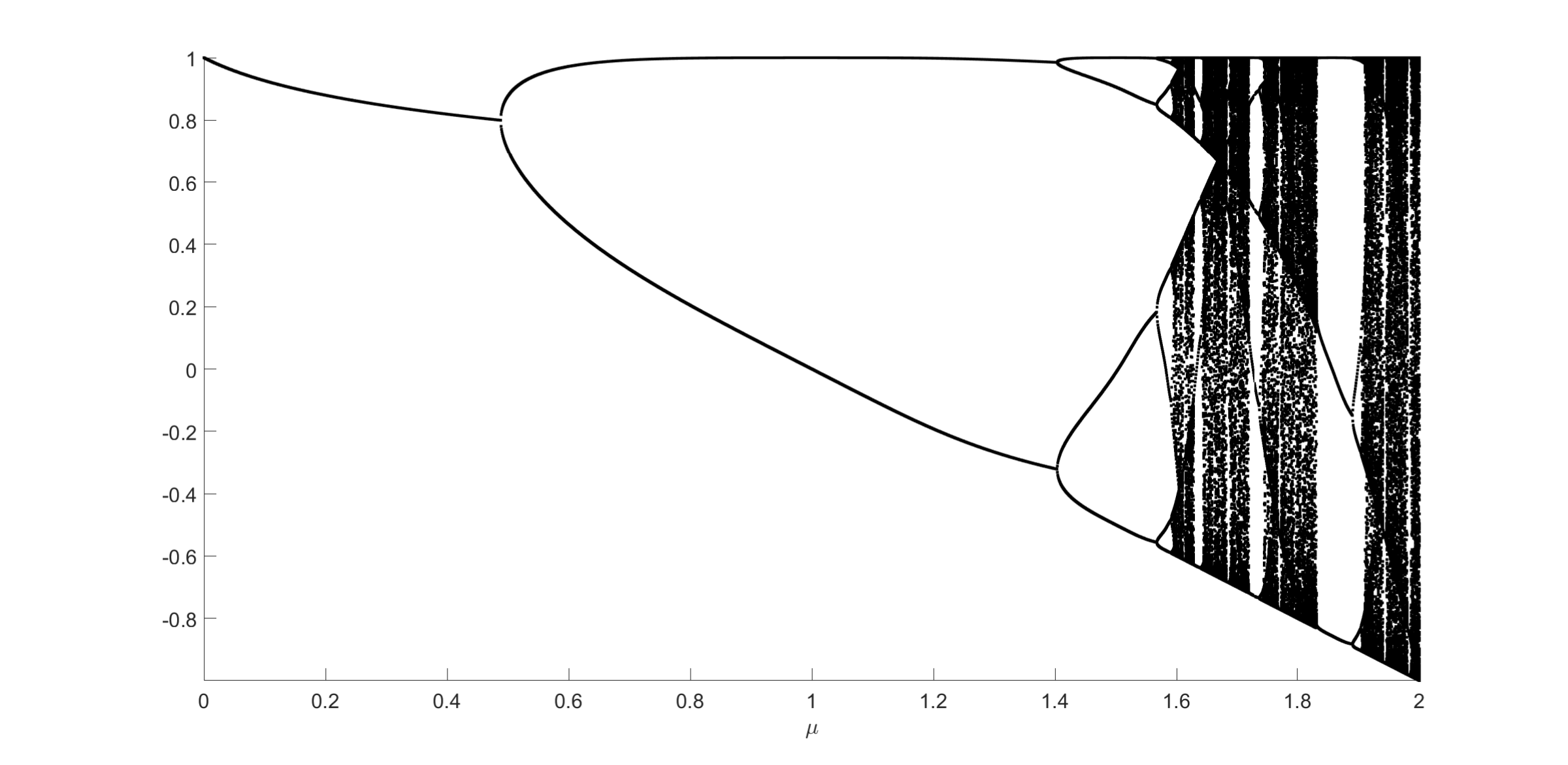} 
\end{center} 
\caption{\textbf{Period doubling for a degree 4 family:} this image illustrates
period doubling cascades for the quartic family $q_{4, \mu}$.  It 
happens that while the bifurcation diagram exhibits many of the same 
qualitative feature as in the degree $2$ case illustrated in Figure \ref{fig:cascades}, 
the quantitative rates governing the bifurcations are 
completely different.  That is, the cascades are governed by different universal constants.  
}
\label{fig:cascades_quar}
\end{figure}

\subsection{Period-tupling cascades, and their renormalization fixed points} \label{sec:background}

Period doubling cascades are central to study of 
nonlinear dynamical systems.
The topic was popularized in the 1976
paper of May in Nature \cite{mayLogistic}, but the story 
goes back much further, as illustrated for example
by the 1962 paper of Myrberg on  
the existence of period doubling in polynomial families  
\cite{MR0161968}, and the 1964 results of Sharkovsky
\cite{MR0159905}.  A lively discussion of
more than a century of ``pre-renormalization'' literature on period doubling
is found in  the 2019 review article of Collet
 \cite{COLLET2019287}.

A typical period doubling cascade is illustrated in 
the top frame of Figure \ref{fig:cascades}, 
for the quadratic family (defined in Equation \eqref{eq:quadFamily}
but recalled here for the sake of clarity) 
\begin{equation*} 
q_{2,\mu}(x) = 1 - \mu x^2.
\end{equation*}
Less typical is the cascade illustrated in Figure \ref{fig:cascades_quar}
for the quartic family 
\[
q_{4,\mu}(x) = 1 - \mu x^4.
\]
In both cases, the parameter $\mu$ represented on the 
horizontal axis, and the vertical axis depicts the attracting set for the 
map at the given parameter value.

So for example, focusing on Figure \ref{fig:cascades}, we see that 
 when $\mu < 0.75$, the map has a single attracting fixed point.  Near 
$\mu = 0.75$ an attracting period two orbit is born in a period doubling
bifurcation. Further doubling bifurcations occur at $\mu \approx 1.25$ 
and beyond, resulting in attracting orbits of period $4$, $8$, $16$ etcetera.  
After 
\[
 \mu_\infty \approx 1.4...
\]
orbits are attracted to a chaotic invariant set.

In the late 1970's, papers by 
Feigenbaum \cite{feigenbaum1978quantitative,feigenbaum1979universal}
and Coullet and Tresser \cite{coullet1978iterations,MR0512110}
presented compelling numerical evidence in support of the 
conjecture that the period doubling bifurcations occur at 
parameters $\mu_n$ having 
\[
\mu_n \to 1.4... = \mu_\infty, \quad \quad \mbox{as } n \to \infty,
\]
and that the ratios between the lengths between successive bifurcations
satisfy
\begin{equation} \label{eq:feigenbaum}
\lim_{n \to \infty } \frac{\mu_{n} - \mu_{n-1}}{\mu_{n+1} - \mu_{n}} 
= \lambda_2 \approx 4.66920...
\end{equation}
That is, the constant $\lambda_2$
eventually facilitates locating the next bifurcation in the 
sequence, once the first several are known.

The truly remarkable discovery was the ``universality'' of this behavior.
Loosely speaking, it was observed that
the numerical bifurcation diagrams associated with
one parameter families of degree two unimodal maps appear to 
satisfy Equation \eqref{eq:feigenbaum} 
\textit{with the same universal constant } $\lambda_2 \approx 4.66920$.
This is referred to as \textit{the Feigenbaum constant}.
The description of period doubling summarized 
above is part of what became known as 
\textit{the Feigenbaum conjectures}.

The key point, from the perspective of the present work, 
is that the the universality properties just described are 
related to the renormalization fixed points described in the 
introduction.  Indeed, it turns out 
that the Feigenbaum conjectures are equivalent to the 
statement that the $m=2$ renormalization operator 
has a degree $d=2$ hyperbolic fixed point satisfying certain 
geometric hypotheses. More precisely, the 
one dimensional unstable manifold attached to the fixed point
should intersect transversally the (codimension-one) set of 
superstable period two functions. Moreover the Feigenbaum constant
is seen to be the unstable eigenvalue of the $m = 2$ 
renormalization operator at the fixed point.   For a more detailed 
discussion of the equivalence between the Feigenbaum conjectures
and renormalization fixed points, we refer to the work of
Collet, Eckmann, Lanford in \cite{MR0588048}.
We also refer the interested reader to the paper   
\cite{10.1119/1.19190} by Coppersmith
for a heuristic derivation of the renormalization operator
via period doubling.

The first existence proof for a locally unique 
fixed point for the $m=2$ renormalization operator was given by Lanford in 1982
\cite{iii1982computer}, and a complete proof of the Feigenbaum conjectures
by Eckman and Wittwer appeared in 1987
\cite{eckmann1987complete}.  The results just mentioned
were established with computer assistance, and we refer to the
memoir of  Eckman, Wittwer, and Koch \cite{MR0727816} for a complete 
description.  Further developments are mentioned in Remark \ref{rem:penPaper} below.

Shortly after the work of Feigenbaum-Coullet-Tresser on period doubling, 
it was discovered that period $m$-tupling cascades are similarly
governed by higher order renormalization operators $R_m$ with $m \geq 3$,
and their universal constants (unstable eigenvalues).  
To the best of our knowledge, 
the first study of this kind was the 1979 paper of Derrida, Gervois, and Pomeau
\cite{MR0524170}, which considered universal properties of period tripling cascade
($m = 3$), and proposed a geometric explanation based on renormalization.

A period trippling bifurcation occurs when a period $k$ orbit bifurcates 
into an orbit of period $3k$, and a tripling cascade is a sequence of 
bifurcations into orbits of period $k 3^n$ for $n \geq 1$.
These bifurcations are studied further in the work of 
Gol'berg, Sinai, and Khanin \cite{MR0693727},
Hu and Satija \cite{HU1983143}, 
Delbourgo and Kenny \cite{MR0788959}, and 
Delbourgo, Hart, and Kenny \cite{PhysRevA.31.514,10.1071/PH860189}.
The paper \cite{PhysRevA.31.514} just cited is especially 
interesting from the perspective 
of the present work, and many of its results -- for example the 
values given in  Table 1 of that paper -- 
are validated in the present work.   

For the sake of intuition, 
Figures \ref{fig:cascades3} 
and \ref{fig:cascades31} illustrate two different 
period tripling ($m=3$) cascades. The first starts from a 
period 2 orbit and leads to orbits of period 6, 18, and beyond. 
The second starts 
from period 3 and leads to orbits of period 9, 27 and beyond.
Note that for $m > 2$ the bifurcations are of fold or 
tangent type, rather than the pitchfork type seen in the $m=2$ case.
The cascades  illustrated in Figures \ref{fig:cascades3} and \ref{fig:cascades31}
are governed by the constant $\lambda_3$
given in Table \ref{table:1}.  We note that similar 
$m$-tupling cascades are associated with each of the fixed points 
described in Theorem \ref{thm:main}, but that plotting the resulting 
cascades is increasing delicate.

\begin{figure}[h!]
\begin{center}
\includegraphics[height=9.0cm]{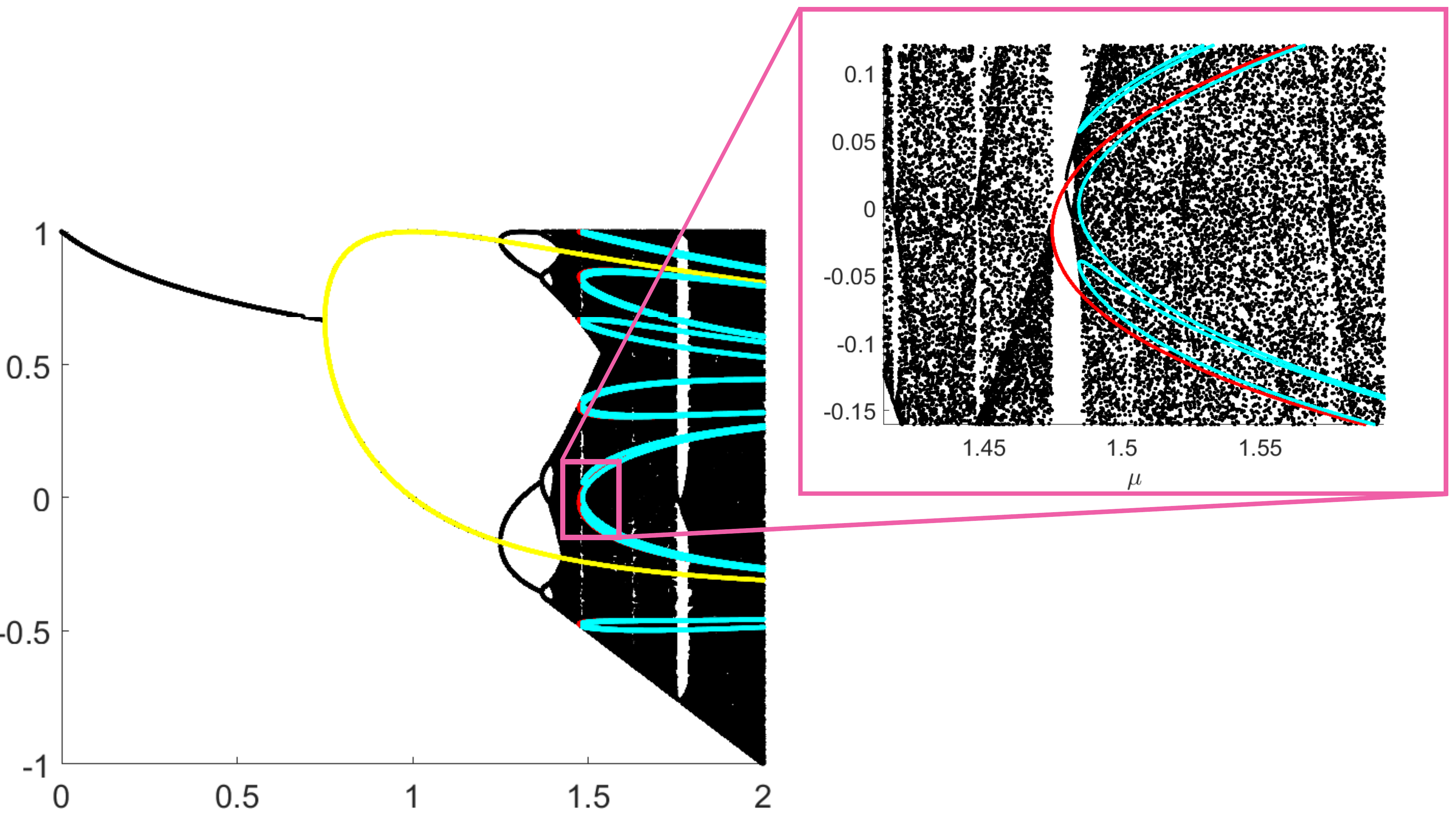}
\end{center}
\caption{\textbf{Trippling cascade from the stable period two orbit:}
the figure illustrates a period trippling cascade, starting from the 
stable period two orbit.  Again, it is important to note that for 
$m > 2$, the bifurcations appearing in the cascade are of 
tangency or fold type and hence appear ``to come out of nowhere''
(the bifurcation curves are not connected as in the $m = 2$ case).
The left frame illustrates the entire bifurcation diagram of the 
quadratic family (Equation \eqref{eq:quadFamily} with $d = 2$)
and the period two is highlighted in yellow.  Just below $\mu = 1.5$
the period trippling bifurcation occurs, and a stable and unstable period 
6 are born.  The next bifurcation occurs almost immediately, resulting 
in a stable and unstable period 18 orbit.  This is illustrated in the 
right frame of the figure where one branch of the period 6
is shown in red, followed by three branches of the period 18 in blue (the red and blue are difficult to distinguish in the frame on the left).
Bifurcations to period 54 and 
beyond are difficult to illustrate graphically, but we note that 
they can be computed numerically, for instance by exploiting that they 
follow from the universal
constant $\lambda_3$ given in the second row of Table \ref{table:1}. In order to better emphasize that the different orbits appearing in a period tripling cascade are not ``connected'', the period 2, 6 and 18 orbits under consideration here are depicted even for values of $\mu$ for which they have become unstable, contrarily to what is typically done for such bifurcation diagram like in Figure~\ref{fig:cascades} (where only stable orbits appear). }
\label{fig:cascades3}
\end{figure}

\begin{figure}[h!]
\begin{center}
\includegraphics[height=9.0cm]{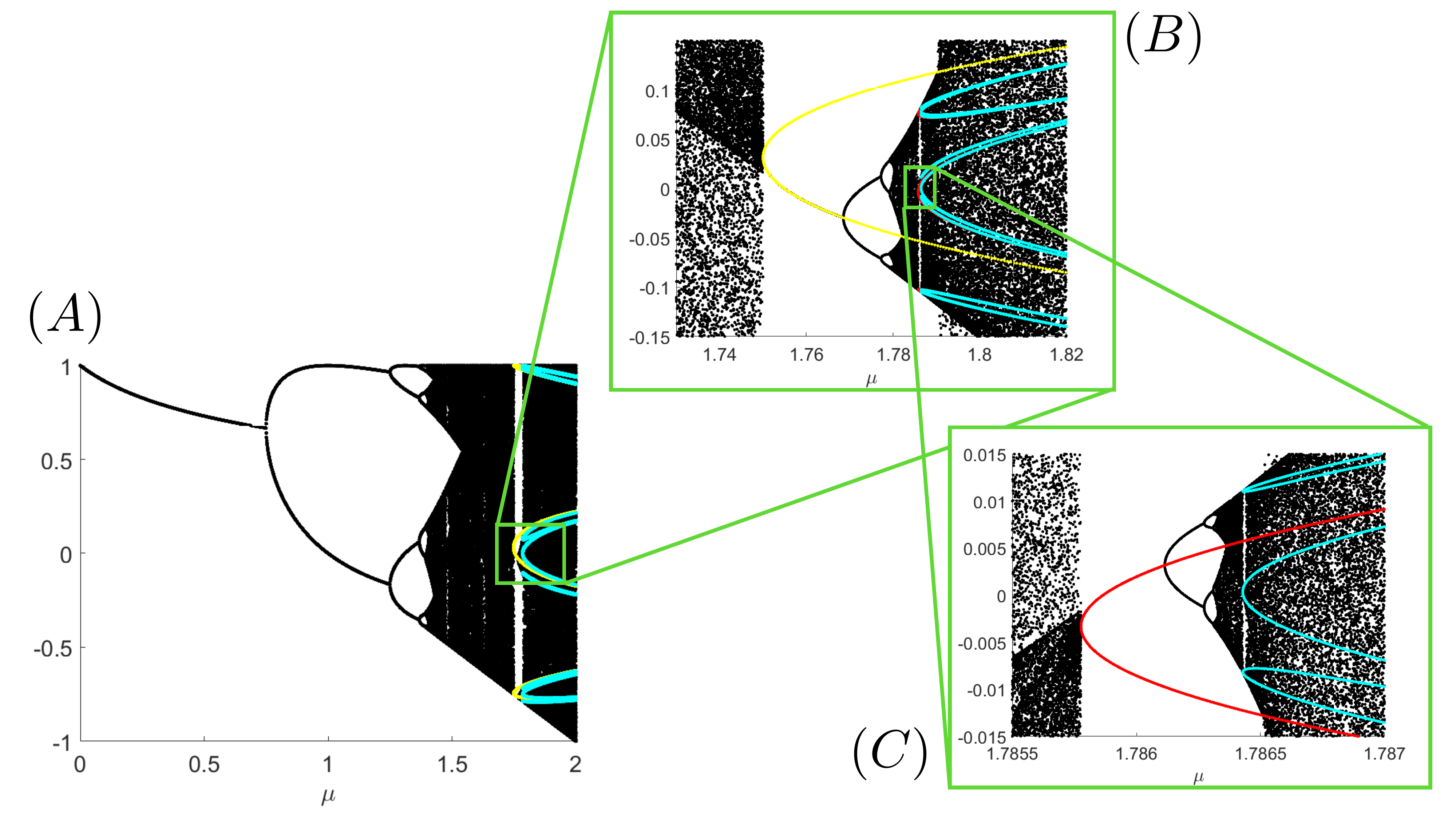}
\end{center}
\caption{\textbf{Trippling cascade from the stable period three orbit:}
we consider again the quadratic family  
(Equation \eqref{eq:quadFamily} with $d = 2$) and focus on 
the period three orbit which is born in a fold or tangency bifurcation 
near $\mu = 1.75$.  This gives rise to a stable and unstable branch
of period three, shown as three yellow curves.  
The period three participates in a period trippling bifurcation resulting 
in orbits of period $9$, $27$, ...
Frame (A) recalls the bifurcation diagram, and frame (B) 
zooms in on the middle component curve of the period three
so that we can see the birth of the period 9 orbit near 
$\mu = 1.785$, depicted in red.  Frame (C) zooms in further so that we can see 
three components of the period 27 family, depicted in blue. For period 81 and beyond it
is increasingly difficult to resolve the picture graphically, but the 
appropriate families can be located using the universal
 constant $\lambda_3$ given in the second row of Table \ref{table:1}. }
\label{fig:cascades31}
\end{figure}

\begin{remark}[Related existence and multiplicity results] \label{rem:large_m}
{\em
The results of the present study are complementary to those 
of the 1984 paper \cite{MR0763756} by Eckmann, Epstein, and Wittwer,
where the authors prove an existence theorem for fixed points of $R_m$,  
with $m$ ``large enough'' (for a particular fixed family of kneading sequences 
depending on $m$.) They derive asymptotic expansions for 
$f_m$, $\lambda_m$, and $\alpha_m$, which hold for large enough $m$.  
By contrast, the techniques of the present work facilitate the proof of  
existence results for any non-perturbative value of $m$, and any 
kneading sequence.  

An interesting problem for 
future consideration would be to determine (possibly with computer assistance)
an explicit bound $m_\infty$ such that the results of  \cite{MR0763756}
hold for $m \geq m_\infty$. Most likely, such $m_\infty$ would be greater than 
$10$, but one could the try to push the methods of the present work
to establish the results for $11 \leq m \leq m_\infty$, and therefore
obtain quantitative results for renormalization fixed points for all $m$.

We also mention the related perturbative work of Eckmann and Wittwer \cite{MR788690},
where they study fixed points of the $m=2$ renormalization operator 
as a function of the degree.  That is, they prove the existence of 
renormalization fixed points of degree $d$ near members of the 
standard family $q_{\mu, 2d}(x) = 1 - \mu x^{2d}$ in the limit as $d \to \infty$.
This is a highly singular perturbation and the authors use 
Borel transform techniques to analyze the resulting divergent power series.
}
\end{remark}

\begin{remark}[Results from the wider literature] \label{rem:penPaper}
{\em
The discussion above is heavily skewed toward the 
perspective of the present work, in the sense that we have emphasized 
the literature on numerical simulations and computer assisted proofs.  
That being said, the general theory of one dimensional unimodal maps is among the
most fully developed and celebrated in dynamical systems theory.
While a thorough review this literature is a task far beyond the scope of the
present work, we mention briefly the existence proofs of Sullivan and McMullen
\cite{MR0934326,MR1184622,mcmullen1996renormalization},
and also results on the local analysis of stable/unstable manifolds
and the more global invariant renormalization 
horseshoe found in the work of Lyubich \cite{MR1689333,MR1935840}.
Briefly, a renormalization horseshoe is a heteroclinic loop 
in function space, formed by a chain of transverse intersections 
between the unstable and stable manifolds of 
various renormalization fixed points.

The results mentioned in the preceding paragraph are proven for
the $m = 2$ renormalization operator and degree $d=2$ fixed points. 
Extensions to higher order unimodal maps and their
renormalization horseshoes are found in the work 
of Avila and Lyubich \cite{MR2854860}.  
Of particular interest to us, they show that every
fixed point of $R_m$, $m \geq 2$ is hyperbolic, with a single unique unstable 
eigenvalue and the remainder of the spectrum is stable. 
We refer also to the review article of Lyubich~\cite{MR2968954}
 for a substantive discussion of this literature 
complete with many additional references.  
}
\end{remark}

\begin{remark}[Principle of combinatorial rescaling of exponents (PASCE)] \label{rem:PASCE}
{\em
Renormalization horseshoes are considered from a
quantitative perspective by de la Llave, Olvera, and Petrov in their paper
\cite{de2011combination}.  There, the authors present compelling numerical evidence 
for their \textit{Principle of Approximate Combination of Scaling Exponents}, 
or PACSE.  Very roughly speaking, the idea of PACSE is that if 
$p_1$ and $p_2$ are renormalization 
fixed points for $R_{m}$ and $R_{n}$ respectively
(associated with some fixed choice of finite combinatorics), 
having associated universal constants
(unstable eigenvalues) $\lambda_1, \lambda_2 > 1$, and if the unstable
manifold of $p_1$ intersects transversally the stable manifold of $p_2$ and vice versa,
then in some neighborhood of the resulting heteroclinic cycle
there is an infinite sequence $\{p_{k_j}\}_{j = 1}^\infty$
of fixed points for renormalization operators $R_{k_j}$.  
Moreover, the universal constants 
$\lambda_{k_j}$ (and resulting combinatorics)
can be (approximately) computed via some explicit combinatorial formulas
involving $\lambda_1$ and $\lambda_2$.  The authors provide detailed numerical 
evidence for the PACSE, not only for renormalization of one dimensional maps, but 
also for renormalization of quasiperiodic transitions in circle maps, 
boundaries of Segel disks, and break down of invariant circles in 
area-preserving twist maps.  Beyond this, they argue that the PACSE should hold
for many other renormalization phenomena in mathematics and mathematical 
physics.

Given that so many results for renormalization fixed points and renormalization 
horseshoes for one dimensional unimodal maps
have been established using analytic arguments 
(Remark \ref{rem:penPaper} above) we should stress that an advantage of 
constructive, computational arguments like those developed in the 
present work (and in many other places, as discussed above) 
is that they provide information about the quantitative
features of fixed points and also precise bounds on the universal scalings.
The tools developed in the present work form part of 
a toolkit needed for verifying the hypotheses,
and applying the results of the PACSE theory in explicit examples. 
In a future work, we will combine the techniques developed below 
with the computer assisted methods for 
studying transverse connecting orbits in infinite dimensional 
discrete time dynamical systems developed by  
de la Llave and the third author in  \cite{MR3516860,MR3735860}.
This program will result in a computational framework for mathematically  
rigorous PACSE computations which could potentially be applied in
many renormalization settings.  
}
\end{remark}

The remainder of the paper is organized as follows. Section~\ref{sect:setup} introduces a zero-finding problem equivalent to Equation~\eqref{eq:renormDef}, that will be used for enclosing the renormalization fixed point via a Newton-Kantorovich argument, together with the precise Banach space in which this argument will be applied. A suitable finite dimensional projection and associated a priori error estimates are presented in Section~\ref{sec:aPriori}. The details of the Newton-Kantorovich argument are then laid out in Section~\ref{sec:fixed_point}, and the necessary estimates derived in Section~\ref{sec:bounds}. A similar procedure is presented in Section~\ref{sec:eigenvalue} in order to rigorously enclose the unstable eigenvalue of the obtained renormalization fixed points. Further results are then discussed in Section~\ref{sec:results}. Details about the computer-assisted parts of the proofs can be found in the appendix.

\section{Setup}\label{sect:setup}
Our approach, being based on Newton's method,
requires us to rewrite the fixed point  
problem for $R_m$ as an equivalent zero finding problem.
In the process we also ``unwrap'' one of the composition
terms by introducing a scalar 
variable $\alpha$.  

So, for $h \in C([-1,1], \mathbb{R})$ and $\alpha \in \mathbb{R}$,
let 
\begin{equation}
\label{eq:tR_m}
\tR_m(\alpha, h) := h^m(\alpha x),
\end{equation}
and define 
$\phi_m \colon \mathbb{R} \times C([-1,1], \mathbb{R}) \to C([-1,1], \mathbb{R})$
by 
	\begin{equation}
	\label{eq:phi_m}
		\phi_m(\alpha,h)(x)=\alpha h(x) - h^m(\alpha x)=\alpha h(x) - \tR_m(\alpha,h)(x).
	\end{equation}

The equation above is one equation in two unknowns 
$\alpha$ and $h$.  To balance the equations, 
we introduce a phase condition $h(0) = 1$,
and define the ``square'' operator 
	$\Phi_m \colon \mathbb{R} \times C([-1, 1], \mathbb{R}) \to  \mathbb{R} \times C([-1, 1], \mathbb{R})$
	by the formula 
	
	\begin{equation}
	\label{eq:Phi_m}
		\Phi_m(\alpha,h)=\begin{pmatrix}
			h(0)-1 \\ \phi_m(\alpha,h)
		\end{pmatrix}.
	\end{equation}
Our goal is to find a zero of $\Phi_m$, in an appropriate function space.
It is easy to check that, if $h$ is zero of $\Phi_m$ then $h^m(0) = \alpha$ follows
from the phase condition,  so that finding a zero of $\Phi_m$ is equivalent
to finding a fixed point of the $m$-th order 
Feigenbaum-Cvitanovi\'{c} renormalization operator 
defined in Equation 
\eqref{eq:renormDef}.

To solve the equation, we will discretize the domain of $\Phi_m$ using 
Chebyshev series.  To this end, 
choose $\rho\geq 1$ and define,
\begin{equation*}
	\ell^1_\rho = \left\{ h=h_0 + 2 \sum_{k=1}^{\infty}h_{k} T_{k}, \, \ \left\Vert h 
	\right\Vert_{\ell^1_\rho}:= 
	\vert h_0\vert +2\sum_{k=1}^{\infty}\vert h_{k}\vert \rho^{k} <\infty \right\},
\end{equation*}
where $T_k$ are the Chebyshev polynomials of the first kind defined recursively by 
\[
T_0(x) = 1, 
\]
\[
T_1(x) = x,
\]
and 
\[
T_{k+1}(x) = 2x T_k(x) - T_{k-1}(x), 
\]
for $k \geq 1$.  By making the change of variables 
$x = \cos(\theta)$ one obtains that 
\[
T_k(\cos(\theta)) = \cos(k \theta).
\]
This illustrates the usual connection between Chebyshev and 
Fourier cosine series. 

Throughout the paper, we refer to the $h_k$ as the Chebyshev coefficients of $h$. We recall that any function which is at least Lipschitz continuous on $[-1,1]$ (which will be the case for all functions considered in this work) admits a unique Chebyshev series expansion, see for instance~\cite{Tre13}.

We also consider the subspace of even functions within $\ell^1_\rho$, namely
	\begin{equation*}
		\ell^{1,\even}_\rho = \left\{ h\in\ell^1_\rho,\ h_{2k+1}=0\ \forall k\geq 0 \right\}.
	\end{equation*}
	
Finally, we  seek zeros of $\Phi_m$ in the space 
\begin{equation*}
		\X_\rho = 
		\R\times \ell^{1,\even}_\rho,
\end{equation*}
endowed with the following norm
\begin{equation*}
		\left\Vert (\alpha,h) \right\Vert_{\X_\rho} = 
		\vert\alpha\vert + \left\Vert h \right\Vert_{\ell^1_\rho}.
\end{equation*}
Note that $\ell^{1,\even}_\rho$ can be identified with a subspace of the space $X_\rho$ presented in the introduction (see Lemma~\ref{lem:C0_VS_ell1}).

\section{A priori estimates} \label{sec:aPriori}
	
We need a finite dimensional projection on $\X_\rho$, 
together with some a priori projection error estimates.

	\begin{definition}
		For any $K\in\N_{\geq 0}$, we denote by $\ell^{1,K}_\rho$ the subspace of polynomial functions of degree at most $K$ embedded in $\ell^1_\rho$, i.e.,
		\begin{equation*}
			\ell^{1,K}_\rho = \left\{ h\in\ell^1_\rho,\ h_{k}=0\ \forall k>K \right\},
		\end{equation*}
		and similarly for $\ell^{1,K,\even}_\rho$.
		We also introduce
		\begin{equation*}
			\X^K_\rho = \R\times \ell^{1,K,\even}_\rho.
		\end{equation*}
		For any $h\in \ell^1_\rho$, we denote by $\Pi^K h$ the unique polynomial of degree at most $K$ taking the same values as $h$ on the Chebyshev nodes $x_0,x_1,\ldots,x_K$ given by
		\begin{align*}
			x_k = \cos\left(\frac{K-k}{K}\pi\right), \quad \forall~0\leq k\leq K.
		\end{align*}
		Given $h$ in $\ell^1_\rho$, we denote by $\check h_k$ the coefficients (in the Chebyshev basis) of the interpolation polynomial $\Pi^K$, that is,
		\begin{align*}
			\Pi^K h = \check h_0 + 2\sum_{k=1}^K \check h_k T_k.
		\end{align*}
	\end{definition}
	We emphasize that the projection operator $\Pi^K : \ell^1_\rho \to \ell^{1,K}_\rho$ is not the truncation operator, i.e. $\check h_k \neq h_k$ (unless $h\in\ell^{1,K}_\rho$). However, these coefficients are still related, as made explicit in the following statement (see 
	\cite[Theorem 4.2]{Tre13}).
	\begin{lemma}
		\label{lem:aliasing}
		Let $\rho\geq 1$, $h\in\ell^1_\rho$, and  $K\in\N_{\geq 1}$. Then,
		\begin{align*}
			\check h_0 = h_0 + 2\sum_{l=1}^\infty h_{2Kl},
		\end{align*}
		\begin{align*}
			\check h_k = h_k + \sum_{l=1}^\infty \left(h_{2Kl - k} + h_{2Kl + k}\right) \qquad k=1,\ldots,K-1,
		\end{align*}
		and
		\begin{align*}
			\check h_K = h_K + \sum_{l=1}^\infty h_{2Kl+K}.
		\end{align*}
	\end{lemma}
	\begin{remark}
		When we restrict ourselves to $\ell^{1,\even}_\rho$, $\Pi^K$ can be defined in terms of half the nodes only, as $x_{K-k} = -x_k$ for all $0\leq k\leq K$, and whenever $h\in\ell^{1,\even}_\rho$ we have $h(x_{K-k}) = h(x_k)$.
	\end{remark}
	
	We will also use the notation $\Pi^{\infty}h := h - \Pi^{K}h$ for convenience, and extend $\Pi^{K}$ to $\X_\rho$ as $\Pi^{K}(\alpha, h) := (\alpha, \Pi^{K}h)$.

	\subsection{Useful facts about Chebyshev series and Chebyshev interpolation} \label{sec:Cheb}
	
	We recall here a few well known facts about Chebyshev approximation which will prove useful in this work, and again refer to~\cite{Tre13} for a wider discussion and further references.

	\begin{lemma}
		\label{lem:Tk_ellipse}
		Let $\rho\geq 1$, then
		\begin{equation*}
			\left\Vert T_k \right\Vert_{\CC^0_\rho}  = \frac{1}{2}\left(\rho^k+\rho^{-k}\right).
		\end{equation*}
	\end{lemma}
	\begin{proof}
		This identity follows directly from the fact that
		\begin{align}
			\label{eq:id_analytic}
			T_k\left(\frac{z+z^{-1}}{2}\right) = \frac{z^k+z^{-k}}{2},\quad \forall~z\in\C\setminus\{0\}.
		\end{align}
		In order to check that~\eqref{eq:id_analytic} holds, one can for instance notice that both sides are analytic functions of $z$ on $\C\setminus\{0\}$, which coincide on the unit circle (because $T_k(\cos\theta) = \cos(k\theta)$), therefore they must coincide everywhere on $\C\setminus\{0\}$.
	\end{proof}
	
	\begin{lemma}
		\label{lem:derCheb}
		Let $\rho > \rho' >1$. For any $h\in\ell^1_\rho$, we have that $h'\in\ell^1_{\rho'}$
		 and that 
		\begin{align*}
			h' = 2 \left(\left(\sum_{l=0}^\infty (2l+1)h_{2l+1} \right)T_0 + 2\sum_{k=1}^\infty \left(\sum_{l=0}^\infty (k+2l+1)h_{k+2l+1} \right)T_k \right).
		\end{align*}
	\end{lemma}

	\subsection{Decay of the coefficients and comparison of norms}
	
	While the space $\X_\rho$ in which we will apply our Newton-Kantorovich argument is equipped 
	with an $\ell^1_\rho$ norm, we often prefer to discuss our final results in terms of 
	$C^0_\nu$ norms.  Moreover, for certain steps in our arguments,
	the $C^0$ norms are easier to compute; for example when dealing with compositions.  In this subsection we consider the problem of passing from one to the other.
	We also require bounds which relate the norm of the derivative of a function 
	to its $C^0$ norm, after giving up a portion of the domain.  
	Of course we need explicit constants throughout, and 
	 these are controlled by leveraging the decay of the Chebyshev 
	 coefficients, as described in the following lemma.

	\begin{lemma}
		\label{lem:decay_fk}
		Let $\rho\geq 1$ and $h$ an analytic function on $\E_\rho$. Then, for all $k\geq 0$, the Chebyshev coefficient $h_k$ of $h$ satisfies
		\begin{align*}
			\left\vert h_k \right\vert \leq \frac{\left\Vert h\right\Vert_{\CC^0_\rho}}{\rho^k}.
		\end{align*}
		If $\rho>1$, then we also have an estimate on the coefficients of $\Pi^K h$, namely
		\begin{align*}
			\vert \check h_0 \vert \leq \left\Vert h  \right\Vert_{\CC^0_\rho} \frac{\rho^{2K}+1}{\rho^{2K}-1},
		\end{align*}
		\begin{align*}
			\vert \check h_k \vert \leq \frac{\left\Vert h  \right\Vert_{\CC^0_\rho}}{\rho^k} \frac{\rho^{2K} + \rho^{2k}}{\rho^{2K}-1} , \qquad k=1,\ldots,K-1,
		\end{align*}
		and
		\begin{align*}
			\vert \check h_K \vert \leq \frac{\left\Vert h  \right\Vert_{\CC^0_\rho}}{\rho^K} \frac{\rho^{2K}}{\rho^{2K}-1}.
		\end{align*}
	\end{lemma}
	\begin{proof}
		The estimate on the coefficients $h_k$ is nothing but \cite[Theorem 8.1]{Tre13}. Together with Lemma~\ref{lem:aliasing} it yields the estimates on the coefficients $\check h_k$.
	\end{proof}

	\begin{lemma}
		\label{lem:C0_VS_ell1}
		Let $\rho\geq 1$ and $h\in\ell^1_\rho$, then $h$ is analytic on $\E_\rho$ and 
		\begin{equation*}
			\left\Vert h\right\Vert_{\CC^0_\rho} \leq \left\Vert h \right\Vert_{\ell^1_\rho}.
		\end{equation*} 
		Conversely, if $\rho<\nu$ and $h$ is an analytic function on $\E_\nu$, then
		\begin{equation*}
			\left\Vert  h\right\Vert_{\ell^1_\rho} \leq \frac{\nu+\rho}{\nu-\rho}\left\Vert  h\right\Vert_{\CC^0_\nu}.
		\end{equation*}
		Furthermore, if $h$ is even, then
		\begin{equation*}
			\left\Vert  h\right\Vert_{\ell^1_\rho} \leq \frac{\nu^2+\rho^2}{\nu^2-\rho^2}\left\Vert  h\right\Vert_{\CC^0_\nu}.
		\end{equation*}
		
	\end{lemma}
	\begin{proof}
		The analyticity of $h$ on $\E_\rho$ follows from \cite[Theorem 8.3]{Tre13}, and the fact that the $\CC^0_\rho$ norm is controlled by the $\ell^1_\rho$ norm is straightforward:
		\begin{align*}
			\left\Vert h\right\Vert_{\CC^0_\rho} &\leq \vert h_0\vert \left\Vert T_0 \right\Vert_{\CC^0_\rho} +2\sum_{k=1}^{\infty}\vert h_k\vert \left\Vert T_k \right\Vert_{\CC^0_\rho},
		\end{align*}
		and Lemma~\ref{lem:Tk_ellipse} shows that $\left\Vert T_0 \right\Vert_{\CC^0_\rho} \leq \rho^k$.
		The second identity follows from Lemma~\ref{lem:decay_fk}. Indeed,
		\begin{align*}
			\left\Vert h \right\Vert_{\ell^1_\rho} &= \vert h_0\vert +2\sum_{k=1}^{\infty}\vert h_k\vert \rho^{ k} \\
			&\leq \left(1 +2\sum_{k=1}^{\infty}\left(\frac{\rho}{\nu}\right)^k \right) \left\Vert h \right\Vert_{\CC^0_\nu}\\
			&= \frac{\nu+\rho}{\nu-\rho} \left\Vert h \right\Vert_{\CC^0_\nu}.
		\end{align*}
		The last estimate is obtained by using the fact that half the $h_k$ are zero when $h$ is even.
	\end{proof}
	\begin{proposition}
		\label{prop:derivative}
		Let $1\leq\rho<\nu$,
		\begin{align*}
			\sigma_{\rho,\nu}^\even = \sup_{n\in\N_{\geq 1}} \frac{2n}{\nu^{2n}}\left(\rho\frac{\rho^{2n}-1}{\rho^2-1} +\rho^{-1}\frac{\rho^{-2n}-1}{\rho^{-2}-1} \right),
		\end{align*} 
		\begin{align*}
			\sigma_{\rho,\nu}^\odd = \sup_{n\in\N_{\geq 0}} \frac{2n+1}{\nu^{2n+1}}\left(1+\rho^2\frac{\rho^{2n}-1}{\rho^2-1} +\rho^{-2}\frac{\rho^{-2n}-1}{\rho^{-2}-1} \right),
		\end{align*}
		and
		\begin{align*}
			\sigma_{\rho,\nu} = \max\left(\sigma_{\rho,\nu}^\even ,\ \sigma_{\rho,\nu}^\odd  \right).
		\end{align*}
		Then, for all $h$ in $\ell^1_\nu$
		\begin{equation*}
			\left\Vert  h'\right\Vert_{\CC^0_\rho}\leq \sigma_{\rho,\nu} \left\Vert h\right\Vert_{\ell^1_\nu}.
		\end{equation*}
		Furthermore, if $h$ is even, then
		\begin{equation*}
			\left\Vert  h'\right\Vert_{\CC^0_\rho}\leq \sigma_{\rho,\nu}^\even \left\Vert h\right\Vert_{\ell^1_\nu}.
		\end{equation*}
		
	\end{proposition}
	\begin{proof}
		Starting from the formula of Lemma~\ref{lem:derCheb}, taking the $\CC^0_\rho$ norm and using Lemma~\ref{lem:Tk_ellipse}, we get
		\begin{align*}
			\left\Vert h' \right\Vert_{\CC^0_\rho} &\leq 2 \left(\sum_{l=0}^\infty \frac{2l+1}{\nu^{2l+1}}\left\vert h_{2l+1}\right\vert \nu^{2l+1} + \sum_{k=1}^\infty\sum_{l=0}^\infty \frac{k+2l+1}{\nu^{k+2l+1}}(\rho^k+\rho^{-k}) \left\vert h_{k+2l+1}\right\vert \nu^{k+2l+1}\right) \\
			&= 2\left(\sum_{n=0}^{\infty}\left\vert h_{2n+1}\right\vert\nu^{2n+1} \left( \frac{2n+1}{\nu^{2n+1}}+\sum_{k=1}^n \frac{2n+1}{\nu^{2n+1}} (\rho^{2k}+\rho^{-2k})\right) \right. \\
			&\qquad \left. + \sum_{n=0}^{\infty}\left\vert h_{2n+2}\right\vert \nu^{2n+2} \sum_{k=0}^n \frac{2n+2}{\nu^{2n+2}}(\rho^{2k+1}+\rho^{-(2k+1)}) \right) \\
			&\leq 2\left(\sigma_{\rho,\nu}^\odd  \sum_{n=0}^{\infty}\left\vert h_{2n+1}\right\vert\nu^{2n+1} + \sigma_{\rho,\nu}^\even \sum_{n=1}^{\infty}\left\vert h_{2n}\right\vert \nu^{2n} \right) \\
			&\leq \sigma_{\rho,\nu} \left\Vert h\right\Vert_{\ell^1_\nu}.
		\end{align*}
		When $h$ is even, all the $h_{2n+1}$ vanish, therefore we indeed get $ \left\Vert h' \right\Vert_{\CC^0_\rho} \leq \sigma_{\rho,\nu}^\even \left\Vert h\right\Vert_{\ell^1_\nu}$.
	\end{proof}
	\begin{remark}
		Since $\nu>\rho$, the suprema defining $\sigma^\even_{\rho,\nu}$ and $\sigma^\odd_{\rho,\nu}$ are in fact maxima, and we provide in Lemma~\ref{lem:sigma_expl} a simple procedure allowing to explicitly compute them. When $\rho = 1$, the formula for these constants simply become
\begin{align*}
\sigma_{\rho,\nu}^\even = \sup_{n\in\N_{\geq 1}} \frac{(2n)^2}{\nu^{2n}}, \quad \text{and} \quad \sigma_{\rho,\nu}^\odd = \sup_{n\in\N_{\geq 0}} \frac{(2n+1)^2}{\nu^{2n+1}},
\end{align*}
and it is then straightforward to check that these suprema can be computed as
\begin{align*}
\sigma_{\rho,\nu}^\even = \max_{1\leq n \leq \lfloor \frac{1}{\ln\nu}\rfloor } \frac{(2n)^2}{\nu^{2n}},\quad \text{and}\quad \sigma_{\rho,\nu}^\odd = \max_{0\leq n \leq \lfloor \frac{1}{\ln\nu}\rfloor} \frac{(2n+1)^2}{\nu^{2n+1}}.
\end{align*}
	\end{remark}
	
	\subsection{Projection errors}
	
	This subsection contains crucial projection/interpolation error estimates, again with explicit constants. Because our analysis involves going back and forth between $\ell^1_\rho$ and $\CC^0_\nu$, we need to estimate the interpolation error in $\CC^0$ in terms of the $\ell^1$ norm, and vice-versa.
	
	\begin{lemma}
		\label{lem:PiVSrho}
		Let $\rho\geq 1$ and $h\in\ell^1_\rho$. Then
		\begin{equation*}
			\left\Vert \Pi^K h \right\Vert_{\ell^1_\rho} \leq \left\Vert h \right\Vert_{\ell^1_\rho}.
		\end{equation*}
	\end{lemma}
	\begin{proof}
		This is an immediate consequence of Lemma~\ref{lem:aliasing}.
	\end{proof}
	
	\begin{remark}
		The above estimate is one of the reasons we conduct our fixed point argument using the $\ell^1_\rho$ norm rather than the $\CC^0_\rho$ norm. Indeed, in the latter norm, $\Pi^K h$ is only controlled by $h$ times a constant behaving roughly like $\rho^K$, which quickly becomes detrimental when $\rho$ is larger than $1$ (in many of our computer-assisted proofs we will use $\rho = 2$).
	\end{remark}

	\begin{proposition}
		\label{prop:interp_alpharho}
		Let $1\leq \rho\leq\nu$, and $h\in\ell^1_\nu$. Then
		\begin{equation*}
			\left\Vert h-\Pi^Kh \right\Vert_{\CC^0_\rho} \leq \Upsilon^{0,1}_{\rho,\nu,K} \left\Vert h  \right\Vert_{\ell^1_\nu},
		\end{equation*}
		where
		\begin{equation*}
			\Upsilon^{0,1}_{\rho,\nu,K} = \frac{1}{2} \frac{\rho^{K-1}+\rho^{-(K-1)}+\rho^{K+1}+\rho^{-(K+1)}}{\nu^{K+1}}.
		\end{equation*}
		Furthermore, if $h$ is even and $K$ is even,
		\begin{equation*}
			\left\Vert h-\Pi^Kh \right\Vert_{\CC^0_\rho} \leq \Upsilon^{0,1,\even}_{\rho,\nu,K} \left\Vert h  \right\Vert_{\ell^1_\nu},
		\end{equation*}
		where
		\begin{equation*}
			\Upsilon^{0,1,\even}_{\rho,\nu,K} = \frac{1}{2} \frac{\rho^{K-2}+\rho^{-(K-2)}+\rho^{K+2}+\rho^{-(K+2)}}{\nu^{K+2}}.
		\end{equation*} 
	\end{proposition}
	\begin{proof}
		The starting point is Lemma~\ref{lem:aliasing}, which allows us to write the interpolation error as
		\begin{align}
			\label{eq:interp_error}
			h - \Pi^Kh & = \left(2\sum_{l=1}^\infty h_{2Kl}\right) T_0 + 2 \sum_{k=1}^{K-1} \left(\sum_{l=1}^\infty \left( h_{2Kl+k}  + h_{2Kl-k}\right)\right) T_k  + 2\left(\sum_{l=1}^\infty h_{2Kl+K}\right) T_K \nonumber \\
			&\quad + 2\sum_{k=K+1}^\infty h_k T_k.
		\end{align}
		Taking the $\CC^0_\rho$ norm and simply using Lemma~\ref{lem:Tk_ellipse} together with the triangular inequality, we get
		\begin{align}
			\label{eq:interp_error_C0}
			\left\Vert h - \Pi^Kh \right\Vert_{\CC^0_\rho} &\leq   2\sum_{l=1}^\infty \left\vert h_{2Kl}\right\vert  + \sum_{k=1}^{K-1} \sum_{l=1}^\infty \left( \left\vert h_{2Kl+k}\right\vert  + \left\vert h_{2Kl-k}\right\vert \right) \left(\rho^k+\rho^{-k}\right) \nonumber\\
			&\quad + \sum_{l=1}^\infty \left\vert h_{2Kl+K}\right\vert \left(\rho^K+\rho^{-K}\right) + \sum_{k=K+1}^\infty \left\vert h_k \right\vert \left(\rho^k+\rho^{-k}\right).
		\end{align}
		Reorganizing the terms sightly, and taking worst cases in $k$ and $l$, we end up with
		\begin{align*}
			\left\Vert h - \Pi^Kh \right\Vert_{\CC^0_\rho} & \leq   \frac{2+\rho^{2K}+\rho^{-2K}}{\nu^{2K}}\sum_{l=1}^\infty \left\vert h_{2Kl}\right\vert \nu^{2Kl}  \\
			&\quad + \frac{\rho+\rho^{-1}+\rho^{2K+1}+\rho^{-(2K+1)}}{\nu^{2K+1}} \sum_{k=1}^{K-1} \sum_{l=1}^\infty  \left\vert h_{2Kl+k}\right\vert \nu^{2Kl+k} \\
			&\quad + \frac{\rho^{K-1}+\rho^{-(K-1)}+\rho^{K+1}+\rho^{-(K+1)}}{\nu^{K+1}}\sum_{k=1}^{K-1} \sum_{l=1}^\infty  \left\vert h_{2Kl-k}\right\vert \nu^{2Kl-k} \\ 
			&\quad + \frac{\rho^{K}+\rho^{-K}+\rho^{3K}+\rho^{3K}}{\nu^{3K}} \sum_{l=1}^\infty \left\vert h_{2Kl+K}\right\vert \nu^{2Kl+K} \\
			&\leq \frac{1}{2}\max_{a\in\{-(K-1),0,1,K\}} \frac{\rho^a+\rho^{-a}+\rho^{2K+a}+\rho^{-(2K+a)}}{\nu^{2K+a}}  \left\Vert h  \right\Vert_{\ell^1_\nu}.
		\end{align*}
		Since $\rho\leq \nu$, the term 
		\begin{align*}
			\frac{\rho^a+\rho^{-a}+\rho^{2K+a}+\rho^{-(2K+a)}}{\nu^{2K+a}} = \frac{1}{\nu^{2K}}\left(\left(\frac{\rho}{\nu}\right)^a + \left(\frac{1}{\rho\nu}\right)^a\right) + \left(\frac{\rho}{\nu}\right)^{2K+a} + \left(\frac{1}{\rho\nu}\right)^{2K+a}
		\end{align*}
		is non-increasing with $a$, hence the maximum over $a$ is reached for $a=-(K-1)$. 
		
		When $h$ is even, only the $h_k$ with $k$ even remain in the above computation. If $K$ is also even, the worst term (which is the factor in front of $h_{2Kl-k}$ for $l=1$ and $k=K-1$) drops out, and the next worst one (in front of $h_{2Kl-k}$ for $l=1$ and $k=K-2$) give the announced constant.
	\end{proof}
	
	\begin{proposition}
		\label{prop:interp_alpharho0}
		Let $1\leq \rho<\nu$, and $h$ an analytic function on $\E_\nu$. Then
		\begin{equation*}
			\left\Vert h-\Pi^Kh \right\Vert_{\ell^1_\rho} \leq \Upsilon^{1,0}_{\rho,\nu,K} \left\Vert h  \right\Vert_{\CC^0_\nu},
		\end{equation*}
		where
		\begin{equation*}
			\Upsilon^{1,0}_{\rho,\nu,K}=\frac{2}{\nu^{2K}-1}\left(\frac{\rho}{\nu-\rho}\left(1-\left(\frac{\rho}{\nu}\right)^K\right)+\frac{(\nu\rho)^K-1}{\nu\rho-1}\right)+\frac{2\rho}{\nu-\rho}\left(\frac{\rho}{\nu}\right)^K,
		\end{equation*}
		Furthermore, if $h$ is even and $K$ is even,
		\begin{equation*}
			\left\Vert h-\Pi^Kh \right\Vert_{\ell^1_\rho} \leq \Upsilon^{1,0,\even}_{\rho,\nu,K} \left\Vert h  \right\Vert_{\CC^0_\nu},
		\end{equation*}
		where
		\begin{equation*}
			\Upsilon^{1,0,\even}_{\rho,\nu,K}=\frac{2}{\nu^{2K}-1}\left(\frac{\rho^2}{\nu^2-\rho^2}\left(1-\left(\frac{\rho}{\nu}\right)^K\right)+\frac{(\nu\rho)^K-1}{(\nu\rho)^2-1}\right)+\frac{2\rho^2}{\nu^2-\rho^2}\left(\frac{\rho}{\nu}\right)^K.
		\end{equation*}
	\end{proposition}
	\begin{proof}
		The starting point is again to write the interpolation error as~\eqref{eq:interp_error}, and then to take the $\ell^1_\rho$ norm instead of the $\CC^0_\rho$ norm, which simply means each $\frac{\rho^k+\rho^{-k}}{2}$ should be replaced by $\rho^k$  in~\eqref{eq:interp_error_C0}. Next, we estimate each $\left\vert h_k\right\vert$ using the first part of Lemma~\ref{lem:decay_fk} and the fact that $h$ is analytic on $\E_\nu$, which yields
		\begin{align*}
			&\left\Vert h - \Pi^Kh \right\Vert_{\ell^1_\rho}   \\
			& \quad \leq  \left( 2\sum_{l=1}^\infty \frac{1}{\nu^{2Kl}}  + 2 \sum_{k=1}^{K-1} \sum_{l=1}^\infty \left( \frac{\rho^k}{\nu^{2Kl+k}}  + \frac{\rho^k}{\nu^{2Kl-k}} \right)  + 2\sum_{l=1}^\infty \frac{\rho^K}{\nu^{2Kl+K}} + 2\sum_{k=K+1}^\infty \frac{\rho^k}{\nu^{k}} \right) \left\Vert h \right\Vert_{\CC^0_\nu},
		\end{align*}
		and obtaining the formula for $\Upsilon^{1,0}_{\rho,\nu,K}$ is just a matter of putting together the sums of all those geometric series. Also using that only the terms with even indices remain when $h$ is even yields the second constant.
	\end{proof}

	\subsection{Practical considerations}
	
	We make here several remarks, related to the way we actually use some of the theoretical estimates presented in the previous two subsections in practice.
	
	\begin{remark}
		\label{rem:2ways}
		In various places below, we need to compute or estimate the coefficients of $\Pi^K g$ for some function $g$. If $g$ is explicit enough, we can compute each entry of the vector $\left(g(x_k)\right)_{0\leq k\leq K}$ and then get the coefficients of $\Pi^K g$ using the DFT (or more precisely, the Discrete Cosine Transform, see Appendix~\ref{sec:DCT}). Similarly, if we are only able to get component-wise upper bounds for $\left(\left\vert g(x_k) \right\vert \right)_{0\leq k\leq K}$, we get upper-bounds for the coefficients of $\Pi^K g$. While this is usually fine for moderately large values of $k$, when $k$ becomes large such estimates may fail to capture the expected decay of the coefficients, which can be problematic when $\rho$ is somewhat larger than $1$, for instance if we need to compute or estimate $\left\Vert  \Pi^K g \right\Vert_{\ell^1_\rho}$.
		On way to remedy this is to also compute or estimate $\left\Vert g\right\Vert_{\CC^0_{\nu}}$ for some $\nu\geq \rho$, and then use the second part of Lemma~\ref{lem:decay_fk} to get an estimate on the coefficients of $\Pi^K g$ with a guaranteed decay at a rate $\nu^{-k}$. For each coefficient, we can then take the minimum between the estimate obtained via the values at the nodes, and the estimate obtained form the $\CC^0_\nu$ norm. This strategy is reminiscent of the one presented in~\cite{Les18bis}.
	\end{remark}

	\begin{remark}
		Another recurring task will be to compute quantities like $\left\Vert g \right\Vert_{\CC^0_\nu}$, where $g$ is some analytic function on $\E_\nu$ (usually a polynomial). It is worth noticing that we in fact only need to compute the supremum of $\vert g\vert$ on the boundary of $\E_\nu$, because the maximum is necessarily reached there thanks to the maximum modulus principle. If we have access to the Chebyshev coefficients of a polynomial $g$, an efficient way of rigorously enclosing this supremum using interval arithmetic together with the FFT, coming from~\cite{BerBreLesMir24} (see also~\cite{HarSan24}), is recalled in Appendix~\ref{sec:intervalFFT}. Unfortunately, such an algorithm is of little use when $g$ is the composition of several polynomials, and we only know the coefficients of the individual pieces. In that case, we instead adaptively subdivide the boundary of $\E_\nu$ into small pieces, and directly evaluate over these pieces with interval arithmetic to get an enclosure of the supremum of $\vert g\vert$.
	\end{remark}
	
	\begin{remark}
		\label{rem:opt}
		Many estimates to come will be of the form $C_\nu \left\Vert g \right\Vert_{\CC^0_\nu}$, for some function $g$ and $\nu>1$, where $C_\nu$ is a computable constant depending on $\nu$. Whenever we face such a quantity, we numerically optimize the value of $\nu$ in order to make this as small as possible.
	\end{remark}
	
	In order to use such estimates, we need to know that g is actually analytic on some Bernstein ellipse of explicit size. The following Lemma can prove useful in determining (lower bounds on) domains of analytically.
	
	\begin{lemma}
		\label{lem:inclusion_fE}
		Let $\nu\geq 1$, $\psi$ analytic on $\E_\nu$, and
		\begin{align*}
			\eta = \max_{z\in \partial\E_\nu} \vert \psi(z)-1\vert + \vert \psi(z)+1\vert.
		\end{align*}
		Then, for all $\rho\geq\frac{\eta+\sqrt{\eta^2-4}}{2}$, $\psi(\E_\nu)\subset \E_\rho$.
	\end{lemma}

	\section{Fixed point reformulation}\label{sec:fixed_point}
	In the remainder of the paper, $K$ denotes a positive even integer, and $(\ba,\bh)$ an element of $\X^K_\rho$, which should be thought of as an approximate zero of $\Phi_m$ (recall equation~\eqref{eq:Phi_m}), obtained numerically, satisfying $0<\ba<1$. Our goal is to prove the existence of an exact zero of $\Phi_m$ (i.e., of an exact fixed point of $R_m$) near $(\ba,\bh)$, and to provide an explicit and small error bound.
	
	We define
	\begin{equation*}
		J^\dag= \Pi^K \restriction{D\Phi_m(\ba,\bh)}{\X_\rho^K},
	\end{equation*}
	which is a linear operator on $\X_\rho^K$.
	
	Since $\X_\rho^K$ is finite dimensional (of dimension $K/2+1$), we can compute numerically an approximate inverse $J$ of $J^\dag$. Finally, we define the linear operator $A$ by
	\begin{equation*}
		\left\{
		\begin{aligned}
			&A\, \Pi^K (\alpha,h) = J\, \Pi^K(\alpha,h) \\
			&A\, (I-\Pi^K) (\alpha,h) = \left(0,\, \frac{1}{\ba}\,  (I-\Pi^K)h \right).
		\end{aligned}
		\right.
	\end{equation*}
	This leads to the fixed-point operator
	\begin{align}
	\label{eq:defT}
		T: (\alpha,h)\mapsto (\alpha,h) - A\, \Phi_m(\alpha,h),
	\end{align}
	and our goal is now to prove that $T$ is a contraction on a small neighborhood of $(\ba,\bh)$ in $\X_\rho$. This will be accomplished thanks to a Newton-Kantorovich-type argument, which has become very common for computer-assisted proofs~\cite{AriKocTer05,BerBreLesVee21,DayLesMis07,Plu92,Yam98}.

	\begin{theorem}
		\label{th:fixed_point}
		Let $r^*\in(0,+\infty]$. Assume there exist nonnegative constants $Y$ and $Z$ such that
		\begin{subequations}
			\label{eq:YZ}
			\begin{align}
				\left\Vert A\Phi_m(\ba,\bh) \right\Vert_{\X_\rho} &\leq Y \label{eq:Y}\\
				\sup_{(\alpha,h)\in\B_{\X_\rho}\left((\ba,\bh),r^*\right)}\left\Vert 
				I - AD\Phi_m(\alpha,h) \right\Vert_{\X_\rho}  &\leq Z, \label{eq:Z}
			\end{align}
		\end{subequations}
		where $\B_{\X_\rho}\left((\ba,\bh),r^*\right)$ is the closed ball of center $(\ba,\bh)$ and radius $r^*$ in $\X_\rho$.
		If
		\begin{align}
			Z &< 1, \label{eq:cond1}
		\end{align}
		then, for any $r$ satisfying 
		\begin{align}
			\label{eq:r}
			\frac{Y}{1-Z} \leq r \leq r^*,
		\end{align}
		there exists a unique $(\alpha_*,h_*) \in  \B_{\X_\rho}\left((\ba,\bh),r \right)$
		so that 
		\[
		\Phi_m(\alpha_*, h_*) = 
		\left(
		\begin{array}{c}
		0 \\
		\mathbf{0}
		\end{array}
		\right).
		\]
	\end{theorem}
	\begin{remark}
		\label{rem:fixed_point}
		Often in the literature, the above $Z$ estimate is considered for an arbitrary $r<r^*$ and split into two parts:
		\begin{align*}
			\left\Vert I - AD\Phi_m(\ba,\bh) \right\Vert_{\X_\rho} &\leq Z_1 \\
			\sup_{(\alpha,h)\in\B_{\X_\rho}\left((\ba,\bh),r^*\right)}  \left\Vert AD^2\Phi_m(\alpha,h)
			 \right\Vert_{\X_\rho} & \leq Z_2,
		\end{align*}
		so that
		\begin{align*}
			\sup_{(\alpha,h)\in\B_{\X_\rho}\left((\ba,\bh),r\right)} \left\Vert 
			I - AD\Phi_m(\alpha,h) \right\Vert_{\X_\rho} \leq Z_1 + Z_2  r,
		\end{align*}
		This allows to isolate the crucial part, namely $I - AD\Phi_m(\ba,\bh)$, and to estimate it as sharply as possible, because the derivative is now taken at a fixed and explicit point, but it then requires to also control locally the second derivative (or at least to get a Lipschitz bound on the first derivative locally). For our specific problem, where $\Phi_m$ contains a composition operator, looking at higher order derivatives means getting more and more complicated formulas, which we avoid by directly working with~\eqref{eq:Z}. The supremum over $\B_{\X_\rho}\left((\ba,\bh),r^*\right)$ is then handled directly by using interval arithmetic in combination with some a priori error estimates (see Section~\ref{sec:Z}). The downside of this approach is that we get less sharp bounds, but how less sharp they are really depends on the choice of $r^*$. In practice, we do take $r^*$ really small (see Section~\ref{sec:results} for explicit values), which alleviates this drawback.  The main cost 
		of the maneuver is that we do then absolutely need to get rather sharp $Y$ bounds: if $Y>r^*$, then there are no $r$ satisfying~\eqref{eq:r}. Another byproduct of this choice, which is of no real consequence for this work, it that we only obtain isolation results on the (now very small)
		neighborhood of size $r_*$.
	\end{remark}

	\section{Bounds in Theorem~\ref{th:fixed_point}}
	\label{sec:bounds}
	
	In this section, we derive computable estimates satisfying assumption~\eqref{eq:YZ} of Theorem~\ref{th:fixed_point}. We start by introducing notation for the terms appearing in the Frechet derivative of $\phi_m$ (recall equation~\eqref{eq:phi_m}), and then derive separately a $Y$ bound and a $Z$ bound.
	
	For any $(\ta,\th)$ and $(\alpha,h)$ in $\X_\rho$,
	
	\begin{equation}
		\label{eq:dphidalpha}
		\partial_{\alpha}\phi(\ta,\th)\alpha= \alpha(
		\th-\tf ),
	\end{equation} 
	where
	\begin{align*}
		\tf(x) = x\prod_{j=0}^{m-1} \th'\left(\th^j(\ta x)\right),
	\end{align*}
	and
	\begin{equation}
		\label{eq:dphidh}
		\partial_{h}\phi(\ta,\th)h (x)   = \ta h(x) - \sum_{j=0}^{m-1} \tg_j(x)  h\left(\th^{j}(\ta x)\right),
	\end{equation}
	where
	\begin{align*}
		\tg_j(x) = \prod_{l=j+1}^{m-1} \th'\left(\th^l(\ta x) \right),\quad j=0,\ldots,m-1, 
	\end{align*}
	with the convention that the empty product is equal to $1$, i.e. $\tg_{m-1} = 1$.

\subsection[Y estimate in \eqref{eq:Y}]{$\bm{Y}$ estimate in \eqref{eq:Y}}
\label{sec:Y}
	
	This subsection is devoted to the estimate $Y$ satisfying~\eqref{eq:Y}. We split $A\Phi_m(\ba,\bh)$ as $$\Pi^K(A\Phi_m(\ba,\bh)) + (I-\Pi^K)(A\Phi_m(\ba,\bh)),$$ and estimate both terms separately. The bounds for the terms are denoted by $Y^K$ and $Y^{\infty}$ respectively, and sometimes referred to as the \emph{finite part} and the \emph{tail part}.
	
	For the finite part, we simply take
	\begin{equation*}
		Y^K=\left\Vert  J\, \Pi^K \Phi_m(\ba,\bh) \right\Vert_{\X_\rho}.
	\end{equation*}
	That is, we compute the coefficients of $\Pi^K \Phi_m(\ba,\bh)$ using Remark~\ref{rem:2ways}, and then simply multiply the result by $J$ and compute the $\ell^1_\rho$ norm of the result.

For the tail part, we have to estimate 
	\begin{align}
	\label{eq:Yinf_startingpoint}
		\frac{1}{\vert\ba\vert}\left\Vert  \left(I-\Pi^K\right) \Phi_m(\ba,\bh) \right\Vert_{\X_\rho} = \frac{1}{\vert\ba\vert}\left\Vert  \left(I-\Pi^K\right) \left(\ba\bh - \bh^m (\ba\cdot)\right) \right\Vert_{\ell^1_\rho}.
	\end{align}
Since $\ba\bh - \bh^m (\ba\cdot)$ is itself a polynomial, in principle once should be able to compute its coefficients exactly, and then to exactly evaluate the r.h.s. of~\eqref{eq:Yinf_startingpoint}. However, $\ba\bh - \bh^m (\ba\cdot)$ is of very large degree ($K^m$), therefore computing its coefficients accurately enough can be very challenging in practice. In particular, if the magnitude of the obtained coefficients start plateauing around machine epsilon, the resulting $\ell^1_\rho$ norm could become extremely large if $\rho>1$, which is the case here.
In order to alleviate this difficulty, and estimate the r.h.s. of~\eqref{eq:Yinf_startingpoint} as sharply as possible, we will split it into two parts: one which should be \emph{close} to $\left(I-\Pi^K\right) \left(\ba\bh - \bh^m (\ba\cdot)\right)$ but that we can compute \emph{accurately}, and a second part we we can only estimate but which should hopefully be of relatively small magnitude compared to the first part. More concretely, we consider an integer $K_Y$ and use the triangle inequality
\begin{align*}
		\frac{1}{\vert\ba\vert}\left\Vert  \left(I-\Pi^K\right) \left(\ba\bh - \bh^m (\ba\cdot)\right) \right\Vert_{\ell^1_\rho} &\leq 
		\frac{1}{\vert\ba\vert}\left\Vert  \Pi^{K_Y} \left(\ba\bh - \bh^m (\ba\cdot)\right) - \Pi^{K} \left(\ba\bh - \bh^m (\ba\cdot)\right) \right\Vert_{\ell^1_\rho} \\
		&\quad  + \frac{1}{\vert\ba\vert}\left\Vert  \left(I-\Pi^{K_Y}\right) \left(\ba\bh - \bh^m (\ba\cdot)\right) \right\Vert_{\ell^1_\rho}.
	\end{align*}
In practice, $K_Y$ should be chosen larger than $K$, so that $\Pi^{K_Y} \left(\ba\bh - \bh^m (\ba\cdot)\right)$ approximates $\ba\bh - \bh^m (\ba\cdot)$ well, but not too large so that the coefficients of $\Pi^{K_Y} \left(\ba\bh - \bh^m (\ba\cdot)\right)$ can still be computed accurately. The interpolation error $\left(I-\Pi^{K_Y}\right)$ is then estimated using Proposition~\ref{prop:interp_alpharho0}, which yields
\begin{align*}
		\frac{1}{\vert\ba\vert}\left\Vert  \left(I-\Pi^K\right) \left(\ba\bh - \bh^m (\ba\cdot)\right) \right\Vert_{\ell^1_\rho} &\leq 
		\frac{1}{\vert\ba\vert}\left\Vert  \Pi^{K_Y} \left(\ba\bh - \bh^m (\ba\cdot)\right) - \Pi^{K} \left(\ba\bh - \bh^m (\ba\cdot)\right) \right\Vert_{\ell^1_\rho} \\
		&\quad  + \frac{\Upsilon^{1,0,\even}_{\rho,\nu,K_Y}}{\vert\ba\vert} \left\Vert \ba\bh - \bh^m (\ba\cdot) \right\Vert_{\CC^0_{\nu}} := Y^{\infty},
	\end{align*}
for some $\nu>\rho$ chosen according to Remark~\ref{rem:opt}. Note that $K_Y>K$ makes the constant $\Upsilon^{1,0,\even}_{\rho,\nu,K_Y}$ smaller than the $\Upsilon^{1,0,\even}_{\rho,\nu,K}$ that would have appeared if we had use Proposition~\ref{prop:interp_alpharho0} directly on~\eqref{eq:Yinf_startingpoint}, and therefore the second term in $Y^{\infty}$ should in principle be small compared to $\frac{1}{\vert\ba\vert}\left\Vert  \Pi^{K_Y} \left(\ba\bh - \bh^m (\ba\cdot)\right) - \Pi^{K} \left(\ba\bh - \bh^m (\ba\cdot)\right) \right\Vert_{\ell^1_\rho}$, which we can just compute.
	
	\subsection[Z estimate in \eqref{eq:Z}]{$\bm{Z}$ estimate in \eqref{eq:Z}}
	\label{sec:Z}
	
	This subsection is devoted to the estimate $Z$ satisfying~\eqref{eq:Z}, which will be split in three parts:
	\begin{align}
		\sup_{(\ta,\th)\in B_{\X_\rho}\left((\ba,\bh),r^*\right)} \left\Vert DT(\ta,\th) \right\Vert_{\X_\rho} &= \sup_{(\ta,\th)\in B_{\X_\rho}\left((\ba,\bh),r^*\right)} \left\Vert I - AD\Phi_m(\ta,\th)\right\Vert_{\X_\rho} \nonumber\\
		&\leq \sup_{(\ta,\th)\in B_{\X_\rho}\left((\ba,\bh),r^*\right)} \left\Vert \Pi^{K}\left(I - AD\Phi_m(\ta,\th)\right)\Pi^K\right\Vert_{\X_\rho} \label{eq:Z_KK}\\
		&\quad + \sup_{(\ta,\th)\in B_{\X_\rho}\left((\ba,\bh),r^*\right)} \left\Vert \Pi^{K}\left(I - AD\Phi_m(\ta,\th)\right)\Pi^\infty\right\Vert_{\X_\rho} \label{eq:Z_Kinf}\\
		&\quad + \sup_{(\ta,\th)\in B_{\X_\rho}\left((\ba,\bh),r^*\right)} \left\Vert \Pi^{\infty}\left(I - AD\Phi_m(\ta,\th)\right)\right\Vert_{\X_\rho}. \label{eq:Z_inf}
	\end{align}
	\begin{remark}
		In principle, one might optimize the efficiency of the whole procedure by using here and in the definition of $A$ a $K$ which is different (typically smaller) than the $K$ used for obtaining the numerical solution, but we will not do so in this work.
	\end{remark}
	
	\subsubsection{Dealing with~\eqref{eq:Z_KK}}
	
	The bound derived in this section for \eqref{eq:Z_KK} will be denoted by $Z^{K,K}$ when reporting numerical values or in the code. 
	
	For a given $(\ta,\th)\in B_{\X_\rho}\left((\ba,\bh),r^*\right)$ we have that
	\begin{align*}
		\left\Vert \Pi^{K}\left(I - AD\Phi_m(\ta,\th)\right)\Pi^K\right\Vert_{\X_\rho} &= \left\Vert I_K - J\Pi^{K}D\Phi_m(\ta,\th)\Pi^K\right\Vert_{\X_\rho},
	\end{align*}
	where $I_K$ is the identity operator on $\X_\rho^K$. Therefore, we merely have to compute the norm of a finite dimensional operator, the only slight difficulty being than $(\ta,\th)$ are arbitrary elements in $\B_{\X_\rho}\left((\ba,\bh),r^*\right)$. We deal with that by using interval arithmetic. In particular, 
	\begin{align*}
		\ta &\in \ba + [-r^*,r^*],\\
		\forall~x\in[-1,1],\quad \th(x) &\in \bh(x) + [-r^*,r^*],\\
		\forall~x\in[-1,1],\quad \th'(x) &\in \bh'(x) + \sigma^{\even}_{1,\rho}[-r^*,r^*],
	\end{align*}
	the second estimate being a consequence of Lemma~\ref{lem:C0_VS_ell1}, and the last one following from Proposition~\ref{prop:derivative}.
			
		\subsubsection{Dealing with~\eqref{eq:Z_Kinf}}
		
		The bound derived in this section for \eqref{eq:Z_Kinf} will be denoted by $Z^{K,\infty}$ when reporting numerical values or in the code. 
		
		For a given $(\ta,\th)\in B_{\X_\rho}\left((\ba,\bh),r^*\right)$ we have that
		\begin{align*}
			\left\Vert \Pi^{K}\left(I - AD\Phi_m(\ta,\th)\right)\Pi^\infty\right\Vert_{\X_\rho} &= \left\Vert J \Pi^K D\Phi_m(\ta,\th)\Pi^\infty\right\Vert_{\X_\rho} \\
			&= \sup_{\left\Vert h \right\Vert_{\ell^{1,\even}_\rho \leq 1}} \left\Vert J \Pi^K \left( D\Phi_m(\ta,\th)(0,\Pi^\infty h) \right)\right\Vert_{\X_\rho}.
		\end{align*}
		Therefore, according to ~\eqref{eq:dphidh} we have to estimate, for any $h\in\ell^{1,\even}_\rho$ with $\left\Vert h \right\Vert_{\ell^1_\rho} \leq 1$,
		\begin{align*}
			\Pi^K D\Phi_m(\ta,\th)\left(0,\Pi^\infty h\right)&= \begin{pmatrix}
				\left(\Pi^\infty h\right)(0) \\
				\Pi^K\left[\ta \left(\Pi^\infty h\right) - \sum_{j=0}^{m-1}\tg_j(\cdot) \left(\Pi^\infty h\right)(\th^j(\ta \cdot)) \right]
			\end{pmatrix}\\
			&= \begin{pmatrix}
				\left(\Pi^\infty h\right)(0) \\
				-\sum_{j=0}^{m-1} \Pi^K\left[\tg_j(\cdot)\left(\Pi^\infty h\right)(\th^j(\ta \cdot))\right]
			\end{pmatrix}.
		\end{align*}
		First, notice that since $K$ is even, then $0$ is among the Chebyshev nodes, and thus $\left(\Pi^\infty h\right)(0)=0$. 
		
		We then get bounds for (the absolute values of) the Chebyshev coefficients of 
		\begin{equation}
			\label{eq:cheb_for_Z}
			\Pi^K\left[\tg_j(\cdot)\left(\Pi^\infty h\right)(\th^j(\ta \cdot))\right],\quad j=0,\ldots,m-1,
		\end{equation}
		following Remark~\ref{rem:2ways}. We finally add these bounds back together, multiply the output by $\vert J\vert$ and take the $\left\Vert\cdot\right\Vert_{\X_\rho}$ norm to get a bound on~\eqref{eq:Z_Kinf}. 
		
		Let us be more explicit about the way we bound the Chebyshev coefficients of~\eqref{eq:cheb_for_Z}. As explained in Remark~\ref{rem:2ways}, we in fact derive two different estimates and then take the minimum between the two.
		
		First, after having checked that $\th^j(\ta x_k)\in \E_{\beta}$ for all $0\leq k\leq K$ for some $1\leq \beta \leq \rho$, we use Proposition~\ref{prop:interp_alpharho} to get, for all $0\leq k\leq K$,
		\begin{align*}
			\left\vert \tg_j(x_k)\left(\Pi^\infty h\right)(\th^j(\ta x_k))\right\vert &\leq \left\vert \tg_j(x_k)\right\vert\left\Vert \left(\Pi^\infty h\right) \right\Vert_{\CC^0_{\beta}} \\
			&  \leq \left\vert \tg_j(x_k)\right\vert  \Upsilon^{0,1,\even}_{\beta,\rho,K} \left\Vert h \right\Vert_{\ell^1_\rho}.
		\end{align*}
		Multiplying this estimate by $\left\vert M_K^{-1}\right\vert$ we get a bound for the Chebyshev coefficients of~\eqref{eq:cheb_for_Z}:
		\renewcommand{\arraystretch}{1.5}
		\begin{equation}
			\label{eq:ZK1}
			\left\vert \Pi^K D\Phi_m(\ta,\th)\left(0,\Pi^\infty h\right) \right\vert \leq  
			\Upsilon^{0,1,\even}_{\beta,\rho,K} \begin{pmatrix}
				0 \\
				\vert M_K^{-1}\vert \ds \sum_{j=0}^{m-1}
				\begin{pmatrix}
					\left\vert \tg_j(x_0)\right\vert  \\
					\vdots \\
					\left\vert \tg_j(x_k)\right\vert   \\
					\vdots \\
					\left\vert \tg_j(x_K)\right\vert   
				\end{pmatrix} 
			\end{pmatrix}.
		\end{equation}
		\renewcommand{\arraystretch}{1}
		
		A second way of controlling the Chebyshev coefficients of~\eqref{eq:cheb_for_Z} is to consider $\beta_j$ and $\gamma_j$ satisfying
		\begin{equation}
			1\leq \beta_j \leq \rho \leq \gamma_j \qquad\text{and}\qquad \th^j\left(\ta \E_{\gamma_j} \right)\subset\E_{\beta_j},
		\end{equation}
		and then estimate
		\begin{align*}
			\left\Vert \tg_j(\cdot)\left(\Pi^\infty h\right)(\th^j(\ta \cdot))\right\Vert_{\CC^0_{\gamma_j}} &\leq \left\Vert \tg_j\right\Vert_{\CC^0_{\gamma_j}} \left\Vert \Pi^\infty h\right\Vert_{\CC^0_{\beta_j}} \\
			&\leq \left\Vert \tg_j\right\Vert_{\CC^0_{\gamma_j}} \Upsilon^{0,1,\even}_{\beta_j,\rho,K} \left\Vert h \right\Vert_{\ell^1_\rho}.
		\end{align*}
		The second part of Lemma~\ref{lem:decay_fk} then gives us a bound on the Chebyshev coefficients of
		\begin{align*}
			\Pi^K\left[\tg_j(\cdot)\left(\Pi^\infty h\right)(\th^j(\ta \cdot))\right],
		\end{align*}
		which decays like $\gamma_j^{-k}$. Putting all the terms together, we get a second estimate:
		\renewcommand{\arraystretch}{1.5}
		\begin{equation}
			\label{eq:ZK2}
			\left\vert \Pi^K D\Phi_m(\ta,\th)\left(0,\Pi^\infty h\right) \right\vert \leq   
			\begin{pmatrix}
				0 \\
				
				\ds \sum_{j=0}^{m-1} \Upsilon^{0,1,\even}_{\beta_j,\rho,K} \left\Vert \tg_j\right\Vert_{\CC^0_{\gamma_j}}   \begin{pmatrix}
					\frac{\gamma_j^{2K}+1}{\gamma_j^{2K}-1} \\
					\vdots \\
					
					\frac{1}{\gamma_j^k} \frac{\gamma_j^{2K} + \gamma_j^{2k}}{\gamma_j^{2K}-1}  \\
					
					\vdots \\
					\frac{1}{\gamma_j^K} \frac{\gamma_j^{2K}}{\gamma_j^{2K}-1}
				\end{pmatrix}
			\end{pmatrix}
		\end{equation}
		\renewcommand{\arraystretch}{1}
		\begin{remark}
			In practice, we optimize only over the $\gamma_j$, and take 
			\begin{align*}
				\beta_j = \beta_j(\gamma_j) = \frac{1}{2}\left(\eta(\gamma_j)+\sqrt{\eta(\gamma_j)^2-4}\right),
			\end{align*}
			where, according to Lemma~\ref{lem:inclusion_fE}, we take
			\begin{align*}
				\eta(\gamma_j) \geq \max_{z\in \partial\E_{\gamma_j}} \vert \th^j(\ta z)-1\vert + \vert \th^j(\ta z)+1\vert.
			\end{align*}
			This choice ensures that $\th^j\left(\ta\E_{\gamma_j}\right)\subset\E_{\beta_j}$.
			However, for each $j$, we do this optimization independently in each mode. That is, we may in fact select different $\beta_j$ and $\gamma_j$ for each component $k$ of~\eqref{eq:ZK2}.
		\end{remark}
		
		Finally, we take the minimum component-wise between~\eqref{eq:ZK1} and~\eqref{eq:ZK2}, multiply the result by $\vert J\vert$ and take the $\X_\rho$ norm to get a bound on~\eqref{eq:Z_Kinf}.
		
		\subsubsection{Dealing with~\eqref{eq:Z_inf}}

		Similarly to the previous two sections, the bound derived for \eqref{eq:Z_inf} will be denoted by $Z^{\infty}$ when reporting numerical values or in the code.

		For a given $(\ta,\th)\in B_{\X_\rho}\left((\ba,\bh),r^*\right)$ we have that
			
		\begin{align*}
			\left\Vert \Pi^{\infty}\left(I - AD\Phi_m(\ta,\th)\right)\right\Vert_{\X_\rho} 
			&= \sup_{\left\Vert (\alpha,h) \right\Vert_{\X_\rho} \leq 1} \left\Vert \Pi^{\infty}\left( I - AD\Phi_m(\ta,\th)\right) (\alpha,h)\right\Vert_{\X_\rho} .
		\end{align*}
		Therefore, we have to estimate, for $\left\Vert (\alpha,h) \right\Vert_{\X_\rho} \leq 1$,
		
		\begin{align*}
			\left\Vert \Pi^{\infty}\left((\alpha,h) - AD\Phi_m(\ta,\th)(\alpha,h)\right)\right\Vert_{\X_\rho} &= \left\Vert \Pi^{\infty}h -  \frac{1}{\ba}\Pi^{\infty}D\phi(\ta,\th)(\alpha,h)\right\Vert_{\ell^1_\rho}.
		\end{align*}
		According to~\eqref{eq:dphidalpha}-\eqref{eq:dphidh}, we have
		\begin{align*}			\Pi^{\infty}\left(h - \frac{1}{\ba}D\phi(\ta,\th)(\alpha,h)\right) &= \left(1-\frac{\ta}{\ba}\right) \Pi^\infty h \\			&\quad + \frac{1}{\ba}\sum_{j=0}^{m-1}\Pi^{\infty}\left( \tg_j(\cdot)\, h\left((\th^j(\ta \cdot)\right)  \right)\\			&\quad- \frac{\alpha}{\ba}\Pi^{\infty}(\th-\tf),		\end{align*}
		and we bound each term independently in the r.h.s. below
		\begin{align*}
			\left\Vert \Pi^{\infty}\left(h - \frac{1}{\ba}D\phi(\ta,\th)(\alpha,h)\right)\right\Vert_{\ell^1_\rho} & \leq \frac{r^*}{\vert\ba\vert} \left\Vert \Pi^\infty h\right\Vert_{\ell^1_\rho} \\
			&\quad + \frac{1}{\vert\ba\vert}\sum_{j=0}^{m-1}\left\Vert\Pi^{\infty}\left(\tg_j(\cdot)\, h\left((\th^j(\ta \cdot)\right) \right) \right\Vert_{\ell^1_\rho}\\
			&\quad + \frac{\vert\alpha\vert}{\vert\ba\vert}\left\Vert\Pi^{\infty}\left(\tf - \th\right)\right\Vert_{\ell^1_\rho}.
		\end{align*}
		The last term is very similar to the tail part of the $Y$ bound. Indeed, for any $\nu>\rho$ we have
		\begin{equation*}
			\frac{\vert\alpha\vert}{\vert\ba\vert}\left\Vert\Pi^{\infty}\left(\tf -  \th\right)\right\Vert_{\ell^1_\rho} \leq \frac{\vert\alpha\vert}{\vert\ba\vert} \left( \Upsilon^{1,0,\even}_{\rho,\nu,K}\left\Vert \tf\right\Vert_{\CC^0_{\nu}} + r^* \right).
		\end{equation*}
		In order to bound the remaining terms, we take, for $j=0,\ldots,m-1$, $\mu_j>\rho$ such that
		\begin{align*} \th^j(\ta\E_{\mu_j}) \subset \E_\rho
		\end{align*}
		and then use Proposition~\ref{prop:interp_alpharho0} to estimate
		\begin{align*}
			\frac{1}{\vert\ba\vert}\sum_{j=0}^{m-1}\left\Vert\Pi^{\infty}\left(\tg_j(\cdot)\, h\left((\th^j(\ta \cdot)\right) \right) \right\Vert_{\ell^1_\rho} &\leq \frac{1}{\vert\ba\vert}\sum_{j=0}^{m-1}\Upsilon^{1,0,\even}_{\rho,\mu_j,K}\left\Vert \tg_j(\cdot)\, h\left((\th^j(\ta \cdot)\right) \right\Vert_{\CC^0_{\mu_j}} \\
			&\leq \frac{1}{\vert\ba\vert}\left( \sum_{j=0}^{m-1}\Upsilon^{1,0,\even}_{\rho,\mu_j,K}\left\Vert \tg_j \right\Vert_{\CC^0_{\mu_j}}\right) \left\Vert h \right\Vert_{\ell^1_\rho}.
		\end{align*}
		Putting everything together, we have
		\begin{align*}
			\left\Vert \Pi^{\infty}\left(h - \frac{1}{\ba}D\phi(\ta,\th)(\alpha,h)\right)\right\Vert_{\ell^1_\rho} &\leq \frac{1}{\vert\ba\vert}\left(\Upsilon^{1,0,\even}_{\rho,\nu,K}\left\Vert \tf \right\Vert_{\CC^0_{\nu}} + r^*\right)\vert\alpha\vert\\
			&\quad + \frac{1}{\vert\alpha\vert}\left(\sum_{j=0}^{m-1}\Upsilon^{1,0,\even}_{\rho,\mu_j,K}\left\Vert \tg_j \right\Vert_{\CC^0_{\mu_j}} + r^*\right)\left\Vert h \right\Vert_{\ell^1_\rho}\\
			&\leq \frac{1}{\vert\alpha\vert}\max\left[\Upsilon^{1,0,\even}_{\rho,\nu,K}\left\Vert \tf \right\Vert_{\CC^0_{\nu}}+ r^*, \right. \\
			&\qquad\qquad\quad \left. \sum_{j=0}^{m-1}\Upsilon^{1,0,\even}_{\rho,\mu_j,K}\left\Vert \tg_j \right\Vert_{\CC^0_{\mu_j}} + r^*\right] \left(\vert\alpha\vert + \left\Vert h \right\Vert_{\ell^1_\rho}\right) \\
			&\leq \frac{1}{\vert\alpha\vert}\left( \max\left[\Upsilon^{1,0,\even}_{\rho,\nu,K}\left\Vert \tf \right\Vert_{\CC^0_{\nu}},\ \sum_{j=0}^{m-1}\Upsilon^{1,0,\even}_{\rho,\mu_j,K}\left\Vert \tg_j \right\Vert_{\CC^0_{\mu_j}}\right] + r^* \right),
		\end{align*}
		which upper-bounds~\eqref{eq:Z_inf}.

		\section{Eigenvalue}\label{sec:eigenvalue}
		
		In this section, we assume we have obtained a fixed point $f$ of $R_m$, using the computer-assisted proof described in Section~\ref{sec:fixed_point}. Examples of such results are provided in Section~\ref{sec:results}. Our goal is now to get a rigorous enclosure of the associated universal constant $\lambda$, i.e., the unstable eigenvalue of $DR_m(f)$. To that end, we abandon the intermediate problem $\Phi_m$ that was used to obtain the fixed point, and return to the standard representation of $R_m(f)$. Nonetheless, the Jacobian $DR_m$ contains terms that were already present in $D\Phi_m$, and in practice we are therefore able to take advantage of some of the computations already performed. 		
		We first express $R_m(f)$ in terms of $\tR_m(f^m(0), f)$ (recall equation~\eqref{eq:tR_m}).

We have that 		
		\begin{align*}
			R_m(f)(x) = \frac{1}{\alpha(f)}f^m\left(\alpha(f)x\right) = \frac{1}{\alpha(f)}\tR_m\left(\alpha(f),f\right)(x) \qquad \text{with} \qquad \alpha(f)=f^m(0) .
		\end{align*}
Then,
		\begin{align*}
			DR_m(f)(h)(x) &= \frac{-1}{{(\alpha(f))}^2} \partial_f\tR_m(\alpha(f),f)(h)(0) \times f^m(\alpha(f)x) \\
			&\quad + \frac{1}{\alpha(f)}\left( \partial_f\tR_m(\alpha(f),f)(h)(x) 
			+ (f^m)^{'} (\alpha(f)x) \times \partial_f\tR_m(\alpha(f),f)(h)(0) x \right),
		\end{align*}	

		and more explicitly, 		
		\begin{align*}
			DR_m(f)(h)(x) &= \frac{1}{f^m(0)} \sum_{j=0}^{m-1} \left( \prod_{l=j+1}^{m-1} f'\left(f^l\left(f^m(0)x\right)\right)\right)h\left(f^j\left(f^m(0)x\right)\right) \\
			&\quad + \left( \frac{1}{f^m(0)}x\prod_{j=0}^{m-1} f'\left(f^j\left(f^m(0)x\right)\right) - \frac{1}{\left(f^m(0)\right)^2}f^m\left(f^m(0)x\right) \right) \\
			&\quad \times \sum_{j=0}^{m-1} \left( \prod_{l=j+1}^{m-1} f'\left(f^l\left(0\right)\right)\right)h\left(f^j\left(0\right)\right) \\
			&= \frac{1}{\delta_m(0)} \left( \sum_{j=0}^{m-1} \xi_{j+1}(x) h\left(\delta_j(x)\right) + \left( x\xi_0(x) - \frac{\delta_m(x)}{\delta_m(0)} \right)\sum_{j=0}^{m-1} \xi_{j+1}(0)h\left(\delta_j(0)\right) \right),
		\end{align*}
		where
		\begin{align*}
			\delta_j(x) = f^j\left(f^m(0)x\right),\quad \xi_j(x) = \prod_{l=j}^{m-1} f'\left(\delta_l(x)\right).
		\end{align*}
		Let $f$ be a fixed point of $R_m$. We look for $\lambda\in\C$ and $u\in\ell^{1,\even}_\rho$ such that $(\lambda,u)$ is a zero of
		\begin{align*}
			F(\lambda,u) = \begin{pmatrix}
				u_0 -1 \\
				DR_m(f)u - \lambda u
			\end{pmatrix},
		\end{align*}
		where $u_0$ is the zero-th Chebyshev coefficient of $u$.
		\begin{remark}
			Some normalization is needed to ensure that the eigenpair we try to validate is isolated, but the specific choice of enforcing $u_0=1$ is somewhat arbitrary.
		\end{remark}
Similarly to what we did in Section~\ref{sec:fixed_point}, we reformulate the zero-finding problem into a fixed-point problem in order to validate a posteriori an approximate eigenpair.
We have,
		\begin{align*}
			DF(\bl,\bu) = \begin{pmatrix}
				0 & E_0 \\
				-\bu & DR_m(f)-\bl I   
			\end{pmatrix},
		\end{align*}
		where $E_0$ is the map $u\mapsto u_0$. Since $DR_m(f)$ is compact, we take $A$ as
		\begin{equation*}
			\left\{
			\begin{aligned}
				&A\, \Pi^K (\lambda,u) = J\, \Pi^K(\lambda,u) \\
				&A\, (I-\Pi^K) (\lambda,u) = \left(0,\, -\frac{1}{\bl}\,  (I-\Pi^K)u \right),
			\end{aligned}
			\right.
		\end{equation*}
		where $J$ is an approximate inverse of
		\begin{equation*}
			J^\dag= \Pi^K \restriction{DF(\bl,\bu)}{\X_\rho^K},
		\end{equation*}
		This leads to the fixed-point operator
		\begin{align*}
			T: (\lambda,u)\mapsto (\lambda,u) - A\, F(\lambda,u).
		\end{align*}
		We again will again use Theorem~\ref{th:fixed_point}, but this time since $F$ is merely quadratic in $(\lambda,u)$ we split the $Z$ estimate into a $Z_1$ and $Z_2$ part as explained in Remark~\ref{rem:fixed_point}. Specifically, we write, for an arbitrary $(\lambda,u)$ in $\X_\rho$: 
		
		\begin{align*}
			\left\Vert DT(\lambda,u) \right\Vert_{\X_\rho} &\leq 	 \left\Vert DT(\bl,\bu) \right\Vert_{\X_\rho} + \left\Vert DT(\lambda,u) - DT(\bl,\bu) \right\Vert_{\X_\rho} \\
			&\leq Z_1 + Z_2 r . 
		\end{align*}

		We derive suitable estimates $Y$, $Z_1$ and $Z_2$ below. Many of the calculations are very similar to the ones of Section~\ref{sec:bounds}, and we do not repeat all the details.
		
		\subsection[Y estimate in \eqref{eq:Y}]{$\bm{Y}$ estimate in \eqref{eq:Y}}

		The finite part is simply
		\begin{equation*}
			Y^K=\left\Vert  J\, \Pi^K F(\bl,\bu) \right\Vert_{\X_\rho}.
		\end{equation*}
		We estimate the coefficients of $\Pi^K F(\bl,\bu)$ using Remark~\ref{rem:2ways}, and then simply multiply the result by $J$ and compute the $\ell^1_\rho$ norm of the result.

		For the tail part, we use Proposition~\ref{prop:interp_alpharho0} to estimate
		\begin{align*}
			\frac{1}{\vert\bl\vert}\left\Vert  \left(I-\Pi^K\right) F(\bl,\bu) \right\Vert_{\ell^1_\rho} = \frac{1}{\vert\bl\vert}\left\Vert  \left(I-\Pi^K\right) \left( DR_m(f)\bu - \bl\bu\right) \right\Vert_{\ell^1_\rho}.
		\end{align*}
We then proceed is in Section~\ref{sec:Y}, and take
\begin{align*}
Y^{\infty} = \frac{1}{\vert\bl\vert}\left\Vert  \left(\Pi^{K_Y}-\Pi^K\right) \left( DR_m(f)\bu - \bl\bu\right)\right\Vert_{\ell^1_\rho}  + \frac{\Upsilon^{1,0,\even}_{\rho,\nu,K_Y}}{\vert\bl\vert}\left\Vert DR_m(f)\bu\right\Vert_{\CC^0_\nu},
\end{align*}
for some $K_Y>K$ and some $\nu>\rho$ chosen according to Remark~\ref{rem:opt}.

		\subsection[$Z_1$ estimate in Remark \ref{rem:fixed_point}]{$\bm{Z_1}$ estimate in Remark \ref{rem:fixed_point}}
		
		In this section, we estimate
		\begin{align}
			\left\Vert DT(\bl,\bu) \right\Vert_{\X_\rho} &= \left\Vert I - AD F(\bl,\bu)\right\Vert_{\X_\rho} \nonumber\\
			&\leq \left\Vert \Pi^{K}\left(I - AD F(\bl,\bu)\right)\Pi^K\right\Vert_{\X_\rho} \label{eq:Z1_KK}\\
			&\quad +  \left\Vert \Pi^{K}\left(I - AD F(\bl,\bu)\right)\Pi^\infty\right\Vert_{\X_\rho} \label{eq:Z1_Kinf}\\
			&\quad + \left\Vert \Pi^{\infty}\left(I - AD F(\bl,\bu)\right)\right\Vert_{\X_\rho}. \label{eq:Z1_inf}
		\end{align}

		\subsubsection{Dealing with~\eqref{eq:Z1_KK}}
		
		We have that
		\begin{align*}
			\left\Vert \Pi^{K}\left(I - AD F(\bl,\bu)\right)\Pi^K\right\Vert_{\X_\rho} &= \left\Vert I_K - J\Pi^{K}D F(\bl,\bu)\Pi^K\right\Vert_{\X_\rho},
		\end{align*}
		where $I_K$ is the identity operator on $\X_\rho^K$. Therefore, we merely have to compute the norm of a finite dimensional operator. 
 This bound will be denoted by $Z_1^{K,K}$ when reporting numerical values or in the code, similar to the case of the fixed point. 
		
		\subsubsection{Dealing with~\eqref{eq:Z1_Kinf}}
		
		We have that
		\begin{align*}
			\left\Vert \Pi^{K}\left(I - AD F(\bl,\bu)\right)\Pi^\infty\right\Vert_{\X_\rho} &= \left\Vert J \Pi^K D F(\bl,\bu)\Pi^\infty\right\Vert_{\X_\rho} \\
			&= \sup_{\left\Vert h \right\Vert_{\ell^{1,\even}_\rho \leq 1}} \left\Vert J \Pi^K \left( D F(\bl,\bu)(0,\Pi^\infty h) \right)\right\Vert_{\X_\rho}.
		\end{align*}
		Therefore, we have to estimate, for any $h\in\ell^{1,\even}_\rho$ with $\left\Vert h \right\Vert_{\ell^1_\rho} \leq 1$,
		\begin{align*}
			\Pi^K D F(\bl,\bu)\left(0,\Pi^\infty h\right)&= \begin{pmatrix}
				E_0\left(\Pi^\infty h\right) \\
				\Pi^K\left[DR_m(f)\Pi^\infty h - \bl\Pi^\infty h \right]
			\end{pmatrix}\\
			&= \begin{pmatrix}
				E_0\left(\Pi^\infty h\right) \\
				\Pi^K\left[DR_m(f)\Pi^\infty h\right]
			\end{pmatrix}.
		\end{align*}
		The first component is easy to bound. Indeed, according to Lemma~\ref{lem:aliasing} we have
		\begin{align*}
			\left\vert E_0\left(\Pi^\infty h\right) \right\vert &= \left\vert h_0 - \check{h}_0 \right\vert 
			= \left\vert 2\sum_{l=1}^\infty h_{2Kl} \right\vert 
			\leq \frac{1}{\rho^{2K}} \left\Vert h\right\Vert_{\ell^1_\rho}.
		\end{align*}
		We then get bounds for (the absolute values of) the Chebyshev coefficients of $\Pi^K\left[DR_m(f)\Pi^\infty h\right]$, following Remark~\ref{rem:2ways}. That is, we in fact derive two different estimates:
		\renewcommand{\arraystretch}{1.5}
		\begin{equation*}
			\left\vert \Pi^K D F(\bl,\bu)\left(0,\Pi^\infty h\right) \right\vert \leq  
			\begin{pmatrix}
				\rho^{-2K} \\
				\frac{\Upsilon^{0,1,\even}_{1,\rho,K}}{\vert \delta_m(0)\vert}\vert M_K^{-1}\vert \ds \sum_{j=0}^{m-1}
				\begin{pmatrix}
					\left\vert \xi_{j+1}(x_0)\right\vert + \left\vert \left(x_0\xi_0(x_0)-\frac{\delta_m(x_0)}{\delta_m(0)}\right) \xi_{j+1}(0)\right\vert \\
					\vdots \\
					\left\vert \xi_{j+1}(x_k)\right\vert + \left\vert \left(x_k\xi_0(x_k)-\frac{\delta_m(x_k)}{\delta_m(0)}\right) \xi_{j+1}(0)\right\vert  \\
					\vdots \\
					\left\vert \xi_{j+1}(x_K)\right\vert + \left\vert \left(x_K\xi_0(x_K)-\frac{\delta_m(x_K)}{\delta_m(0)}\right) \xi_{j+1}(0)\right\vert  
				\end{pmatrix} 
			\end{pmatrix},
		\end{equation*}
		and
		\begin{align*}
			&\left\vert \Pi^K D F(\bl,\bu)\left(0,\Pi^\infty h\right) \right\vert \leq \\
			&  
			\begin{pmatrix}
				\rho^{-2K} \\
				\ds \frac{1}{\vert \delta_m(0)\vert}\sum_{j=0}^{m-1} \left(\Upsilon^{0,1,\even}_{\beta_j,\rho,K} \left\Vert \xi_{j+1}\right\Vert_{\CC^0_{\gamma_j}}   \begin{pmatrix}
					\frac{\gamma_j^{2K}+1}{\gamma_j^{2K}-1} \\
					\vdots \\
					\frac{1}{\gamma_j^k} \frac{\gamma_j^{2K} + \gamma_j^{2k}}{\gamma_j^{2K}-1}  \\
					\vdots \\
					\frac{1}{\gamma_j^K} \frac{\gamma_j^{2K}}{\gamma_j^{2K}-1}
				\end{pmatrix}
				+
				\left\vert\xi_{j+1}(0)\right\vert \Upsilon^{0,1,\even}_{1,\rho,K} \left\Vert \cdot\xi_0 - \frac{\delta_m}{\delta_m(0)}\right\Vert_{\CC^0_{\gamma}}    \begin{pmatrix}
					\frac{\gamma^{2K}+1}{\gamma^{2K}-1} \\
					\vdots \\
					\frac{1}{\gamma^k} \frac{\gamma^{2K} + \gamma^{2k}}{\gamma^{2K}-1}  \\
					\vdots \\
					\frac{1}{\gamma^K} \frac{\gamma^{2K}}{\gamma^{2K}-1}
				\end{pmatrix}
				\right)
			\end{pmatrix} 
			\renewcommand{\arraystretch}{1}
		\end{align*}
take the minimum component-wise between~the two, and finally multiply the result by $\vert J\vert$ and take the $\X_\rho$ norm to get a bound on~\eqref{eq:Z1_Kinf}. This bound will be denoted by $Z_1^{K,\infty}$ when reporting numerical values or in the code. 
		
				\subsubsection{Dealing with~\eqref{eq:Z1_inf}}
				
				We have that
				
				\begin{align*}
					\left\Vert \Pi^{\infty}\left(I - AD F(\bl,\bu)\right)\right\Vert_{\X_\rho} 
					&= \sup_{\left\Vert (\lambda,h) \right\Vert_{\X_\rho} \leq 1} \left\Vert \left((0,\Pi^{\infty}h) + A \Pi^{\infty}D F(\bl,\bu)(\lambda,h) \right)\right\Vert_{\X_\rho} .
				\end{align*}
				Therefore, we have to estimate, for $\left\Vert (\lambda,h) \right\Vert_{\X_\rho} \leq 1$,
				
				\begin{align*}
					\left\Vert \Pi^{\infty}\left(h + \frac{1}{\bl}\left(DR_m(f)h-\bl h\right)\right)\right\Vert_{\ell^1_\rho} = \frac{1}{\vert \bl\vert} \left\Vert \Pi^{\infty}\left(DR_m(f)h\right)\right\Vert_{\ell^1_\rho}.
				\end{align*}
				In order to bound the first terms appearing in $DR_m(f)h$, we take for $j=0,\ldots,m-1$, $\mu_j>\rho$ such that
				\begin{align*} \delta_j(\E_{\mu_j}) \subset \E_\rho,
				\end{align*}
				and then use Proposition~\ref{prop:interp_alpharho0} to estimate
				\begin{align*}
					\sum_{j=0}^{m-1}\left\Vert\Pi^{\infty}\left(\xi_j\, h\circ \delta_j\right) \right\Vert_{\ell^1_\rho} &\leq \sum_{j=0}^{m-1}\Upsilon^{1,0,\even}_{\rho,\mu_j,K}\left\Vert \xi_j\, h\circ \delta_j \right\Vert_{\CC^0_{\mu_j}} \\
					&\leq \sum_{j=0}^{m-1}\Upsilon^{1,0,\even}_{\rho,\mu_j,K}\left\Vert \xi_j \right\Vert_{\CC^0_{\mu_j}}\left\Vert h \right\Vert_{\ell^1_\rho}.
				\end{align*}
				The other terms are can be dealt with as in the $Y$ bound, i.e.
				\begin{align*}
					\left\Vert\Pi^{\infty}\left( \cdot\xi_0 - \frac{\delta_m}{\delta_m(0)} \right)\right\Vert_{\ell^1_\rho} \leq \Upsilon^{1,0,\even}_{\rho,\nu,K} \left\Vert\cdot\xi_0 - \frac{\delta_m}{\delta_m(0)} \right\Vert_{\CC^0_\nu},
				\end{align*}
				for any $\nu>\rho$.
				Putting everything together, we have
				\begin{align*}
					\frac{1}{\vert \bl\vert} \left\Vert \Pi^{\infty}\left(DR_m(f)h\right)\right\Vert_{\ell^1_\rho} &\leq  \frac{1}{\vert\bl\delta_m(0)\vert}\left(\sum_{j=0}^{m-1}\Upsilon^{1,0,\even}_{\rho,\mu_j,K}\left\Vert \xi_j \right\Vert_{\CC^0_{\mu_j}} + \Upsilon^{1,0,\even}_{\rho,\nu,K} \left\Vert\cdot\xi_0 - \frac{\delta_m}{\delta_m(0)} \right\Vert_{\CC^0_\nu}\sum_{j=0}^{m-1}\left\vert \xi_{j+1}(0)\right\vert \right)\left\Vert h \right\Vert_{\ell^1_\rho}\\
					&\leq \frac{1}{\vert\bl\delta_m(0)\vert}\left(\sum_{j=0}^{m-1}\Upsilon^{1,0,\even}_{\rho,\mu_j,K}\left\Vert \xi_j \right\Vert_{\CC^0_{\mu_j}} + \Upsilon^{1,0,\even}_{\rho,\nu,K} \left\Vert\cdot\xi_0 - \frac{\delta_m}{\delta_m(0)} \right\Vert_{\CC^0_\nu}\sum_{j=0}^{m-1}\left\vert \xi_{j+1}(0)\right\vert \right),
				\end{align*}
				which upper-bounds~\eqref{eq:Z1_inf}. This bound will be denoted by $Z_1^\infty$ when reporting numerical values or in the code. 
				
				\subsection[$Z_2$ estimate in Remark \ref{rem:fixed_point}]{$\bm{Z_2}$ estimate in Remark \ref{rem:fixed_point}}
				
				Noticing that $DF$ is linear, we simply have, for any $(\lambda,u)$ in $\X_\rho$,				
				\begin{align*}
					\left\Vert DT(\lambda,u) - DT(\bl,\bu) \right\Vert_{\X_\rho} &= \left\Vert  A \left(DF(\lambda,u) - DF(\bl,\bu) \right)\right\Vert_{\X_\rho}\\
					&\leq 
					\left\Vert 
					A\begin{pmatrix}
						0 & 0 \\
						-(u - \bu) & -(\lambda - \bl) I   
					\end{pmatrix} \right\Vert_{\X_\rho}\\	
					&\leq \left\Vert 
					A\right\Vert_{\X_\rho} \left\Vert (\lambda,u) - (\bl,\bu)\right\Vert_{\X_\rho},
				\end{align*}
				hence we can take
				$Z_2 = \left\Vert 
				A\right\Vert_{\X_\rho}$.
				
				
				\section{Results} \label{sec:results}
				
In this section, we first give the proofs of Theorem~\ref{thm:main}, and then present two additional results. The first one is a 500-digits accurate proof for the  classical $m=2$ case, and the second one provides extra fixed points for $m$ between 5 and 10. Finally, we present the proof of Theorem~\ref{thm:d2}.

\medskip

\noindent\textit{Proof of Theorem~\ref{thm:main}.}
Fix $m$ in $\{2,3,\ldots,10\}$, and consider the associated map $\Phi_m$ defined in equation~\eqref{eq:Phi_m}. Let $(\ba,\bh)$ be the approximate zero of $\Phi_m$ for the selected value of $m$, stored in the folder \texttt{renor\_code\_submit} (note that $\bh$ is a polynomial, represented in Figure~\ref{fig:theFixedPoints}).

Let $\rho=2$ and $r^*$ be as in Table~\ref{table:rstarandbounds}. Using \texttt{final\_script\_gen\_m\_d.jl}, we then evaluate the bounds $Y$ and $Z$ derived in Section~\ref{sec:bounds}, which satisfy assumption~\eqref{eq:YZ} of Theorem~\ref{th:fixed_point}. These bounds are evaluated using interval arithmetic via the IntervalArtithmetic.jl library~\cite{IntervalArithmeticJ}, to account for rounding errors. We then check that $Z<1$ and $r:=\frac{Y}{1-Z}\leq r^*$. Theorem~\ref{th:fixed_point} then yields the existence of a unique zero $(\alpha_*,h_*)$ of $\Phi_m$ in $\X_\rho$ such that $\Vert (\alpha_*,h_*) - (\ba,\bh) \Vert_{\X_\rho} \leq r$. Since $(\alpha_*,h_*)$ is a zero of $\Phi_m$, $h_*$ is a fixed point of $R_m$, and $h_*^m(0)=\alpha_*$. These $\bh$, $h_*$, $\alpha_*$ and $r$ are the $\bar{f}_m$, $f_m$, $\alpha_m$ and $\epsilon_m$ of Theorem~\ref{thm:main}. 

Next, we consider the approximate eigenpair $(\bl,\bu)$ of $DR_m(f_m)$ stored in the folder\linebreak \texttt{renor\_code\_submit}. Still using \texttt{final\_script\_gen\_m\_d.jl}, we then evaluate the bounds $Y$ and $Z$ (split as $Z_1$ and $Z_2$) derived in Section~\ref{sec:eigenvalue}, and apply once more Theorem~\ref{th:fixed_point}. Denoting by $\tilde{r}$ the error bound obtained this time, we get the existence of a unique eigenpair $(\lambda_m,u_m)$ of $DR_m(f_m)$ in $\X_\rho$ such that $\Vert (\lambda,u_m) - (\bl,\bu) \Vert_{\X_\rho} \leq \tilde{r}$. In particular, $\vert \lambda_m - \bl\vert \leq \tilde{r}$, which provides the enclosure of $\lambda_m$ reported in Table~\ref{table:1}. \hfill \qed 


\begin{remark}
Further details about the proof, such as the obtained values for each part of the bounds, are provided in Appendix~\ref{sec:bounds_results}. 
The entire proof runs for around 0.5 minutes for $m=2$ and 10 minutes for $m=10$ on an Intel(R) Core(TM) i7-1065G7 CPU @ 1.30GHz 1.50GHz processor with 16GB of ram.
\end{remark}

\begin{theorem}
\label{thm:m2}
For $m=2$ and the fixed point $f_2$ of $R_2$ obtained in Theorem~\ref{thm:main}, the universal constant $\lambda_2$ and $\alpha_2=f_2(f_2(0))=f_2(1)$ satisfy

$
\longformula {\lambda_2 \in 4.66920160910299067185320382046620161725818557747576863274565134300413433021131473713868974402394801381716598485518981513440862714202793252231244298889089085994493546323671341153248171421994745564436582379320200956105833057545861765222207038541064674949428498145339172620056875566595233987560382563722564800409510712838906118447027758542854198011134401750024285853824983357155220522360872502916788603626745272133990571316068753450834339344461037063094520191158769724322735898389037949462572512890979489\pm 10^{-500}},
$

$
\longformula{\alpha_2 \in -0.3995352805231344898575804686336937194335442804669527275170730449124380166088380429818445948741812667617940648468383667140945404846164364373609475570184545976789402326870225485797735028209746477510392557978775073697474932326975513734923082122088541722241308330948027391890574703944646041606699384157782298900077729901354421213971924552385259444903372376975537750905488329754433672693681140505788840461793440186571478080760841608146499827233996549139348743626575822619683926231334765669900667138742123579\pm 10^{-502}}.
$
\end{theorem}
\begin{proof}
The proof proceeds as the one of Theorem~\ref{thm:main}, except we use extended precision ($2^{12}$ digits) and a polynomial approximation of higher order ($K=680$) in order to compute a much finer approximate solution and then to evaluate the bounds, together with a much smaller $r^*$ (equal to $10^{-470}$). The entire proof can be reproduced by running \texttt{final\_script\_m2\_long-Copy\_final.jl}.
\end{proof}
\begin{remark}
Further details regarding the different bounds can also be found in Appendix~\ref{sec:bounds_results}. 
\end{remark}

Finally, we compute several distinct fixed points for $m$ between $5$ and $10$, along with their associated universal constants.

\begin{figure}[h!]
	
	\includegraphics[scale=0.5]{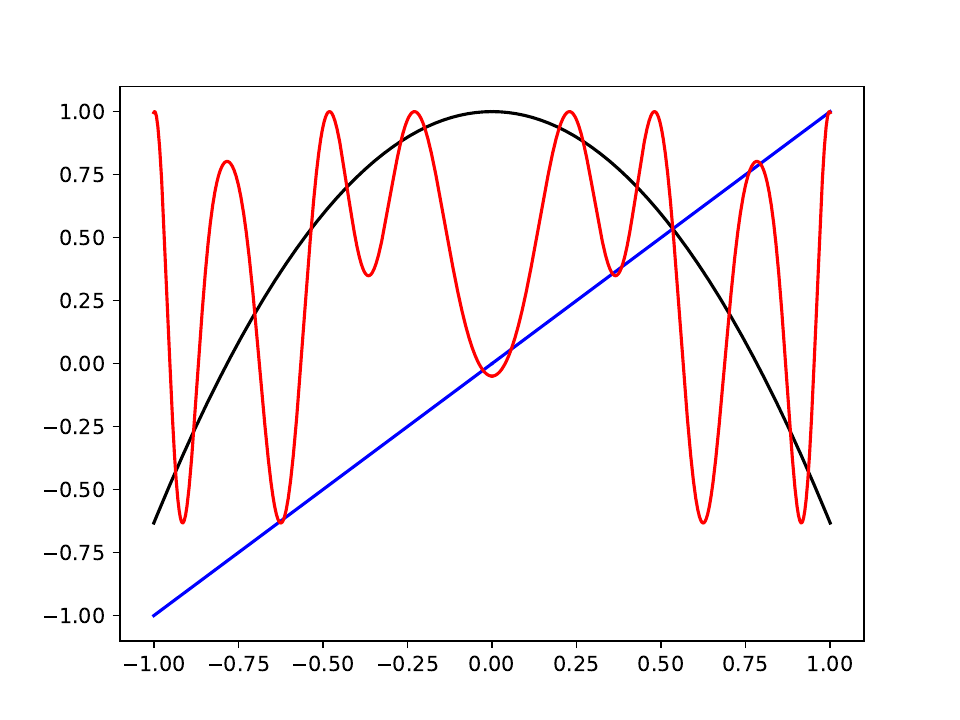}\hfill	
	\includegraphics[scale=0.5]{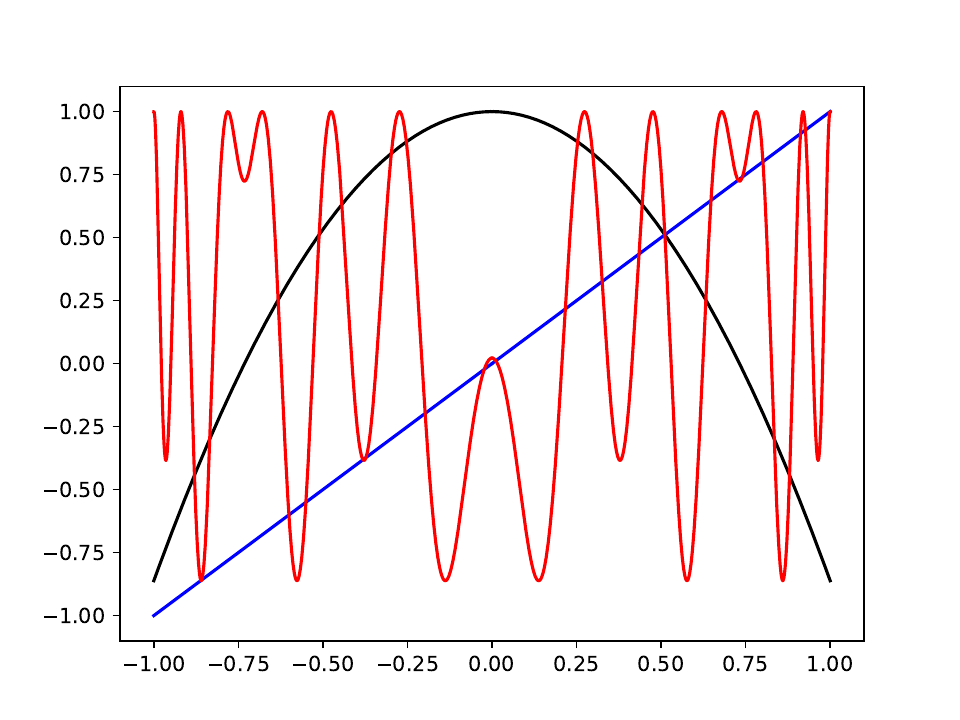}\hfill		
	\\
	\includegraphics[scale=0.5]{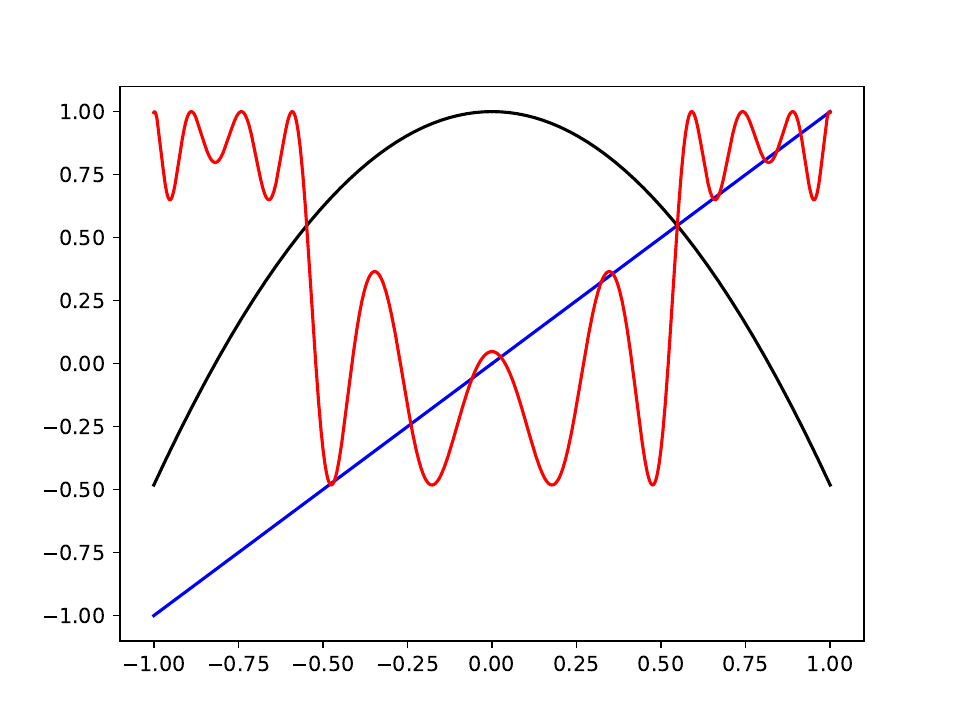}\hfill	
	\includegraphics[scale=0.5]{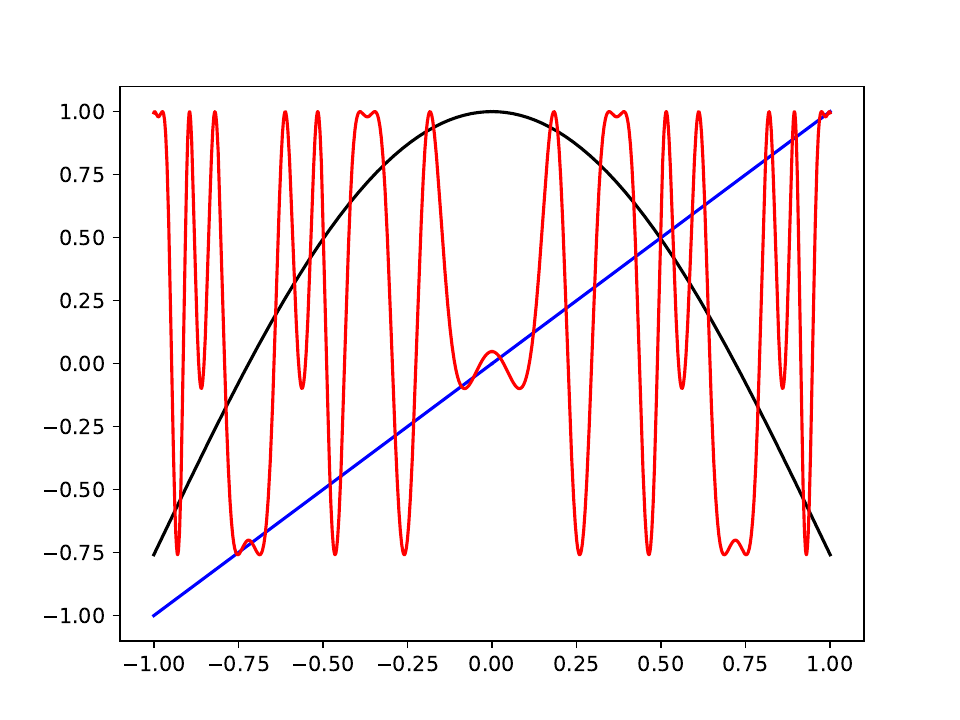}\hfill	
	\\	
	\includegraphics[scale=0.5]{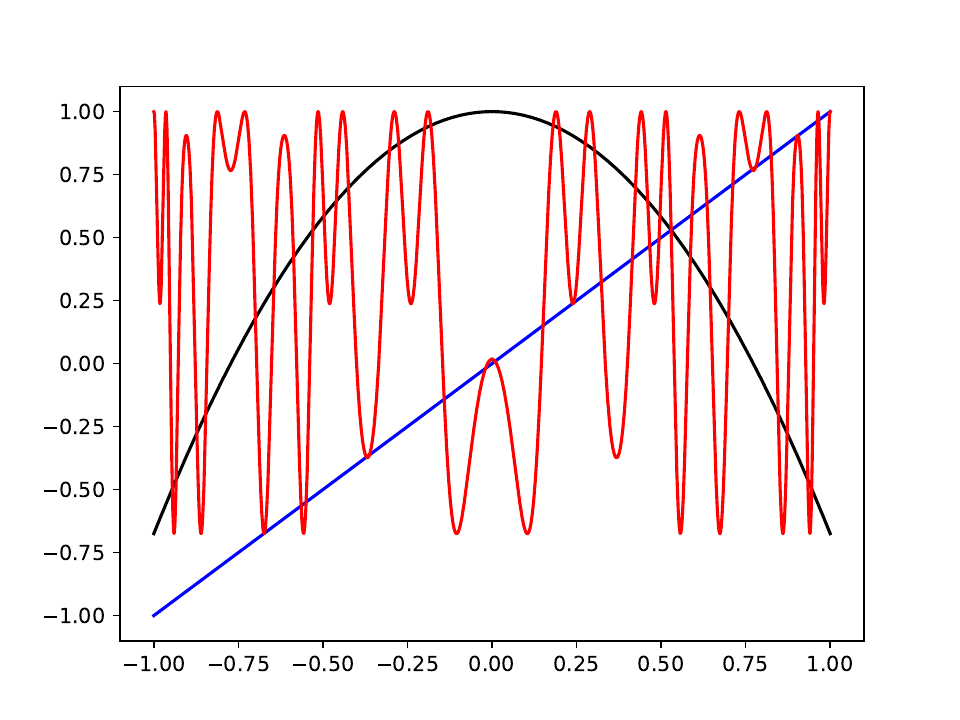}\hfill	
	\includegraphics[scale=0.5]{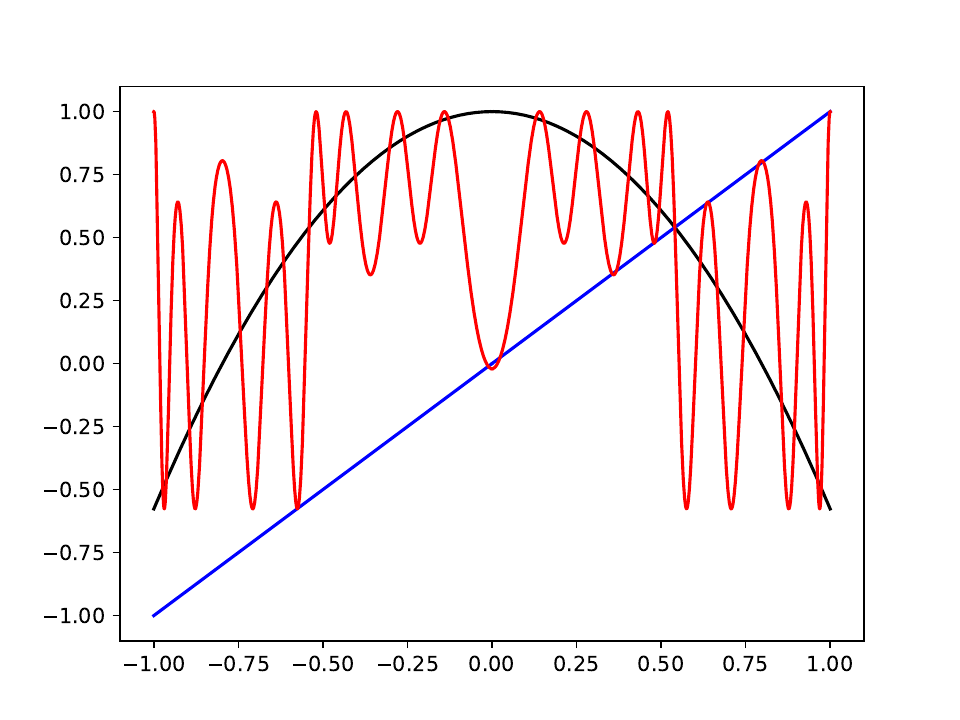}\hfill		
	\caption{Distinct fixed points with $m=5$, $m=6$ and $m=7$. On the left are the fixed points from Theorem~\ref{thm:main} for these values of $m$, already depicted on Figure~\ref{fig:theFixedPoints} and reproduced here for comparison. On the right are the $m$v2 fixed points from Theorem~\ref{thm:extra}. The black curves show the graph of the 
		polynomials $\bar{f}_m$ on $[-1,1]$, 
		the red curve is the composition of $\bar{f}_m$ with itself 
		$m$ times without rescaling, and the blue line is the fixed point line $y=x$.}
\label{fig:5to7}	
\end{figure}

\begin{figure}[h!]
	\includegraphics[scale=0.5]{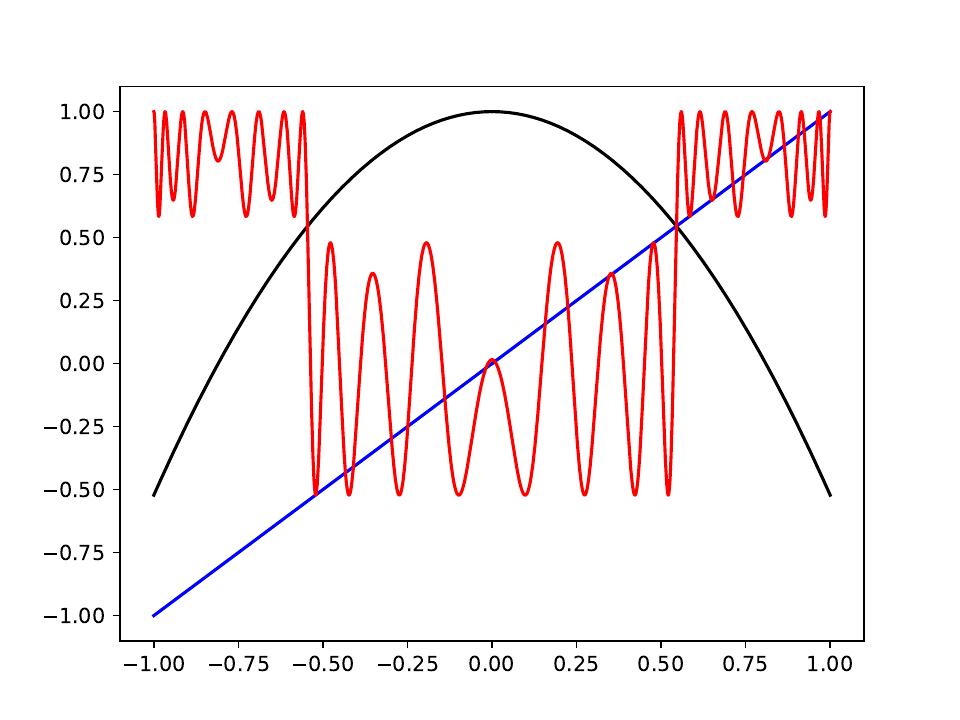}\hfill	
	\includegraphics[scale=0.5]{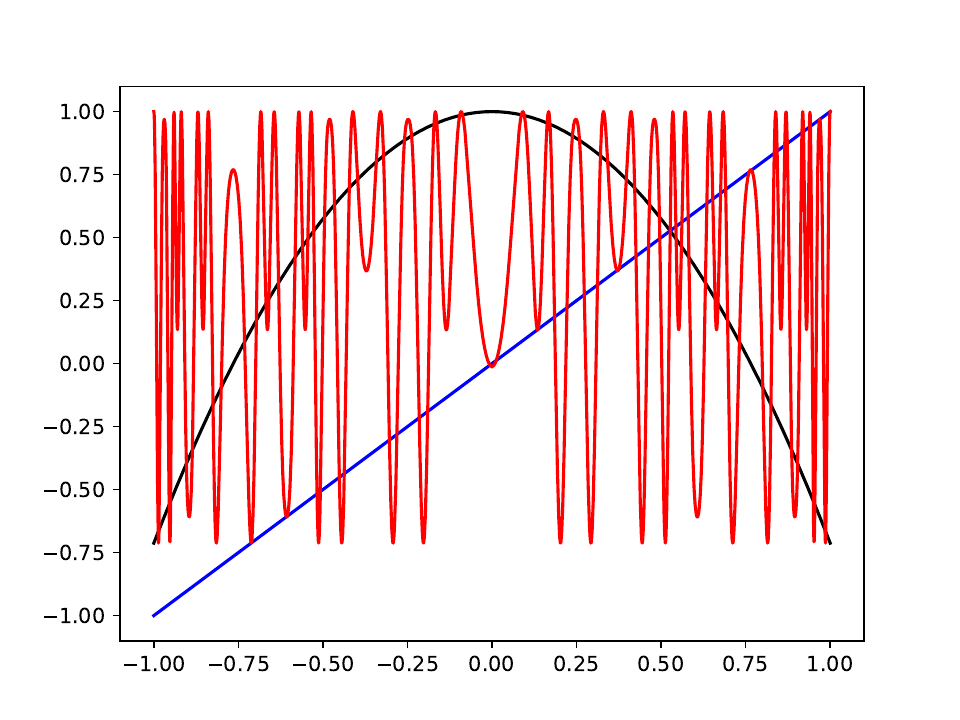}\hfill	
	\\
	\includegraphics[scale=0.5]{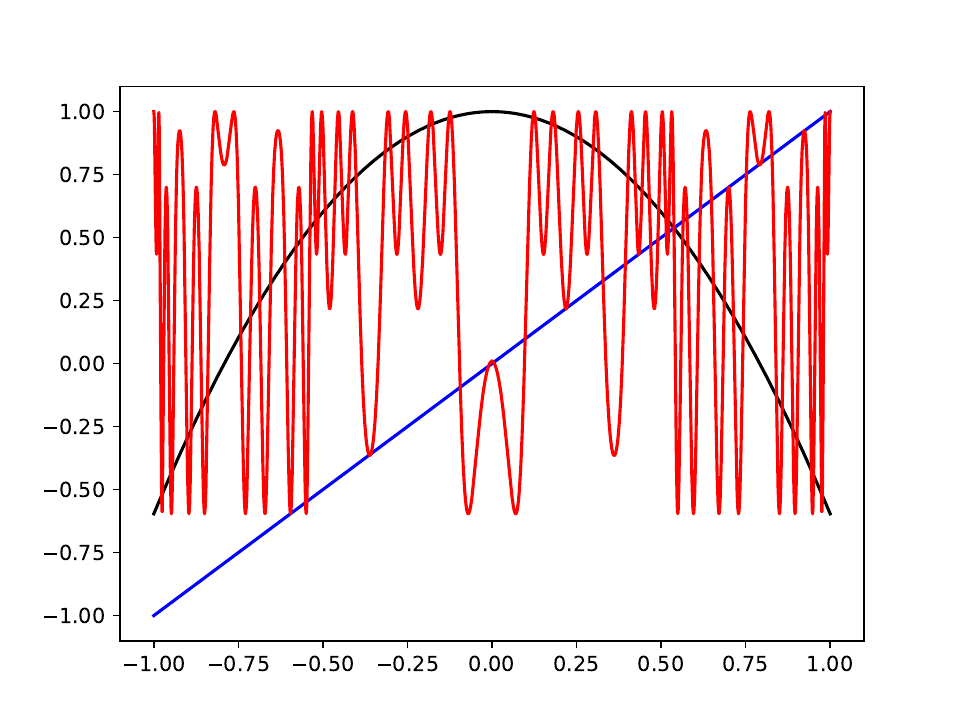}\hfill	
	\includegraphics[scale=0.5]{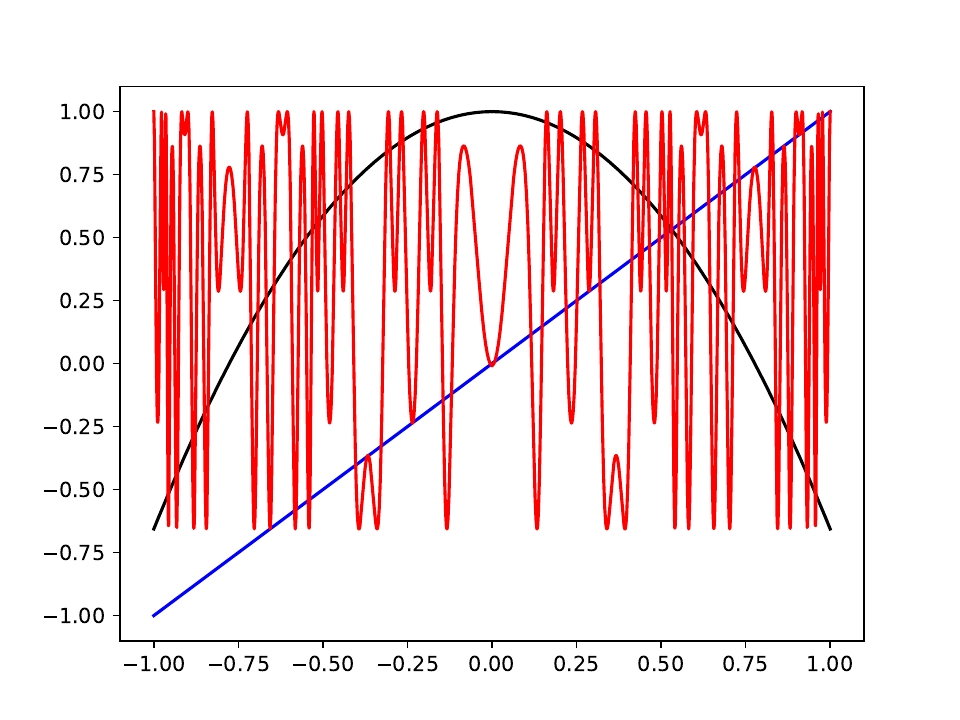}\hfill	
	\\
	\includegraphics[scale=0.5]{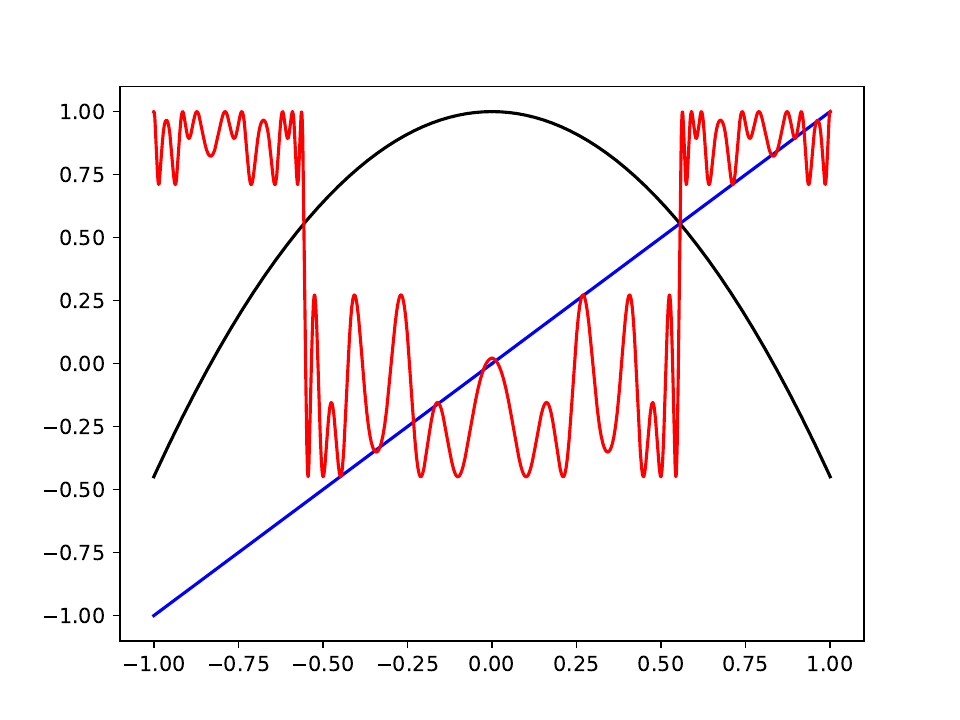}\hfill	
	\includegraphics[scale=0.5]{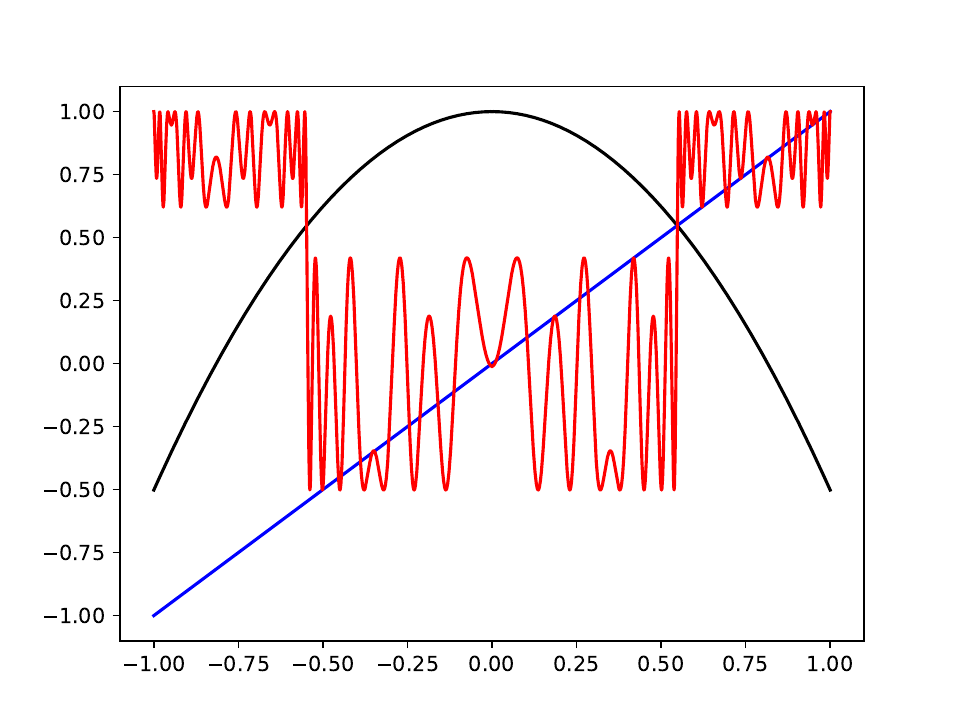}\hfill		
	\caption{Distinct fixed points with $m=8$, $m=9$ and $m=10$. On the left are the fixed points from Theorem~\ref{thm:main} for these values of $m$, already depicted on Figure~\ref{fig:theFixedPoints} and reproduced here for comparison. On the right are the $m$v2 fixed points from Theorem~\ref{thm:extra}. The black curves show the graph of the 
		polynomials $\bar{f}_m$ on $[-1,1]$, 
		the red curve is the composition of $\bar{f}_m$ with itself 
		$m$ times without rescaling, and the blue line is the fixed point line $y=x$.}
\label{fig:8to10}		
\end{figure}

\begin{theorem}
\label{thm:extra}
For each row in Table~\ref{table:thm_extrafixedpoints}, consider the corresponding polynomial $\bar{f}$ 
 whose precise coefficients in the Chebyshev basis can be found in the folder \texttt{renor\_code\_submit} (many of those $\bar{f}$ are represented in Figure~\ref{fig:5to7} and Figure~\ref{fig:8to10}). There exists $\rho>1$ and an analytic function $f\in \ell^{1,\even}_\rho$ such that $f$ is a fixed point of $R_m$, and $\Vert f - \bar{f} \Vert_{\ell^1_\rho} \leq \epsilon$. All of these fixed points are different from one another, and from the ones obtained in Theorem~\ref{thm:main} (even when the value of $m$ is the same). The corresponding value of $f^m(0)$ belongs to the interval $\alpha$, and the unstable eigenvalue of $DR_m(f)$ belongs to the interval $\lambda$. 
\end{theorem}
\begin{table}[h!]
\begin{tiny}
	\begin{center}
		\begin{tabular}{|c| c| c| c | c |} 
			\hline
			$id$ &  $\epsilon$  & $\alpha$ & $\lambda$  \\ 
			\hline\hline
5v2 & $10^{-18}$ & $0.021831959945968847\pm 10^{-18}$ &
						$ 1287.079118670726855828\pm 10^{-18}$ \\ \hline
5v3 & $10^{-18}$ & $-0.006248774967523457\pm 10^{-18}$ &
						$ 16930.645600440324906026\pm 10^{-18}$ \\ \hline
6v2 & $10^{-18}$ & $0.0477814797951692540\pm 10^{-19}$ &
						$ 218.4117951404949630935\pm 10^{-19}$ \\ \hline
6v3 & $10^{-19}$ & $-0.008694742127194554107054\pm 10^{-24}$ &
						$ 8507.7807835296246946195560996\pm 10^{-25}$ \\ \hline
7v2 & $10^{-18}$ & $-0.020339275729028984\pm 10^{-18}$ &
						$ 1446.441208992749725076\pm 10^{-18}$ \\ \hline
7v3 & $10^{-17}$ & $-0.007611915064702328\pm 10^{-18}$ &
						$ 10169.622465553708925541\pm 10^{-18}$ \\ \hline
7v4 & $10^{-32}$ & $0.0052208251178575989233699652732933677619657\pm 10^{-43}$ &
						$ 22840.37231548534717803508900336129634104003\pm 10^{-38}$ \\ \hline
7v5 & $10^{-32}$ & $-0.0043462191586195449514786183058666896110701\pm 10^{-43}$ &
						$ 35305.72646407373806388964752779792264729955\pm 10^{-38}$ \\ \hline
8v2 & $10^{-32}$ & $-0.011348310995370730350939747117190237592980\pm 10^{-42}$ &
						$ 5829.619563695360963996553495979704788192\pm 10^{-36}$ \\ \hline
8v3 & $10^{-32}$ & $0.015062576067338632962028240246726094137566\pm 10^{-42}$ &
						$ 2304.557844448592270375995417525541102801\pm 10^{-36}$ \\ \hline
9v2 & $10^{-32}$ &	 $-0.00696696826234461922363933818695872948849\pm 10^{-41}$ &
						$ 12818.372832195703881107882269687740106639\pm 10^{-36}$ \\ \hline
10v2 & $10^{-17}$ & $-0.010323910150325685\pm 10^{-18}$ &
						$ 4522.772195811370162545\pm 10^{-18}$ \\ \hline
10v3 & $10^{-32}$ & $0.02084892360493885744659877200151557651941478\pm 10^{-44}$ &
						$ 1110.53787417653278160218070069767775988323\pm 10^{-38}$ \\ \hline
		\end{tabular}
	\end{center}
	\caption{\textbf{Data associated with Theorem \ref{thm:extra}:}
	the table records the rigorously verified enclosures for the universal constants associated with $m$-th order renormalization fixed points available in the folder \texttt{renor\_code\_submit} and proven to exists in Theorem~\ref{thm:extra}. The first number in $id$ refers to the value of $m$, the second one is used in the code to distinguish between the different fixed points having the same value of $m$.}
	\label{table:thm_extrafixedpoints}
\end{tiny}
\end{table}

\begin{remark}
Denoting the renormalization fixed points from Theorem~\ref{thm:main} as $m$v1, we have that the fixed points 6v1 and 6v2 have the same universal constants, and so do 8v1 and 8v3, as well as 10v1 and 10v3.
	This phenomena has been observed by other authors and is discussed in more detail in 
	\cite{de2011combination} and the references therein.  
\end{remark}


\noindent\textit{Proof of Theorem~\ref{thm:extra}.}
The proof is exactly the same as the one of Theorem~\ref{thm:main}, except we start with different approximate solutions $(\ba,\bh)$ (and then different approximate eigenpairs $(\bl,\bu)$). Once the error bounds between the exact fixed points and the approximate ones are obtained, is it trivial to check that these fixed points are indeed all different from one another. \hfill \qed


Finally we compute rigorous bounds for the Schwarzian derivatives (see~\eqref{eq:Schwartz}) of the computed fixed points over $[-1,1]\setminus\{0\}$. These values are presented in Table \ref{table:kneading_seq}. 	

\begin{lemma}\label{lemma:Schwarzian}
	The fixed points that were proven to exist in Theorem \ref{thm:main} and Theorem \ref{thm:extra} have negative Schwarzian derivatives $(Sf)(x)$ for all $x$ in $[-1,1]\setminus\{0\}$. More precisely, defining $(S_1f)(x):=2(f'(x))^2 (Sf)(x)$, for each fixed point $S_1f$ is bounded above on $[-1,1]$ by the value given in the last column of Table \ref{table:kneading_seq}. 
	
\end{lemma}

\begin{proof}
	Notice that      
	\begin{equation*}
		(S_1f)(x) = 2f'''(x)f'(x) - 3(f''(x))^2, 
	\end{equation*}
which is now defined on $[-1,1]$, $0$ included. Since $\bf$ is a polynomial which we know explicitly, we can easily rigorously enclose $\max_{x\in [-1,1]} (S_1\bf)(x)$ and find an upper bound 
\begin{align*}
\max_{x\in [-1,1]} (S_1\bf)(x) \leq Y^K_{S_1f}.
\end{align*}
Provided we can also estimate $\left\Vert (S_1f) - (S_1\bar{f}) \right\Vert_{\CC^0_1}$, we then readily get and upper bound for $\max_{x\in [-1,1]} (S_1f)(x)$.

Indeed, we denote $h= f - \bar{f}$ and look for the following bounds 
	\begin{align*}
		&\left\Vert h \right\Vert_{\ell^1_\rho}  \leq r_{\text{min}} \\
		&\left\Vert h^{(j)} \right\Vert_{\CC^0_1}  \leq r^{(j)}_{\text{min}}\\
		&\left\Vert \bar{f}^{(j)} \right\Vert_{\CC^0_1} \leq {M^K_j},
	\end{align*}
	where $h^{(j)}$ denotes the $jth$ derivative of $h$, and $\rho$ refers to the weight of the $\ell^1_\rho$ space in which the fixed point was validated. For each fixed point, an explicit value for $r_{\text{min}}$ is provided by Theorem~\ref{thm:main} or Theorem~\ref{thm:extra}. Proposition~\ref{prop:derivative} then allows us to take $r^{(1)}_{\text{min}} = \sigma_{1,\rho}\, {r}_{\text{min}}$. Combining Proposition~\ref{prop:derivative} with Lemma~\ref{lem:C0_VS_ell1}, we then obtain
	\begin{align*}
	\left\Vert h'' \right\Vert_{\CC^0_1} &\leq \sigma_{1,1+\delta} \left\Vert h' \right\Vert_{\ell^1_{1+\delta}} \\
	&\leq  \frac{2+3\delta}{\delta} \, \sigma_{1,1+\delta}\left\Vert h' \right\Vert_{\CC^0_{1 + 2\delta}} \\
	&\leq \frac{2+3\delta}{\delta}\,\sigma_{1,1+\delta}\,\sigma_{1+2\delta,\rho}\, {r}_{\text{min}},
	\end{align*}
	with $\delta = \frac{\rho -1}{3}$, which gives a computable expression for $r^{(2)}_{\text{min}}$, and a similar calculation yields
	\begin{align*}
	r^{(3)}_{\text{min}} = \frac{(1+2\delta)^2 + (1+\delta)^2}{(2+3\delta)\delta}\, \frac{(2+7\delta)}{\delta} \, \sigma_{1,1+\delta}\, \sigma_{1+2\delta,1+3\delta}\,\sigma_{1 + 4\delta,\rho} \,{r}_{\text{min}},
	\end{align*}
where this time $\delta=\frac{\rho -1}{5}$.
Since $\bar{f}$ is a known polynomial, computing a bound $M^K_j$ is straightforward. Putting everything together, we can then get an upper bound for:
	\begin{align*}
		\left\Vert (S_1f) - (S_1\bar{f}) \right\Vert_{\CC^0_1} &= 
		\left\Vert 2f'''(x)f'(x) - 2\bar{f}'''(x)\bar{f}'(x) \right\Vert_{\CC^0_1} + \left\Vert 3(f''(x))^2 - 3(\bar{f}''(x))^2 \right\Vert_{\CC^0_1} \\ &\leq Y^1_{S_1f} + Y^2_{S_1f} . 
	\end{align*}
Computable expressions for the two terms can be obtained by taking
	\[
	Y^1_{S_1f} = 2\left((M^K_3 r^{(1)}_{\text{min}} + M^K_1 r^{(3)}_{\text{min}} + r^{(3)}_{\text{min}} r^{(1)}_{\text{min}}\right),
	\]
	and 
	\[
	Y^2_{S_1f} = 3r^{(2)}_{\text{min}} \left(2M^K_2 + r^{(2)}_{\text{min}} \right). 
	\]
We then have 
\begin{align*}
\max_{x\in [-1,1]} (S_1f)(x) \leq M_{S_1f} := Y^K_{S_1f} + Y^1_{S_1f} + Y^2_{S_1f},
\end{align*}	
whose values are shown in Table \ref{table:kneading_seq}.
\end{proof}
\begin{remark}
In the proof of Lemma~\ref{lemma:Schwarzian}, we made the choice of going via Lemma~\ref{lem:C0_VS_ell1} and Proposition~\ref{prop:derivative} in order to obtain $\CC^0$ estimates on higher order derivatives. The resulting estimates are admittedly not very sharp, but they are still good enough to allows us to conclude that the Schwarzian derivative is negative. If a more precise control on higher order derivatives was needed, one could get sharper results by directly deriving new estimates, instead of going back an forth between $\CC^0$ and $\ell^1$, as is done for instance in~\cite[Lemma 5.11]{BleBluBreEng24}.
\end{remark}

\noindent\textit{Proof of Theorem~\ref{thm:d2}.}
The proof first repeats verbatim the arguments and calculations of the proof of Theorem~\ref{thm:main} (with different approximate fixed points, and different choices of $K$ and of precision, which can be found in Appendix~\ref{sec:bounds_results} and in the file \texttt{final\_script\_gen\_m\_d.jl}). This provides us with the existence of a fixed point $f$ close to the approximate one $\bf$. The last remaining step of the proof is to establish that $f$ is exactly of degree $d=4$, i.e. that the Taylor expansion of $f$ at zero writes
\begin{equation*} 
f(z) = 1 + a_4 z^4 + o(z^4).
\end{equation*}
Since our fixed point argument was applied in $\R\times \ell^{1,\even}_\rho$, we only need to prove that $f''(0) = 0$. This will in fact be established a posteriori, without need for any further calculation.

Indeed, note that $\tilde\ell^{1,\even}_\rho := \left\{ f \in \ell^{1,\even}_\rho, f''(0) = 0\right\}$ is a closed subspace of $\ell^{1,\even}_\rho$ which is stable by composition, hence $\tilde\X_\rho := \R \times \tilde\ell^{1,\even}_\rho$ is a closed subspace of $\X_\rho$ which is stable by our zero-finding map $\Phi_m$. The subspace $\tilde\X_\rho$ is therefore also stable by the fixed point operator
\begin{align*}
\tilde{T} : (\alpha,f)\mapsto (\alpha,f) - \left(D\Phi_m(\alpha,f)\right)^{-1} \Phi_m(\alpha,f).
\end{align*}
As shown in the proof of~\cite[Theorem 2.15]{BerBreLesVee21}, if the assumptions of Theorem~\ref{th:fixed_point} are satisfied, then not only is the operator $T$ defined in~\eqref{eq:defT} a contraction on the ball $\B_{\X_\rho}\left((\ba,\bf),r\right)$, but so is $\tilde{T}$. Hence, provided the approximate fixed point $(\ba,\bf)$ belongs to $\tilde{\X_\rho}$, $\tilde{T}$ is also a contraction on the ball $\B_{\tilde\X_\rho}\left((\ba,\bh),r\right)$ in $\tilde{\X_\rho}$, and therefore there exists a unique fixed point of $\tilde{T}$ in this ball. By local uniqueness, this must be the same fixed point as the one already proven to exist in $\B_{\X_\rho}\left((\ba,\bf),r\right)$, therefore the fixed point $f$ is indeed of degree $d=4$. 

This argument requires us to ensure that the approximate fixed point $\bf$ lies exactly in $\tilde\ell^{1,\even}_\rho$, but this is easy to ensure using interval arithmetic (we again refer to \texttt{final\_script\_gen\_m\_d.jl} for details).
\hfill \qed
					
\section*{Acknowledgment}	

The Authors would like to sincerely thank Rafael de la Llave for many insightful discussions
as this project developed.  JG was partialy supported by NSF grant DMS - 2001758,
JDMJ was partially supported by NSF grant DMS - 2307987, and  
MB was partially supported by the ANR project CAPPS: ANR-23-CE40-0004-01
during this work. Part of this work was conducted during the thematic semester titled “Computational
Dynamics - Analysis, Topology \& Data”, supported by the Centre de Recherches Math\'ematiques
at the University of Montreal and the Simons Foundation. MB and JDMJ sincerely thank these institutions
for their funding and for providing the opportunity to explore this research.

		
\section*{Appendix}				
\appendix

\section{Computing $\sigma^\even_{\rho,\nu}$}
				
				We provide here a computable upper bound for the constant $\sigma^\even_{\rho,\nu}$ introduced in Proposition~\ref{prop:derivative}, together with a criterion ensuring that this upper-bound is actually sharp. A similar procedure can be derived for $\sigma^\odd_{\rho,\nu}$.
				\begin{lemma}
					\label{lem:sigma_expl}
					Let $1\leq \rho<\nu$, and $n_0\in\N_{\geq 1}$ such that $n_0 \geq \frac{1}{2(\ln\nu-\ln\rho)}$. Then
					\begin{align*}
						\sigma^\even_{\rho,\nu} \leq \max\left( \max_{1\leq n\leq n_0-1} \frac{2n}{\nu^{2n}}\left(\rho\frac{\rho^{2n}-1}{\rho^2-1} +\rho^{-1}\frac{\rho^{-2n}-1}{\rho^{-2}-1} \right),  \frac{2n_0}{\nu^{2n_0}}\left(\rho^{2n_0}\frac{\rho}{\rho^2-1} +\frac{\rho^{-1}}{1-\rho^{-2}} \right) \right).
					\end{align*}
					Moreover, if $n_0$ is large enough so that
					\begin{align*}
						\max_{1\leq n\leq n_0-1} \frac{2n}{\nu^{2n}}\left(\rho\frac{\rho^{2n}-1}{\rho^2-1} +\rho^{-1}\frac{\rho^{-2n}-1}{\rho^{-2}-1} \right) \geq \frac{2n_0}{\nu^{2n_0}}\left(\rho^{2n_0}\frac{\rho}{\rho^2-1} +\frac{\rho^{-1}}{1-\rho^{-2}} \right),
					\end{align*}
					then 
					\begin{align*}
						\sigma^\even_{\rho,\nu} =  \max_{1\leq n\leq n_0-1} \frac{2n}{\nu^{2n}}\left(\rho\frac{\rho^{2n}-1}{\rho^2-1} +\rho^{-1}\frac{\rho^{-2n}-1}{\rho^{-2}-1} \right).
					\end{align*}
				\end{lemma}
				\begin{proof}
					For all $n\in\N_{\geq 1}$,
					\begin{align*}
						\frac{2n}{\nu^{2n}}\left(\rho\frac{\rho^{2n}-1}{\rho^2-1} +\rho^{-1}\frac{\rho^{-2n}-1}{\rho^{-2}-1} \right) \leq \frac{2n}{\nu^{2n}}\left(\frac{\rho}{\rho^2-1}\rho^{2n} +\frac{\rho^{-1}}{1-\rho^{-2}} \right),
					\end{align*}
					therefore
					\begin{align*}
						\sigma^\even_{\rho,\nu} \leq \max\left( \max_{1\leq n\leq n_0-1} \frac{2n}{\nu^{2n}}\left(\rho\frac{\rho^{2n}-1}{\rho^2-1} +\rho^{-1}\frac{\rho^{-2n}-1}{\rho^{-2}-1} \right), \sup_{n \geq n_0} \frac{2n}{\nu^{2n}}\left(\rho^{2n}\frac{\rho}{\rho^2-1} +\frac{\rho^{-1}}{1-\rho^{-2}} \right) \right),
					\end{align*}
					and the announced estimates simply follows from the fact that the map
					\begin{align*}
						x\mapsto \frac{x}{\nu^{x}}\left(\frac{\rho}{\rho^2-1}\rho^{x} +\frac{\rho^{-1}}{1-\rho^{-2}} \right)
					\end{align*}
					is decreasing for $x\geq \frac{1}{\ln\nu-\ln\rho}$.
				\end{proof}

\section{Going between Chebyshev coefficients and values at Chebyshev nodes}
				\label{sec:DCT}

				In practice, while we represent an even polynomial $h\in\Pi^{2K}\ell^1_\rho$ by a vector containing its coefficients (of even index) in the Chebyshev basis, i.e. $h = h_0 + 2 \sum_{k=1}^K h_{2k}$, it is sometimes convenient to work instead with the values at the Chebyshev points, for instance to efficiently compute compositions. We point out that the isomorphism between these two representations can be computed explicitly by
				\begin{equation*}
					\begin{pmatrix}
						h(x_0) \\ h(x_1) \\ \vdots \\ h(x_K)
					\end{pmatrix}
					=M_K \begin{pmatrix}
						h_0 \\ h_2 \\ \vdots \\ h_{2K}
					\end{pmatrix},
				\end{equation*}
				where
				
					\begin{equation*}
						M_K=\begin{pmatrix}
							\cos(0\theta_0) & 2\cos(2\theta_0) & \ldots & 2\cos(2(K-1)\theta_0) & 2\cos(2K\theta_0) \\
							\cos(0\theta_1) & 2\cos(2\theta_1) & \ldots & 2\cos(2(K-1)\theta_1) & 2\cos(2K\theta_1) \\
							\vdots & \vdots & \ddots & \vdots & \vdots \\
							\cos(0\theta_{K-1}) & 2\cos(2\theta_{K-1}) & \ldots & 2\cos(2(K-1)\theta_{K-1}) & 2\cos(2K\theta_{K-1}) \\
							\cos(0\theta_K) & 2\cos(21\theta_K) & \ldots & 2\cos(2(K-1)\theta_K) & 2\cos(2K\theta_K) \\
						\end{pmatrix},
					\end{equation*}
					and $x_k = \cos\theta_k$ are half the Chebyshev nodes,  i.e.,
					\begin{align*}
						\theta_k = \frac{K-k}{2K}\pi, \quad 0\leq k\leq K.
					\end{align*}
					It should be noted that this is nothing but (one version of) the Discrete Cosine Transform. Therefore, the inverse transformation can also be described explicitly and has a very similar expression

					\begin{equation*}
						\renewcommand{\arraystretch}{1.5}
						M_K^{-1}=\frac{1}{2K}
						\begin{pmatrix}
							\frac{1}{2} \cos(0\theta_0) & \cos(0\theta_1) & \ldots & \cos(0\theta_{K-1}) & \frac{1}{2} \cos(0\theta_K) \\
							\frac{1}{2} \cos(2\theta_0) & \cos(2\theta_1) & \ldots & \cos(2\theta_{K-1}) & \frac{1}{2} \cos(2\theta_K) \\
							\vdots & \vdots & \ddots & \vdots & \vdots \\
							\frac{1}{2} \cos(2(K-1)\theta_0) & \cos(2(K-1)\theta_1) & \ldots & \cos(2(K-1)\theta_{K-1}) & \frac{1}{2} \cos(2(K-1)\theta_K) \\
							\frac{1}{4} \cos(2K\theta_0) & \frac{1}{2}\cos(2K\theta_1) & \ldots & \frac{1}{2}\cos(2K\theta_{K-1}) & \frac{1}{4 }\cos(2K\theta_K) \\
						\end{pmatrix}.
					\end{equation*}
					
					These transformations can also be computed efficiently via Fast Fourier Transform algorithms.

\section{Computing the supremum on a Bernstein Ellipse}
					\label{sec:intervalFFT}
					Consider a polynomial $h$, written in the Chebyshev basis:
					\begin{align*}
						h(x) = h_0 + 2\sum_{k=1}^K h_k T_k(x).
					\end{align*} 
We describe here a strategy originating from~\cite{BerBreLesMir24} (see also~\cite{HarSan24}) allowing to efficiently compute, at least to sharply upper-bound, the $\CC^0_\rho$ norm of $h$.			For any $\rho>1$, we have
					\begin{align*}
						\left\Vert h\right\Vert_{\CC^0_\rho} & =\max\limits_{z\in\E_\rho} \vert h(z)\vert \\
						& =\max\limits_{z\in\partial\E_\rho} \vert h(z)\vert \\
						& =\max\limits_{\theta\in[0,2\pi]} \left\vert h\left(\frac{1}{2}\left(\rho e^{i\theta}+(\rho e^{i\theta})^{-1} \right)\right)\right\vert \\
						& =\max\limits_{\theta\in[0,2\pi]} \left\vert \sum_{k=-K}^K h_k\rho^k e^{ik\theta}\right\vert \\
						& =\max\limits_{\theta\in[0,2\pi]} \left\vert f(\theta)\right\vert,
					\end{align*}
					where
					\begin{align*}
						f(\theta) = \sum_{k=-K}^K f_k e^{ik\theta}, \quad f_k = h_k \rho^k.
					\end{align*}
					Therefore, we simply have to compute (or at least to upper-bound) the supremum of a trigonometric polynomial on $[0,2\pi]$.
					
					Now, given an integer $N\geq 2K$, and considering the uniform grid
					\begin{align*}
						\theta_n = n\frac{2\pi}{N},\quad n=0,\ldots,N,
					\end{align*} 
					we can efficiently evaluate $f(\theta_n)$ for all $n\in\{0,\ldots,N\}$ using the FFT (and padding $f$ by zeros if necessary). In order to get a rigorous enclosure of the image of $f$ on the whole interval $[0,2\pi]$, rather than just of the image of the grid, we notice that, for any $\theta\in[0,2\pi]$, there exists $n\in\{0,\ldots,N\}$ such that $\theta\in \theta_n + [-\delta,\delta]$, with $\delta = \frac{\pi}{N}$. Therefore, with interval arithmetic, 
					\begin{align*}
						f(\theta) \in \sum_{k=-K}^K f_k e^{ik[-\delta,\delta]} e^{ik\theta_n}.
					\end{align*}
					That is, if we consider $\tilde{f}$ the trigonometric polynomial (having interval coefficients) given by
					\begin{align*}
						\tilde f(\theta) = \sum_{k=-K}^K \tilde f_k e^{ik\theta}, \quad \tilde f_k = f_k e^{ik[-\delta,\delta]},
					\end{align*}
					and evaluate $\tilde f$ on the grid $\theta_0,\ldots,\theta_N$, (which can be done efficiently using the FFT), we get $N+1$ intervals whose reunion contains $f([0,2\pi]) = h(\partial\E_\rho)$. We can then easily get an upper-bound for $\vert h(\partial\E_\rho) \vert$, but also for any quantity of the form  $\vert G(h(\partial\E_\rho)) \vert$ if $G$ is an analytic function which we can rigorously evaluate with intervals.
					
		
\section{Implementation and results details on the computer-assisted proofs}
					\label{sec:bounds_results}
					
All the calculations were done using the 
IntervalArithmetic.jl library~\cite{IntervalArithmeticJ} version 0.20.8, in Julia 1.8.5. The proof of Theorem~\ref{thm:main} was done using $\rho=2$ and an extend precision of $128$, i.e., $128$ digits for the mantissa of floating-point numbers, or of $256$. For the proof of Theorem~\ref{thm:extra}, we use a precision of $256$ and sometimes also slightly vary $\rho$ (see Table~\ref{table:rstarandbounds_extra}).  For the proof of Theorem~\ref{thm:d2}, a precision of $256$ was also used, but much smaller values of $\rho$ were required (see Table~\ref{table:rstarandbounds_d2}), whereas the proof of Theorem~\ref{thm:m2} used $\rho=2$ and a precision of $4096$.

For each of the fixed points obtained in Theorem~\ref{thm:main}, we specify in Table~\ref{table:rstarandbounds} and Table~\ref{table:boundsEV} how many Chebyshev modes where used, what value of $r^*$ was selected, and provide the evaluation of the bounds described in Sections~\ref{sec:bounds} and~\ref{sec:eigenvalue}. The same thing is done in Table~\ref{table:rstarandbounds_m2} and Table~\ref{table:boundsEV_m2} regarding Theorem~\ref{thm:m2}, in Table~\ref{table:rstarandbounds_extra} and Table~\ref{table:boundsEV_extra} for Theorem~\ref{thm:extra} and in Table~\ref{table:rstarandbounds_d2} and Table~\ref{table:boundsEV_d2} for Theorem~\ref{thm:d2}

\begin{table}[h!]
	\begin{center}
	\scalebox{0.85}{
		\begin{tabular}{| c | c | c | c | c | c | c | c | c |c|c|} 
			\hline
			$id$ & $K$ &  $r_*$  & $Y^K$ & $Y^\infty$ & $Z^{K,K}$ & $Z^{K,\infty}$ & $Z^{\infty}$ & $r_{\text{min}}$ & $\rho$ & pre	 \\ 
			\hline\hline 
			2v1 & 21 & $10^{-15}$ &9.0362e-25 &  5.2028e-18 & 0.3645 & 1.0928e-06 & 0.0036& 8.2322e-18 & 2.0 & $2^7$\\ 
			\hline							
			3v1 & 15  &   $10^{-16}$ & 3.9273e-66 & 2.1901e-19 & 2.2162e-04 & 1.0741e-06 & 1.0988e-07 & 2.1906e-19  & 2.0 & $2^8$\\
			\hline
			4v1 & 15 & $10^{-16}$ & 7.3335e-65 &  1.4920e-19 &  0.0125 & 1.5686e-07 & 4.0452e-07 & 1.5108e-19 & 2.0 & $2^8$ \\
			\hline
			5v1 & 15 &  $10^{-16}$ & 2.2511e-65 & 1.8554e-19 & 0.0063 & 3.0698e-07  & 2.0087e-07 & 1.8671e-19 & 2.0 & $2^8$ \\
			\hline							
			6v1 & 15 &  $10^{-16}$ & 2.8444e-65 & 1.9278e-19 &0.0119 & 3.5151e-07 & 1.8776e-07 &1.9510e-19 & 2.0 & $2^8$\\
			\hline
			7v1 & 15 &  $10^{-16}$ & 1.9636e-64 & 	1.0690e-19 & 0.1864 & 7.6479e-08 & 5.7263e-07 & 1.3138e-19 & 2.0 & $2^8$\\
			\hline
8v1 & 15 &  $10^{-17}$ &2.6605e-64 & 9.6733e-20 & 0.0358 & 8.8323e-08 & 5.6743e-07 & 1.0032e-19 & 2.0 & $2^8$\\
			\hline
			9v1 & 15 &  $10^{-17}$ & 7.0522e-64 & 9.5075e-19  & 0.1794 & 8.3621e-08& 1.0638e-06 & 1.1586e-19  & 2.0 & $2^8$\\
			\hline				
			10v1 & 15 & $10^{-17}$ & 1.8038e-64 & 	1.4605e-19  &  	0.0408 & 1.9816e-07 & 4.3566e-07 & 1.5225e-19 & 2.0 & $2^8$\\
			\hline
		\end{tabular}
		}
	\end{center}
	\caption{\textbf{Data associated with the proof of Theorem~\ref{thm:main}, fixed point.} The finite dimensional projection used for the proof was $\Pi^{2K}$. We report in this table the values obtained for each part of the bounds $Y = Y^K + Y^\infty$ and $Z = Z^{K,K} + Z^{K,\infty} + Z^\infty$ described in Section~\ref{sec:bounds} for the validation of the fixed point (rounded to 4 digits for readability). The value of $r_{min}$ is the smallest error bound provided by Theorem~\ref{th:fixed_point} for the fixed point, i.e. $\frac{Y}{1-Z}$. }
	\label{table:rstarandbounds}
\end{table}

\begin{table}[h!]
	\begin{center}
		\begin{tabular}{| c | c | c | c | c | c | c | c | } 
			\hline
			$id$ & $Y^K$ & $Y^\infty$ & $Z_1^{K,K}$ & $Z_1^{K,\infty}$ & $Z_1^{\infty}$ & $Z_2$ & $r_{\text{min}}$ 	 \\ 
			\hline\hline 
			2v1 & 5.5828e-24 & 9.9824e-17 & 1.7886e-24 & 	2.2226e-06 & 7.5739e-04 & 6.3889 & 9.9900e-17 \\ 
			\hline							
			3v1 & 1.8385e-65	& 7.0423e-19& 6.5072e-66 &9.0997e-06 &3.8589e-14& 5.4928 & 7.0425e-19 \\
			\hline
			
			4v1 & 2.3428e-64 & 2.3490e-19 & 8.0858e-65 & 5.9755e-06 & 1.0491e-17 & 5.1405 & 
2.3490e-19\\
			\hline
			
			5v1 & 1.1164e-64 & 8.0609e-19&	4.0457e-65   & 5.9397e-06  & 1.1398e-17& 5.0968 & 8.0610e-19\\
			\hline							
			6v1 &1.5148e-64 & 1.0564e-18 &	5.5709e-65 & 6.5367e-06 & 1.5722e-17 & 5.2540 &1.0564e-18\\
			\hline
			
			7v1 & 7.0589e-64 & 1.8395e-19 & 2.4611e-64 & 4.3135e-06 & 1.1089e-17  & 5.0100 & 
1.8395e-19 \\
			\hline
			
8v1 & 1.0406e-63 &1.6828e-19 & 3.5045e-64 &	5.5632e-06 & 1.3523e-17 & 5.0586 & 1.6828e-19\\
			\hline
			
			9v1 & 2.6109e-63 & 1.3863e-19   &8.7070e-64 & 9.0446e-06  & 1.1629e-17 & 5.0177 & 1.3863e-19  \\
			\hline	
						
			10v1 &  1.0626e-63 & 6.6486e-19 & 4.0818e-64  & 9.0336e-06 & 1.7825e-17 & 5.1753 & 6.6487e-19       \\
			\hline
		\end{tabular}
	\end{center}
	\caption{\textbf{Data associated with the proof of Theorem~\ref{thm:main}, eigenvalue problem.} We report in this table the values obtained for each part of the bounds $Y = Y^K + Y^\infty$, $Z_1 = Z_1^{K,K} + Z_1^{K,\infty} + Z_1^\infty$ and $Z_2$ described in Section~\ref{sec:eigenvalue} for the validation of the eigenpair of the fixed point (rounded to 4 digits for readability). The value of $r_{min}$ is the smallest error bound provided by Theorem~\ref{th:fixed_point} in that case, i.e. $\frac{1-Z_1-\sqrt{(1-Z_1)^2-4YZ_2}}{2Z_2}$. }
	\label{table:boundsEV}
\end{table}	

%
%

\begin{table}[h!]
	\begin{center}
	\scalebox{0.9}{
		\begin{tabular}{| c | c | c | c | c | c | c | c | c |} 
			\hline
			$id$ & $K$ &  $r_*$  & $Y^K$ & $Y^\infty$ & $Z^{K,K}$ & $Z^{K,\infty}$ & $Z^{\infty}$ & $r_{\text{min}}$	 \\ 
			\hline\hline 
			2v1 & 680 & $10^{-480}$ & 4.001e-803 & 4.788e-503 & 2.400e-51 & 5.893e-103 & 1.016e-117 & 4.789e-503 \\ 
			\hline							
		\end{tabular}
		}
	\end{center}
	\caption{\textbf{Data associated with the proof of Theorem~\ref{thm:m2}, fixed point.} The finite dimensional projection used for the proof was $\Pi^{2K}$. We report in this table the values obtained for each part of the bounds $Y = Y^K + Y^\infty$ and $Z = Z^{K,K} + Z^{K,\infty} + Z^\infty$ described in Section~\ref{sec:bounds} for the validation of the fixed point (rounded to 4 digits for readability). The value of $r_{min}$ is the smallest error bound provided by Theorem~\ref{th:fixed_point} for the fixed point, i.e. $\frac{Y}{1-Z}$. }
	\label{table:rstarandbounds_m2}
\end{table}	

%

\begin{table}[h!]
	\begin{center}
		\begin{tabular}{| c | c | c | c | c | c | c | c |} 
			\hline
			$id$ & $Y^K$ & $Y^\infty$ & $Z_1^{K,K}$ & $Z_1^{K,\infty}$ & $Z_1^{\infty}$ & $Z_2$ & $r_{\text{min}}$	 \\ 
			\hline\hline 
			2v1 & 2.318e-802 & 2.007e-501 & 6.273e-803& 1.470e-102 & 2.176e-118 & 6.389 & 2.007e-501 \\ 
			\hline							
		\end{tabular}
	\end{center}
	\caption{\textbf{Data associated with the proof of Theorem~\ref{thm:m2}, eigenvalue problem.} We report in this table the values obtained for each part of the bounds $Y = Y^K + Y^\infty$, $Z_1 = Z_1^{K,K} + Z_1^{K,\infty} + Z_1^\infty$ and $Z_2$ described in Section~\ref{sec:eigenvalue} for the validation of the eigenpair of the fixed point (rounded to 4 digits for readability). The value of $r_{min}$ is the smallest error bound provided by Theorem~\ref{th:fixed_point} in that case, i.e. $\frac{1-Z_1-\sqrt{(1-Z_1)^2-4YZ_2}}{2Z_2}$. }
	\label{table:boundsEV_m2}
\end{table}	

\begin{table}[h!]
	\begin{center}
	\scalebox{0.85}{
		\begin{tabular}{| c | c | c | c | c | c | c | c | c |c|c|} 
			\hline
			$id$ & $K$ &  $r_*$  & $Y^K$ & $Y^\infty$ & $Z^{K,K}$ & $Z^{K,\infty}$ & $Z^{\infty}$ & $r_{\text{min}}$ & $\rho$ &	pre \\ 
			\hline\hline 
5v2 & 15  & e-16 &9.6038e-65 & 1.2344e-19& 0.0424& 1.0493e-07  & 4.7803e-07 & 1.2890e-19 & 2.0 & $2^8$	\\ \hline

5v3 & 	15  & e-17 & 1.1297e-63 & 1.2313e-19& 0.0762 & 1.3712e-07& 1.8146e-06& 1.3329e-19 & 2.0 & $2^8$	\\ \hline		

6v2 & 	20 & e-18 & 7.4045e-62 & 2.6503e-20 &0.3375 & 2.2102e-09& 
6.1533e-10 & 4.0004e-20 & 2.0 & $2^8$	\\ \hline	
			
6v3 & 	20 &e-19 & 6.1106e-61  & 3.2968e-26 & 0.6772 & 1.0213e-10& 1.2525e-09& 1.0213e-25 & 2.0 & $2^8$	\\ \hline	
			
7v2 & 	15 & e-16 & 1.3907e-64 & 1.0403e-19 & 0.1229 & 1.0228e-07 & 
4.6997e-07 & 1.1860e-19 & 2.0 & $2^8$	\\ \hline	
			
7v3 & 	15  & e-17 & 8.4308e-64 & 1.0696e-19 & 0.1318 & 8.5794e-08& 1.4121e-06 & 1.2319e-19	 & 2.0 & $2^8$\\ \hline		

7v4 & 	40 & e-32 & 3.2529e-50 & 7.9649e-44 & 0.0068 & 1.7355e-20 & 2.7443e-23 & 8.0187e-44 & 1.9 & $2^8$	\\ \hline	
			
7v5 & 	40 & e-32 & 4.7632e-50 & 2.6981e-44 & 0.0118 &	1.8649e-20 & 3.4074e-23 & 2.7302e-44 & 1.9 & $2^8$	\\ \hline	
			
8v2 & 	40 & e-32 & 5.4945e-49 & 3.8926e-34 & 0.0903 & 7.5798e-22 & 	8.2687e-22 & 4.2787e-43 & 2 & $2^8$	\\ \hline	
			
8v3 & 	40 & e-32 & 2.8942e-48 & 3.6926e-43 & 0.4798 & 1.2745e-20 & 2.1043e-21 & 7.0974e-43	& 2 & $2^8$\\ \hline	
			
9v2 & 	40 & e-32 & 1.0590e-48 & 1.3697e-42 & 0.3569 & 3.3272e-22 & 1.2413e-21 & 2.1298e-42 & 2 & $2^8$	\\ \hline		
		
10v2 & 15 & e-17 & 	5.9468e-64  & 9.5371e-20  & 0.1836 & 1.0190e-07 & 8.2397e-07 & 1.1682e-19  & 2 & $2^8$	\\ \hline	
					
10v3 & 40 & e-32 & 9.3033e-51   &5.4053e-45 & 0.0022 & 5.9712e-19 & 1.2769e-23 & 5.4172e-45 & 1.9 & $2^8$	\\ \hline		
		\end{tabular}
		}
	\end{center}
	\caption{\textbf{Data associated with the proof of Theorem~\ref{thm:extra}, fixed point.} Each line in this table corresponds to the same line in Table~\ref{table:thm_extrafixedpoints}. The finite dimensional projection used for the proof was $\Pi^{2K}$. We report in this table the values obtained for each part of the bounds $Y = Y^K + Y^\infty$ and $Z = Z^{K,K} + Z^{K,\infty} + Z^\infty$ described in Section~\ref{sec:bounds} for the validation of the fixed point (rounded to 4 digits for readability). The value of $r_{min}$ is the smallest error bound provided by Theorem~\ref{th:fixed_point} for the fixed point, i.e. $\frac{Y}{1-Z}$. }
	\label{table:rstarandbounds_extra}
\end{table}	

\begin{table}[h!]
	\begin{center}
		\begin{tabular}{| c | c | c | c | c | c | c | c |} 
			\hline
			$id$ & $Y^K$ & $Y^\infty$ & $Z_1^{K,K}$ & $Z_1^{K,\infty}$ & $Z_1^{\infty}$ & $Z_2$ & $r_{\text{min}}$	 \\ 
			\hline\hline 
5v2 & 3.5065e-64 & 1.9947e-19  & 1.2467e-64& 	4.6617e-06  & 1.0958e-17 & 5.0625 & 1.9947e-19 \\ \hline	

5v3 & 3.4301e-63 & 1.4066e-19 & 1.1552e-63 & 2.1629e-05 & 1.0062e-17& 5.0354& 1.4066e-19 \\ \hline

6v2 & 4.1114e-61 & 1.2911e-20 & 1.5742e-61 & 3.2050e-08 & 3.4725e-16 & 6.9025 & 1.2911e-20 \\ \hline

6v3 &2.0046e-60 & 3.7573e-26 & 6.8773e-61 & 1.1497e-08 & 1.0012e-23 & 5.0081 & 3.7573e-26\\ \hline

7v2 & 	5.8229e-64 & 2.0308e-19 & 2.0155e-64 & 4.8508e-06 & 1.1771e-17 & 5.0825 & 2.0308e-19	\\ \hline	
	
7v3 & 	2.7652e-63
 & 1.2522e-19 & 9.4234e-64 &1.0919e-05 &1.1864e-17 & 5.0045 & 1.2522e-19	\\ \hline

7v4 & 	1.0338e-49 & 4.9154e-39 & 3.5308e-50 & 3.2986e-18 & 8.4375e-48 & 4.6150 & 4.9154e-39	\\ \hline

7v5 & 1.4800e-49 & 7.0750e-39 & 5.0415e-50 & 4.2623e-18 & 8.0265e-48 & 4.6195 & 7.0750e-39	\\ \hline

8v2 &2.0165e-48 & 1.1082e-37 & 6.8708e-49 & 6.6485e-20   & 8.5136e-48 & 	5.0976 & 1.1082e-37  \\ \hline

8v3 & 	1.1591e-47 & 6.4142e-37 & 4.0856e-48& 4.4501e-19& 7.1949e-47 & 7.3998  & 6.4142e-37	\\ \hline	

9v2 & 3.6398e-48& 1.9934e-37 & 1.2118e-48  & 4.7326e-20& 9.3407e-48    & 	5.0019  &  1.9934e-37	\\ \hline

10v2 & 2.3912e-63
 & 	1.7105e-19  & 8.2058e-64 & 9.2160e-06 & 1.5440e-17
  & 5.0465 & 1.7105e-19   	\\ \hline

10v3 & 	5.1302e-50 & 2.4690e-35 & 1.9782e-50 & 2.2818e-17  & 3.1944e-31 & 	5.7564 & 2.4690e-39  \\ \hline
		\end{tabular}
	\end{center}
	\caption{\textbf{Data associated with the proof of Theorem~\ref{thm:extra}, eigenvalue problem.} Each line in this table corresponds to the same line in Table~\ref{table:thm_extrafixedpoints}. We report in this table the values obtained for each part of the bounds $Y = Y^K + Y^\infty$, $Z_1 = Z_1^{K,K} + Z_1^{K,\infty} + Z_1^\infty$ and $Z_2$ described in Section~\ref{sec:eigenvalue} for the validation of the eigenpair of the fixed point (rounded to 4 digits for readability). The value of $r_{min}$ is the smallest error bound provided by Theorem~\ref{th:fixed_point} in that case, i.e. $\frac{1-Z_1-\sqrt{(1-Z_1)^2-4YZ_2}}{2Z_2}$. }
	\label{table:boundsEV_extra}
\end{table}

\begin{table}[h!]
	\begin{center}
		\begin{tabular}{| c | c |c| } 
			\hline
			$id$ & Kneading seq. & $M_{S_1f}$\\ 
			\hline\hline 
			2v1 & $( \text{R} )^\infty$ & -25.9364\\ 
			\hline						
			3v1 & $( \text{R L} )^\infty$ & -35.6810\\
			\hline
			4v1 & $( \text{R L L} )^\infty$ & --43.9803\\
			\hline
			5v1 & $( \text{R L R R} )^\infty$ & -31.6453\\
			\hline
			5v2 & $( \text{R L L R} )^\infty$ & -41.2675\\
			\hline	
			5v3 & $( \text{R L L L} )^\infty$ & -46.9408\\
			\hline					
			6v1 &$( \text{R L R R R} )^\infty$ & -25.9669\\
			\hline
			6v2 &$( \text{R L L R L} )^\infty$ & -43.9558\\
			\hline
			6v3 &$( \text{R L L R R} )^\infty$ & -43.6270\\
			\hline
			7v1 & $( \text{R L R R L R} )^\infty$ & -33.6712\\
			\hline
			7v2 & $( \text{R L R R R R} )^\infty$ & -29.5153\\ 
			\hline
			7v3 & $( \text{R L L R L R} )^\infty$ & -40.2905\\ 
			\hline
			7v4 & $( \text{R L L R R R} )^\infty$ & -42.5877\\ 
			\hline
			7v5 & $( \text{R L L R R L} )^\infty$ & -44.5056\\ 
			\hline
			8v1 & $( \text{R L R R R R R} )^\infty$ & -27.5505\\
			\hline
			8v2 & $( \text{R L R R L R R} )^\infty$ & -34.7314\\ 
			\hline
			8v3 & $( \text{R L L L R L L} )^\infty$ & -56.4426\\ 
			\hline
			9v1 & $( \text{R L R R R R L R} )^\infty$ & -30.4991\\ 
			\hline	
			9v2 & $( \text{R L R R L R L R} )^\infty$ & -32.9212\\ 
			\hline				
			10v1 & $( \text{R L R R R L R L R} )^\infty$ & -24.6491\\    
			\hline
			10v2 & $( \text{R L R R R R R L R} )^\infty$ &-26.8891\\ 
			\hline
			10v3 & $( \text{R L R R L R L R R} )^\infty$ & -33.9222\\   
			\hline
		\end{tabular}
	\end{center}
	\caption{\textbf{Kneading sequences for the fixed points:} 
		The powers/products here are with respect to 
		the star product for kneading sequences. See for example
		\cite{coullet1978iterations,de2011combination} for precise definitions.
		The superscript indicates the infinite star product of the sequence in parenthesis. That is, 
		if $A$ is a kneading sequence, then 
		$A^\infty = \lim_{n \to \infty} A^{*n}$,
		where $A^{*n} = A * \ldots * A$ is the iterated star product of the sequence $A$ with itself $n$ times.  
		Note that the kneading sequences given in the table are not periodic, due to the definition 
		of the star operator.  For example $R*R = RLR$ while $R*R*R = R*(RLR) = RLRRRLR$
		etcetera.  That is, this is just a convenient notation for expressing an a periodic sequences.  
		The reported values are not validated, as they are simply computed by following the orbit 
		of $0$ for the numerically computed fixed point $f_m$ and extracting the appropriate combinatorial
		sequence $A$ from the orbit data.
	}	\label{table:kneading_seq}		
\end{table}


\begin{table}[h!]
	\begin{center}
	\scalebox{0.9}{
		\begin{tabular}{| c | c | c | c | c | c | c | c | c | c|} 
			\hline
			$id$ & $\rho$& $K$ &  $r_*$  & $Y^K$ & $Y^\infty$ & $Z^{K,K}$ & $Z^{K,\infty}$ & $Z^{\infty}$ & $r_{\text{min}}$	 \\ 
			\hline\hline 
			$(2,4)$ & 1.197 & 100 & $10^{-40}$ &2.4329e-56 & 2.9855e-42 &  0.1106& 0.0702 &  0.0413 & 3.8377e-42 \\ 
			\hline\hline 
			$(3,4)$ & 1.180 & 100 & $10^{-45}$ & 3.7501e-59&  4.9779e-47 & 2.2046e-05 & 3.4681e-04 & 0.1531& 5.8801e-47 \\ 
			\hline								
		\end{tabular}
		}
	\end{center}
	\caption{\textbf{Data associated with the proof of Theorem~\ref{thm:d2}, fixed point.} The two numbers in \emph{id} denote the order $m$ and the degree $d$ of the renormalization fixed points, as defined in Remark~\ref{rem:FP_properties}. The finite dimensional projection used for the proof was $\Pi^{2K}$. We report in this table the values obtained for each part of the bounds $Y = Y^K + Y^\infty$ and $Z = Z^{K,K} + Z^{K,\infty} + Z^\infty$ described in Section~\ref{sec:bounds} for the validation of the fixed point (rounded to 4 digits for readability). The value of $r_{min}$ is the smallest error bound provided by Theorem~\ref{th:fixed_point} for the fixed point, i.e. $\frac{Y}{1-Z}$. }
	\label{table:rstarandbounds_d2}
\end{table}

\begin{table}[h!]
	\begin{center}
		\begin{tabular}{| c | c | c | c | c | c | c | c |} 
			\hline
			$id$ & $Y^K$ & $Y^\infty$ & $Z_1^{K,K}$ & $Z_1^{K,\infty}$ & $Z_1^{\infty}$ & $Z_2$ & $r_{\text{min}}$	 \\ 
			\hline\hline 
			$(2,4)$ & 1.0460e-53 & 3.2151e-42 & 2.1832e-38& 0.1187& 0.0057 &  3.6570 & 3.6715e-42 \\ 
			\hline\hline 
			$(3,4)$ &2.9721e-58 & 2.4901e-47 & 2.9269e-38& 0.0011 & 0.0018 & 3.3107 & 2.4973e-47\\ 
			\hline									
		\end{tabular}
	\end{center}
	\caption{\textbf{Data associated with the proof of Theorem~\ref{thm:d2}, eigenvalue problem.} The two numbers in \emph{id} denote the order $m$ and the degree $d$ of the renormalization fixed points, as defined in Remark~\ref{rem:FP_properties}. We report in this table the values obtained for each part of the bounds $Y = Y^K + Y^\infty$, $Z_1 = Z_1^{K,K} + Z_1^{K,\infty} + Z_1^\infty$ and $Z_2$ described in Section~\ref{sec:eigenvalue} for the validation of the eigenpair of the fixed point (rounded to 4 digits for readability). The value of $r_{min}$ is the smallest error bound provided by Theorem~\ref{th:fixed_point} in that case, i.e. $\frac{1-Z_1-\sqrt{(1-Z_1)^2-4YZ_2}}{2Z_2}$. }
	\label{table:boundsEV_d2}
\end{table}

\clearpage
\bibliographystyle{abbrv}
\bibliography{references_renor}


\end{document}